\documentclass[12pt]{article}
\usepackage{amsmath}
\usepackage{graphicx,psfrag,epsf}
\usepackage{enumerate}
\usepackage{natbib}
\usepackage{url}
\usepackage{amssymb}
\usepackage{amsmath}
\usepackage{epsfig,amsthm}
\usepackage{graphicx}
\usepackage{mathrsfs}
\usepackage{rotating,booktabs}
\usepackage{dcolumn}
\usepackage{subfigure}
\usepackage{graphicx}
\usepackage{mathrsfs}
\usepackage{multirow}
\usepackage{booktabs}
\usepackage{color}
\usepackage{times}
\usepackage{indentfirst}
\newtheorem{theorem}{{ Theorem}}

\newtheorem{lemma}{{ Lemma}}
\newtheorem{proposition}{{ Proposition}}
\newtheorem{assumption}{{ Assumption}}

\newcommand{\ignore}[1]{}
{}

\definecolor{red}{rgb}{1,0,0}

\newcommand{\indep}{\;\, \rule[0em]{.03em}{.67em} \hspace{-.25em}
\rule[0em]{.65em}{.03em} \hspace{-.25em}
\rule[0em]{.03em}{.67em}\;\,}
\newcommand{\blind}{1}
\input{tcilatex}

\begin{document}

\def\spacingset#1{\renewcommand{\baselinestretch}%
{#1}\small\normalsize} \spacingset{1}

%%%%%%%%%%%%%%%%%%%%%%%%%%%%%%%%%%%%%%%%%%%%%%%%%%%%%%%%%%%%%%%%%%%%%%%%%%%%%%

\if1\blind
{
\title{\textbf{Fr\'{e}chet Sufficient Dimension Reduction for Random Objects}}
\author{ Chao Ying and Zhou Yu \vspace{0.08in}\\
School of Statistics, East China
Normal University }
\maketitle
}
\fi

\if0\blind
{
  \bigskip
  \bigskip
  \bigskip
  \begin{center}
    {\LARGE\bf Fr\'{e}chet Sufficient Dimension Reduction for Random Objects}
\end{center}
  \medskip
} \fi

\begin{abstract}
We in this paper consider Fr\'echet sufficient dimension reduction
with responses being complex random objects in a metric space and high dimension Euclidean predictors. We propose a novel approach called weighted inverse regression ensemble method for linear Fr\'echet sufficient dimension reduction. The method is further generalized as a new operator
defined on reproducing kernel Hilbert spaces for nonlinear Fr\'echet sufficient dimension reduction. We provide theoretical guarantees for the new method via asymptotic analysis. Intensive simulation studies verify the performance of our proposals. And we apply our methods to analyze the handwritten digits data to demonstrate its use in real applications.
\end{abstract}


\spacingset{1.45} % DON'T change the spacing!
\section{Introduction}
Sufficient Dimension Reduction (\cite{SIR1991, Cook:1998}), as a powerful tool to extract the core information hidden in the high-dimensional data, has become an important and rapidly developing research field. For regression with multiple responses $Y \in \mathbb{R}^q$ and multiple predictors $X \in \mathbb{R}^p$, classical linear sufficient dimension reduction seeks a $p\times d$ matrix $\beta$ such that
\begin{align}\label{linear SDR}
Y\indep X \mid \beta^T X,
\end{align}
where $\indep$ stands for independence. The smallest subspace (\cite{YINLICOOK2008}) spanned by $\beta$ with $\beta$ satisfying the above relation (\ref{linear SDR}) is called the central subspace, which is denoted as $\mathcal{S}_{Y|X}$.

Classical methods for identifying the central subspace with one dimensional response include sliced inverse regression (\cite{SIR1991}), sliced average variance estimation (\cite{Cook1991}), the central $k$th moment method (\cite{YINCOOK2002}), the inverse third moment approach (\cite{Yin2003}), contour regression (\cite{Li2005}), directional regression (\cite{DR2007}), the constructive approach (\cite{CONSTRUCTIVEMAVE}), the semiparametric estimation (\cite{SEMISDR2012,EFFICIENTSDR2013}), and many others. \cite{Li2003}, \cite{Zhu2010}, \cite{Li2008} and \cite{Zhu2010} made important extensions for sufficient dimension reduction with multivariate response.

\cite{PSVM2011}, \cite{NonlinearSDR} and \cite{Li:2018} further articulated the general formulation of nonlinear sufficient
dimension reduction as
\begin{align}\label{nonlinear SDR}
Y\indep X \mid  f(X),
\end{align}
where $f:\mathbb{R}^p \mapsto \mathbb{R}^d $ is an unknown vector-valued function of $X$. Nonlinear sufficient dimension reduction actually replaces the linear sufficient predictor $\beta^T X$ by a nonlinear predictor $f(X)$ . The smallest subspace spanned by the functions satisfying the relation (\ref{nonlinear SDR})  is called the central class and denoted as $\mathcal{G}_{Y|X}$. See \cite{NonlinearSDR} and \cite{Li:2018} for more details.

Due to the rapid development of data collection technologies, statisticians nowadays are more frequently encountering complex data that are non-Euclidean and specially do not lie in a vector space. Images (\cite{Peyre2009, Gonzalez2018}), shapes (\cite{Small1996, Simeoni2013}), graphs (\cite{Tsochantaridis2004, Ferretti2018}), tensors (\cite{Zhu2009, Li2017}), random densities (\cite{Petersen2016, Liu2019}) are examples of complex data types that appear naturally as responses in image completion, computer vision, biomedical analysis, signal processing and other application areas.
%{\textcolor{red}{Though the object is high-dimensional, there is actually some lower-dimensional representation since  the object which locally has the topology structure or lying on smooth non-linear surfaces of lower dimensionality. }}
%Such random objects not only behaves high dimensional, but also have certain topology structures.
In particular, in image completion for handwritten digits ({\cite{Tsagkrasoulis2018}}), the upper part of each image was taken as the predictors $X$, and the  bottom half was set as the responses $Y$. Figure 1 in the following illustrates the idea of such image analysis for digits $\{0,8,9\}$.  To predict the bottom half of handwritten digits from their upper half is not an easy task, as the upper parts of image digits $\{0,8,9\}$ are quite similar to each other. %Understanding and visualizing the effects of the predictors on the responses is of paramount importance.
In image analysis, it is common to assume that the images lie on an unknown manifold equipped with a meaningful distance metric. Then it is of great interest to develop general Fr\'echet sufficient dimension reduction method with metric space valued responses. Fr\'echet sufficient dimension reduction for such $X$ and $Y$ is then an immediate need that can facilitate graphical understanding of the regression structure, and is certainly helpful for further image clustering or classification and outlier diagnostics.

\begin{figure}[htbp]
    \centering
    \subfigure{
\includegraphics[height=1.3in,width=5in]{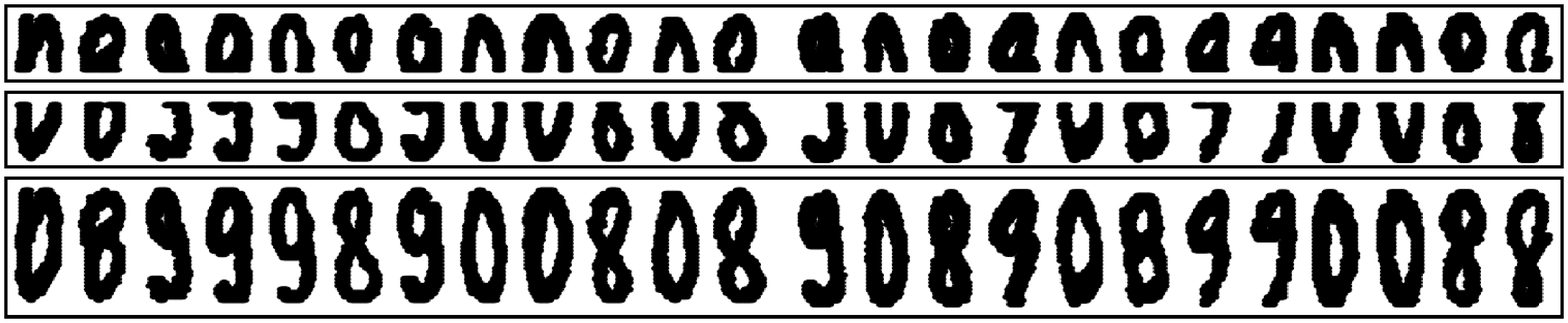}
    }
\caption{ The first row consists of the predictors $X$ which are the upper halves of the image digits $\{0,8,9\}$; The second row consists of the responses $Y$ which are the bottom halves of the image digits $\{0,8,9\}$; The third row consists of the whole image digits $\{0,8,9\}$.}
\end{figure}

\cite{Dubey2019} and \cite{Petersen2019b} provided some fundamental tools for Fr\'echet analysis of such random objects. \cite{Petersen2019a} further proposed a general global and local Fr\'echet regression paradigm for responses being complex random objects in a metric space with Euclidean predictors. Along their pioneering work in Fr\'echet analysis, it is then of great interest to consider linear and nonlinear sufficient dimension reduction for response objects in a metric space when the dimension of Euclidean predictors is relatively high.

As an illustration of Fr\'echet sufficient dimension reduction, we consider  two models:
\begin{align*}
(\text{i}). \quad Y=&(\sin(\beta_1^T X+\varepsilon_1)\sin(\beta_2^T X+\varepsilon_2),\sin(\beta_1^T X+\varepsilon_1)\cos(\beta_2^T X+\varepsilon_2),\cos(\beta_1^T X+\varepsilon_1)),\\
(\text{ii}). \quad Y=&(\sin(f_1(X)+\varepsilon_1)^{1/3},\cos(f_1(X)+\varepsilon_1)^{1/3}),
\end{align*}
where $(\varepsilon_1,\varepsilon_2)^T\sim N(0_2,I_2)$, $X=(x_1,\ldots,x_p)^T \sim N(0_p, I_p)$ with $p=30$, $f_1(X)=x_1^2+x_2^2$, $\beta_1=(0.5,0.5,0,\ldots,0)^T$, and $\beta_2=(0,\ldots,0,0.5,0.5)^T$. For models (i) and (ii), the responses lie on unit spheres. Linear Fr\'echet sufficient dimension reduction for model (i) aims at finding the central subspace $\mathcal{S}_{Y|X}$ with $d=2$, which is the column space spanned by $(\beta_1,\beta_2)$. And the purpose of nonlinear Fr\'echet sufficient dimension reduction for model (ii) is to identify the central class $\mathcal{G}_{Y|X}$ with $d=1$, which is comprised of all measurable functions of $f_1(X)$.

To address this issue, we in this paper propose a novel linear Fr\'echet sufficient dimension reduction method to recover the central subspace $\mathcal{S}_{Y|X}$  defined based on (\ref{linear SDR}) with metric space valued response $Y$. We also provide a consistent estimator of the structural dimension $d$, which is the dimension of the central subspace. The new method is further generalized to nonlinear Fr\'echet sufficient dimension reduction (\ref{nonlinear SDR}) via the reproducing kernel Hilbert space.
The proposed linear and nonlinear  Fr\'echet sufficient dimension reduction estimators are shown to be unbiased for the central subspace $\mathcal{S}_{Y|X}$ and the central class $\mathcal{G}_{Y|X}$ respectively.
Moreover, by taking advantage of the distance metric of the random objects,  both estimators require no numerical optimization or nonparametric smoothing because they can be easily implemented by spectral decomposition of linear operators. The asymptotic convergence results of our proposal are derived for theoretical justifications. We also examine our method via comprehensive simulation studies including responses that consist of probability distributions or lie on the sphere. And the application to the handwritten digits data demonstrates the practical value of our proposal.

\section{Linear Fr\'echet Sufficient Dimension Reduction}
\subsection{Weighted Inverse Regression Ensemble}
 Let $(\Omega, d)$ be a metric space. The linear Fr\'echet sufficient dimension consider the regression with response variable $Y\in \Omega$ and predictors $X\in \mathbb{R}^p$. Let $F$ be the joint distribution of $(X,Y)$ defined on $\mathbb{R}^p\times \Omega$.  And we assume that the conditional distributions $F_{Y|X}$ and $F_{X|Y}$ exist.

 With the linearity condition that $E(X\mid \beta^T X)$ is linear in $X$, \cite{SIR1991} discovered the fundamental property of sliced inverse regression
 \begin{align}\label{prop: sir 1}
 \Sigma^{-1}\{E(X\mid Y)-E(X)\} \in \mathcal{S}_{Y\mid X},
 \end{align}
 where $\Sigma=\mbox{var}(X)$. However, the inverse regression mean $E(X|Y)$ is difficult for us to estimate, as only distances between response objects can be computable for responses in metric space.

Our goal for linear Fr\'echet sufficient dimension is then to borrow the strength of sliced inverse regression without the estimation of the inverse regression function $E(X|Y)$. To introduce our new method, we first recall the martingale difference divergence (MDD) proposed by  \cite{MDD2014} for $Y\in \mathbb{R}^q$  and $X\in \mathbb{R}^p$, which is developed to measure the
conditional mean (in)dependence of $Y$ on $X$, i.e.
\begin{align*}
E(Y|X)=E(Y), \quad \mbox{almost surely}.
\end{align*}
To be specific, $\text{MDD}(Y|X)$ is defined as a nonnegative number that satisfies
\begin{align*}
\text{MDD}^2(Y|X)=- E\left[\{Y-E(Y)\}^T\{Y'-E (Y')\} \| X-X'\|\right],
\end{align*}
where $(X',Y')$ is an independent copy of $(X,Y)$, and $\|\cdot\|$ stands for the Euclidean distance.

To inherit the spirit of sliced inverse regression, we switch the roles of $X$ and $Y$ in martingale difference divergence, and define the following $p\times p$ matrix
\begin{align*}
\Lambda=- E\left[\{X-E(X)\}\{X'-E(X')\}^T d(Y,Y')\right],
\end{align*}
for $(X,Y)\in \mathbb{R}^p\times \Omega$. By the property of conditional expectation, we have
\begin{align}\label{Def: WIRE}
\Lambda=- E\left[E\{X-E(X)|Y\}E\{X'-E(X')|Y'\}^T d(Y,Y')\right].
\end{align}
Invoking the appealing property (\ref{prop: sir 1}) of sliced inverse regression, we see that
\begin{align*}
\Sigma^{-1}\Lambda=-\Sigma^{-1}E\left[E\{X-E(X)|Y\}E\{X'-E(X')|Y'\}^T d(Y,Y')\right]\in \mathcal{S}_{Y|X}.
\end{align*}
We summarize this property in the following proposition.

\begin{proposition}\label{property M1}  $\Lambda$ is positive semidefinite. Assume the linearity condition holds true, then
	$$\textup{Span}\left\{\Sigma^{-1}\Lambda \right\}\subseteq \mathcal{S}_{Y|X}.$$
\end{proposition}
From (\ref{Def: WIRE}), $\Lambda$ can be viewed as the weighted average ensemble of the inverse regression mean $E(X|Y)$, where the weight function is the distance $d(Y,Y')$. We thus call our new method as weighted inverse regression ensemble. The weighted inverse regression ensemble can also be applied for classical linear sufficient dimension reduction with $Y\in \mathbb{R}^q$ and $d(Y,Y')=\|Y-Y'\|$ being the Euclidean distance. Moreover, choosing the number of slices for sliced inverse regression is a longstanding issue in the literature. Compared to sliced inverse regression, our proposal is completely slicing free and is readily applicable to multivariate response data.

Let $M=\Sigma^{-1}\Lambda$ and $(\beta_1,\ldots,\beta_d)$ be the left singular vectors of $M$ corresponding to the $d$ largest singular values. Then  Proposition \ref{property M1} suggests that $(\beta_1,\ldots,\beta_d)$ provides a basis of $\mathcal{S}_{Y\mid X}$. Given a random sample $\{(X_i, Y_i), i=1,\ldots,n\}$ from $(X,Y)$, then $\mu=E(X)$ and $\Sigma=var(X)$ can be estimated as  $\hat \mu=E_n(X)$ and $\hat\Sigma=E_n\{(X-\hat\mu)(X-\hat \mu)^T\}$, where $E_n(\cdot)$ indicates the sample average $n^{-1}\sum_{i=1}^n (\cdot)$. Moreover, we can adopt U-statistics to estimate $\Lambda$ as
\begin{align*}
\hat{\Lambda}=-\sum_{1\le i\neq j\le n} (X_i-\hat{\mu})(X_j-\hat{\mu})^T d(Y_i,Y_j)/\{n(n-1)\}.
\end{align*}
Conduct singular value decomposition on $\hat{M}=\hat\Sigma^{-1}\hat{\Lambda}$. We then adopt the top $d$ left singular vectors $(\hat{\beta_1},\ldots,\hat{\beta}_d)$ of $\hat M$ to recover $\mathcal{S}_{Y|X}$ in  the sample level. And we introduce the following notations to present the central limit theory for the estimation of the central subspace.

\begin{align*}
&\Gamma(X)=(X-\mu)(X-\mu)^T-\Sigma,\quad \Lambda^{(1)}(X,Y,X', Y')=-(X-\mu)(X'-\mu)^T d(Y,Y'),\\
&\Lambda^{(1)}_{1}(X',Y')=E\{\Lambda^{(1)}(X,Y,X',Y')|X', Y'\},\quad
\vartheta=E\{(X-\mu)d(Y,Y')\},\\
&\Theta(X,Y)=\Lambda^{(1)}_{1}(X,Y)-\Lambda+(X-\mu)\vartheta^T+\vartheta(X-\mu)^T,\\
& \zeta_{\ell}(X,Y)=\Sigma^{-1}\Big\{\Theta(X,Y)\Lambda+\Lambda\Theta(X,Y)
-\Gamma(X)\Sigma^{-1}\Lambda\Lambda^T
-\Lambda\Lambda^T\Sigma^{-1}\Gamma(X)\Big\}\Sigma^{-1},\\
& \Upsilon_\ell(X,Y)=\sum_{j=1,j\neq \ell}^{p}\frac{\beta_j\beta_j^T\zeta_{\ell}(X,Y)\beta_\ell}{\lambda_j^2-\lambda_\ell^2},\quad \ell=1,\ldots,d.
\end{align*}

{\theorem\label{theo: asymptotic M} Assume the linearity condition and the singular values $\lambda_\ell$'s are distinct for $\ell=1,\ldots,d$. In addition, assume that $Ed^2(Y,Y') < \infty$ and $X$ has finite fourth moment, then
	\begin{eqnarray}
	n^{1/2}(\hat\beta_\ell-\beta_\ell)\overset{D}{\longrightarrow} N\left({0}_p, \Sigma_\ell\right),
	\end{eqnarray}
as $n\rightarrow \infty$, where $\Sigma_\ell= cov\{\Upsilon_\ell(X,Y)\} $.}

\subsection{Determination of Structural Dimension $d$}
The estimation of structural dimension $d$ is another focus in sufficient dimension reduction. We adopt the ladle estimator proposed by \cite{Ladle2016} for order determination, which extracts the information contained in both the singular values and the left singular vectors of $M$.

%Let $(X^*_1,Y^*_1),\ldots,(X^*_n,Y^*_n)$ be an independent and identically distributed bootstrap sample from the empirical distribution of {$\{(X_i,Y_i),i=1,\ldots,n\}$}.
Let $\mathcal{B}_k=(\hat\beta_1,\ldots,\hat\beta_k)$ be the $p\times k$ matrix consisting of the principal $d$ left singular vectors of $\hat M$.
 {We randomly draw $n$ bootstrap samples of size $n$} and denote the realization of $\mathcal{B}_k$ based on the $i$th bootstrap sample as $\mathcal{B}^*_{k,i}$. The following function is proposed to evaluate the difference between $\mathcal{B}_k$ and its bootstrap counterpart
\begin{align*}
f^0_n(k)=
\begin{cases}
0, & k=0,\\
n^{-1}\sum_{i=1}^n \{1-|\text{det}(\mathcal{B}_k^T \mathcal{B}^*_{k,i})|\}, & k=1,\ldots,p.
\end{cases}
\end{align*}
And $f^0_n(k)$ is further normalized as
%\begin{align*}
$f_n(k)=f^0_n(k)/\{1+\sum_{i=0}^{r_p} f_n^0(i)\},$
%\end{align*}
{where $r_p=p-1$ if $p\leq 10$, $r_p=\lfloor p/\log p \rfloor$ if $p>10$} and $\lfloor a \rfloor$ stands for the largest integer no greater than $a$.
The effect of the singular values are measured as
%\begin{align*}
$g_n(k)=\hat\lambda^2_{k+1}/(1+\sum_{{i=0}}^{r_p}\hat\lambda^2_{i+1}), \quad k=0,1,\ldots,r_p.$
%\end{align*}
And the ladle estimator for structural dimension $d$ is constructed as
\begin{align*}
\hat d=argmin_{k=0,\ldots,r_p} \{f_n(k)+g_n(k)\}.
\end{align*}

To obtain the desired estimation consistency of the structural dimension, we assume that
{\assumption \label{assum2} The bootstrap version kernel matrix $M^*$ satisfies
	\begin{equation}\label{ass1}
	n^{1/2}\{{\rm{vech}} (M^*(M^*)^T)-{\rm{vech}} (\widehat{M}(\widehat{M})^T)\}\rightarrow N(0,{\rm{var}}[{\rm{vech}} \{H(X,Y)\}])
	\end{equation}
	where ${\rm{vech}}(\cdot)$ is the vectorization of the upper triangular part of a matrix and $H(X,Y)=-\Sigma^{-1}(\Gamma(X)-\Sigma)\Sigma^{-1}+\Sigma^{-1}(\Lambda^{(1)}(X,Y)-\Lambda)-\Sigma^{-1}(X-\mu)\vartheta^T-\Sigma^{-1}\vartheta(X-\mu)^T$.
}

{\assumption \label{assum3} For any sequence of nonnegative random variables {{$\{Z_n:n=1,2,\ldots\}$}} involved in this paper, if $Z_n=O_p(c_n)$ for some sequence $\{c_n: n \in N\}$ with $c_n>0$, then $E(c_n^{-1}Z_n)$ exist for each $n$ and $E(c_n^{-1}Z_n)=O(1)$.}\\

 From the proof of Theorem \ref{theo: asymptotic M}, we know that $n^{1/2}\{{\rm{vech}}(\hat{M}\hat M^T)-{\rm{vech}}(MM^T)\}$ also converges in distribution to the right-hand side of \eqref{ass1}. Assumption \ref{assum2} amounts to asserting that asymptotic behaviour of $n^{1/2}(M^{*}(M^{*})^T-\hat{M}\hat{M}^T)$ mimics that of $n^{1/2}(\hat{M}\hat{M}^T-MM^T)$. The validity of this self-similarity was discussed in \cite{Bickel1981}, \cite{Ladle2016}. Assumption \ref{assum3} has also been adopted and verified by \cite{Ladle2016}. The following theorem confirms that the number of useful sufficient predictors for linear Fr\'echet sufficient dimension reduction
can be consistently estimated.
\begin{theorem}\label{theo: ladle} Assume $Ed^2(Y,Y') < \infty$ and $X$ has finite fourth moment. And suppose Assumptions (1)--(2) hold, then
	\begin{align*}
	P_r\{\lim_{n\rightarrow\infty} P_r(\hat d =d |\mathcal{D}) {=1}\} =1,
	\end{align*}
	where $\mathcal{D}=\{(X_1,Y_1),(X_2,Y_2),\ldots\}$ is a sequence of independent copies of $(X,Y)$.
\end{theorem}

\section{Nonlinear Fr\'echet Sufficient Dimension Reduction}
As the descendant of sliced inverse regression, the weighted inverse regression ensemble method will share the similar limitation with sliced inverse regression when dealing with regression functions that are symmetric about the origin (\cite{Cook1991}). To remedy this problem and to further extend the scope of our method, we in the next will consider nonlinear Fr\'echet sufficient dimension reduction defined in (\ref{nonlinear SDR}) using the reproducing kernel Hilbert space. Let $\mathcal{H}_X$ be a reproducing kernel Hilbert space of functions of $X$ generated by a positive definite kernel $\kappa_X$. To extend the idea of weighted inverse regression ensemble for nonlinear Fr\'eceht sufficient dimension reduction, we introduce a new type of operator in the following.

{\definition Let $\mu_X(\cdot)=E\kappa_{X}(\cdot,X)$.  %be the mean element of $X$ in $\mathcal{H}_{X}$.
For $(X,Y)$ and its independent copy $(X',Y')$, we define the weighted inverse regression ensemble operator $\Lambda_{XX'}: \mathcal{H}_{X'}\rightarrow \mathcal{H}_{X}$ such that
\begin{align*}
\Lambda_{XX'}=-E\{(\kappa_{X}(\cdot,X)-\mu_X(\cdot))\otimes(\kappa_{X}(\cdot,X')-\mu_{X'}(\cdot))d(Y,Y')\}.
\end{align*}
}

We assume the following regularity assumptions for theoretical investigations into $\Lambda_{XX'}$.

{\assumption \label{assum4} $E\kappa_X(X,X)<\infty$.}%$\mathrm{ker}(\Sigma_{XX})=\{0\}$

{\assumption \label{assum5} The operator $\Lambda_{XX'}$ has a representation as $\Lambda_{XX'}=\Sigma_{XX}S$, where $S$ is a unique bounded linear operator such that $S:\mathcal{H}_X\rightarrow\mathcal{H}_X $, $S=Q_XSQ_X$ with $Q_X$ being the projection operator mapping $\mathcal{H}_X$  on to $\overline{\rm{ran}}(\Sigma_{XX})$, and $\overline{\rm{ran}}(\Sigma_{XX})$ stands for the closure of the range of the covariance operator $\Sigma_{XX}$.}

{\assumption \label{assum6} $\mathcal{G}_{Y|X}$ is dense in $L_2(P_X|\mathcal{M}_{Y|X})$, where $L_2(P_X|\mathcal{M}_{Y|X})$ denotes the collection of $\mathcal{M}_{Y|X}$-measurable functions in $L_2(P_X)$ and $\mathcal{M}_{Y|X}=\sigma[f(X)]$.}

{\assumption \label{assum7} The eigenfunctions $\psi_i$'s are included in $\mathcal{R}(\Sigma_{XX})$, where $\mathcal{R}(\Sigma_{XX})=\{\Sigma_{XX}f: f\in \mathcal{H}_X\}$.}

{\assumption \label{assum8} Let $(\varepsilon_n)_{n=1}^{\infty}$ be a sequence of positive numbers such that
	$$
	\lim_{n\rightarrow \infty} \varepsilon_n=0,\ \ \ \lim_{n\rightarrow \infty} n^{-1/2}/\varepsilon_n^{3/2}=0.
	$$. }

 Assumption \ref{assum4}, \ref{assum6} and \ref{assum7} are commonly used conditions for reproduce kernel Hilbert spaces in the literature (\cite{NonlinearSDR, Li:2018}).  Assumption \ref{assum5} is similar to the result of Theorem 1 of \cite{Baker1973} that defines the correlation operator, which will guarantee that our proposed operator is compact. Assumption \ref{assum8} is adopted by \cite{Fukumizu2007} for asymptotic analysis of kernel type methods, which is helpful to establish the estimation consistency of nonlinear weighted inverse regression ensemble method.

\begin{proposition} $\Lambda_{XX'}$ is a bounded linear and self-adjoint operator. For any $f,g \in \mathcal{H}_X$,
\begin{align*}
\langle f,\Lambda_{XX'}g\rangle=-E\{(f(X)-Ef(X)) (g(X')-E{g(X')})d(Y,Y')\}.
\end{align*}
Moreover, there exists a separable $\mathbb{R}$-Hilbert space $\mathcal{H}$ and a mapping $\phi: \Omega\rightarrow \mathcal{H}$ such that
\begin{align*}
\langle f,\Lambda_{XX'}f\rangle=2\{E[(f(X)-Ef(X)) (\phi(Y)-E\phi(Y))]\}^2=2(cov[f(X),\phi(Y)])^2.
\end{align*}
\end{proposition}

Proposition 2 implies that our proposed new operator enjoys a similar fashion as the commonly used covariance operator. The new operator also has the potential to  measure the dependence between Euclidean $X$ and random objects $Y$ due to its similarity to the popular Hilbert-Schmidt Independence Criterion (\cite{Gretton2005}). Denote the covariance operator of $X$ as $\Sigma_{XX}=E\{\kappa_{X}(\cdot,X)\otimes \kappa_{X}(\cdot,X)\} - E\kappa_{X}(\cdot,X) \otimes E\kappa_{X}(\cdot,X) $.. The next proposition reveals the relationship between $\Lambda_{XX'}$ and the central class $\mathcal{G}_{Y|X}$.
%For linear operator $T$, the symbol  $\textup{ran} (T)$ stands for the closure of the range of $T$

\begin{proposition}\label{property M2}  Suppose assumptions (3)--(5) hold, then
$$\overline{\textup{ran}}\left\{\Sigma_{XX}^{-1}\Lambda_{XX'}\right\}\subseteq \mathcal{G}_{Y|X}.$$
\end{proposition}

\begin{proposition}\label{property M3}  Suppose assumptions (3)--(5) hold and $\mathcal{G}_{Y|X}$ is complete.
Then,
$$\overline{\textup{ran}}\left\{\Sigma_{XX}^{-1}\Lambda_{XX'}\right\}= \mathcal{G}_{Y|X}.$$
\end{proposition}

Proposition \ref{property M2} suggests that the range of $\Sigma_{XX}^{-1}\Lambda_{XX'}$  is always contained in the central class $\mathcal{G}_{Y|X}$.  Proposition \ref{property M3} further extends the scope in the following aspects. First, it confirms that the nonlinear weighted inverse regression ensemble method is exhaustive in recovering the central class. The exhaustiveness of our nonlinear proposal is an appealing property which may not exist in the linear setting. The second is that the nonlinear weighted inverse regression ensemble method leads to the minimal sufficient predictor satisfying (\ref{nonlinear SDR}), as sufficiency and completeness together imply minimal sufficiency in classical statistical inference. Last but not least, the nonlinear weighted inverse regression ensemble method does not rely on the linear conditional mean assumption
requiring that $E(X|\beta^T X)$ be linear in $X$. By relaxing such a stringent condition, the nonlinear method will have a wide range of applications.

Let $\Lambda_{XX'}^*$ be the adjoint operator of $\Lambda_{XX'}$.  Proposition \ref{property M3} indicates that
\begin{eqnarray}\label{CM11}
\overline{\textup{ran}}\left\{\Sigma_{XX}^{-1}\Lambda_{XX'}\Lambda_{XX'}^*\Sigma_{XX}^{-1}\right\}=\mathcal{G}_{Y|X},
\end{eqnarray}
The space \eqref{CM11} can be recovered by performing the following generalized eigenvalue problem:
\begin{eqnarray}\label{eq11}
\mathrm{max}\ \ \langle f, \Lambda_{XX'}\Lambda_{XX'}^*f\rangle_{\mathcal{H}_X}, \ \mathrm{s.t.} \ \langle f, \Sigma_{XX}f\rangle_{\mathcal{H}_X}=1, f\bot \mathcal{L}_{k-1},
\end{eqnarray}
where $\mathcal{L}_{k}=\textup{Span}(f_1,\ldots,f_{k-1})$ and $f_1,\ldots,f_{k-1}$ are the solutions to this constrained maximization problem in the previous steps. Define the following sample level estimators
\begin{align*}
&\hat\mu_{X}(\cdot)=E_n[\kappa_{X}(\cdot,X_i)] \quad \hat{\Sigma}_{XX}=E_n\{ (\kappa_{X}(\cdot,X_i)-\hat\mu_X(\cdot))\otimes(\kappa_{X}(\cdot,X_i)-\hat\mu_{X}(\cdot))\},\\
&\hat{\Lambda}_{XX}=-\sum_{1\le i\neq j\le n}(\kappa_{X}(\cdot,X_i)-\hat\mu_X(\cdot))\otimes(\kappa_{X}(\cdot,X_j)-\hat\mu_{X}(\cdot))d(Y_i,Y_j)/(n(n-1)).
\end{align*}
The sample version of (\ref{eq11}) then becomes
\begin{eqnarray}\label{nonlinear sample}
	\mathrm{max}\ \ \langle f, \widehat{\Lambda}_{XX'}\widehat{\Lambda}_{XX'}^*f\rangle_{\mathcal{H}_X}, \ \mathrm{s.t.} \ \langle f, (\widehat{\Sigma}_{XX}+\varepsilon_nI)f\rangle_{\mathcal{H}_X}=1.
\end{eqnarray}

Let $V_{XX'}=\Sigma_{XX}^{-1/2}\Lambda_{XX'}\Lambda_{XX'}^*\Sigma_{XX}^{-1/2}$. Then we can verify that  $f_1=\Sigma_{XX}^{-1/2}\psi_1$, where
\begin{align*}
\psi_1=\arg\max_{\substack{g \in \mathcal{H}_X, \|g\|_{ \mathcal{H}_X}=1}}\langle g, V_{XX'}g\rangle_{\mathcal{H}_X}.
\end{align*}
Let $\hat V_{XX'}=(\hat{\Sigma}_{XX}+\varepsilon_n I)^{-1/2}\hat{\Lambda}_{XX'}\hat{\Lambda}_{XX'}^*(\hat{\Sigma}_{XX}+\varepsilon_n I)^{-1/2}$. Then we have
\begin{align*}
\hat f_1(X)=(\hat{\Sigma}_{XX}+\varepsilon_nI)^{-1/2}\hat{\psi}_1, \quad \hat\psi_1=\arg\max_{\substack{g \in \mathcal{H}_X, \|g\|_{ \mathcal{H}_X}=1}}\langle g, \hat V_{XX'}g\rangle_{\mathcal{H}_X}.
\end{align*}

We in the next establish the estimation consistency of our nonlinear Fr\'ecechet sufficient dimension reduction approach. Although we only focus on the first eigenfunction in the following theorem, similar asymptotic results can be derived for the entire central .

{\theorem\label{theorem3} Suppose assumptions (3)--(7) hold. {{In addition, assume that $Ed^2(Y,Y') < \infty$}}, then as $n\rightarrow \infty$
\begin{eqnarray*}
&&\|\hat{V}_{XX'}-V_{XX'}\|_{\text{HS}}=o_p(1),\ \ \
|\langle\hat{\psi}_1,\psi_1\rangle_{ \mathcal{H}_X}|\stackrel{P}{\longrightarrow} 1,\\
&&\|\{\hat{f}_1(X)-E\hat{f}_1(X)\}-\{f_1(X)-Ef_1(X)\}\|{\longrightarrow} 0,
\end{eqnarray*}
where $\|\cdot\|$ in this theorem is the standard $L_2$ norm to measure the distance of functions and $\| \cdot \|_{\text{HS}}$ denotes the Hilbert-Schmidt norm.
}

Let $\eta_i={{\kappa_X(\cdot,X_i)}}-\hat{\mu}_X(\cdot), i=1,\ldots,n$. The estimated eigenfunctions $\hat f_\ell$'s solved from (\ref{nonlinear sample}) can be further characterized as a linear combination of $\eta_i$ such that $\hat f_\ell=\sum_{i=1}^n a_{\ell,i} \eta_i$. Denote $\alpha_\ell=(a_{\ell,1},\ldots, a_{\ell,n})^T$. The next proposition indicates that $\alpha_\ell$ can be obtained through solving an eigen-decomposition problem.

\begin{proposition}\label{coordinate representation.} Let $K_n$ be the $n\times n$ kernel matrix whose $(i, j)$th element is $\kappa_X(X_i, X_j)$. Denote $J_n$ as the $n\times n$ matrix whose elements are all one. Define $G_X= (I_n-J_n/n)K_n
(I_n-J_n/n)$ and let $D_Y$ be the $n\times n$  matrix whose $(i, j)$th element is $d(Y_i,Y_j)$. Then we have $ G_X\alpha_\ell= \gamma_\ell$, where $\gamma_\ell$ is the $\ell$th eigenvector of the following matrix
\[(G_X+\varepsilon_n I_n)^{-1}G_XD_YG_XD_YG_X(G_X+\varepsilon_n I_n)^{-1}.\]
\end{proposition}
 Let $\hat\alpha_\ell= (G_X+\varepsilon_n I_n)^{-1}\gamma_\ell$. Inspired by Proposition 4, the $\ell$th estimated sufficient predictor can then be represented as $\hat f_{\ell}=\sum_{i=1}^n \hat a_{\ell,i} \eta_i$, where $\hat a_{\ell,i}$  is the $i$th element of the $n\times 1$ vector $\hat\alpha_\ell$.

\section{Numerical Studies}
We consider the following models with responses being complex random objects.

\vspace{0.2cm}
{Model I}. Let $\beta_1=(1,1,0,\ldots,0)^T$ and $\beta_2=(0,\ldots,0,1,1)^T$.
$X\sim U[0,1]^p$ and $Y$ is the distribution function with its quantile function being $Q_Y(\tau)=\mu_Y+\sigma_Y \Phi^{-1}(\tau)$, where $\Phi(\cdot)$ is the cumulative distribution function of standard normal, $\mu_Y|X\sim N(\exp(\beta_1^T X), 0.5^2)$. And we consider $\sigma_Y=1$ as case (i) and $\sigma_Y=|\beta_2^T X|$ as case (ii). As $Y$ and its independent copy $Y'$ are
random distribution functions, then we adopt the Wasserstein distance as the metric $d(Y,Y')$. For case (i), $\mathcal{S}_{Y|X}=\text{Span}(\beta_1)$ and $d=1$.
For case (ii), $\mathcal{S}_{Y|X}=\text{Span}(\beta_1,\beta_2)$ and $d=2$.

\vspace{0.2cm}
{Model II.} Consider the following Fr\'echet regression function
\[{m(X)=(\cos(f_1(X)), \sin(f_1(X)))}.\]
Generate $\varepsilon$ from $N(0,0.1^2)$ on the tangent line of $m(x)$. And the response $Y$ is generated as
\[Y= \cos(\varepsilon)m(x) \oplus  \sin(\varepsilon) \varepsilon/|\varepsilon|,\]
where $\oplus$ stands for vector addition. We can verify that $Y\in \Omega$ where $\Omega$ is the unit circle in $\mathbb{R}^2$. Then $d(Y,Y')$ is naturally chosen as the geodesic distance $\arccos(Y^T Y')$. Moreover, we consider case (i)  $f_1(X)=\beta_1^T X$ with $X\sim U[0,1]^p$  and cased (ii) where $f_1(X)=(x_1^2+x_2^2)^{1/2}$ with $X\sim N(0_p,I_p)$ for both linear and nonlinear Fr\'echet sufficient dimension reduction.

\vspace{0.2cm}
{Model III.} Generate $\varepsilon_i$ from $N(0,0.1^2)$ for $i=1,2$. We consider two cases in this study.  The model structure of case (i) is exactly the same as our motivating example (i) illustrated in Section 1 with $X\sim U[0,1]^p$. For case (ii), the response $Y$ is generated as
\[
 Y=(\sin(f_1(X)+\varepsilon_1)^{1/3}\sin(f_2(X)+\varepsilon_2)^{1/3},\sin(f_1(X)+\varepsilon_1)^{1/3}\cos(f_2(X)+\varepsilon_2)^{1/3},\cos(f_1(X)+\varepsilon_1)^{1/3}),\]
where $f_1(X)=x_1^2+x_2^2$ and $f_2(X)=x^2_{p-1}+x^2_p$, and $X\sim N(0_p,I_p)$.
We see that $Y\in \Omega$ where $\Omega$ is the unit sphere in $\mathbb{R}^3$. Again $d(Y,Y')=\arccos(Y^T Y')$ is the geodesic distance.

\vspace{0.2cm}
Model I and case (i) of and Model II and III are adopted for linear Fr\'echet sufficient dimension reduction, while the two rest cases are examples for nonlinear Fr\'echet sufficient dimension reduction. Let $\hat{\beta}$ and $\hat f(X)$ be our proposed linear and nonlinear weighted inverse regression estimators. To evaluate our proposal for linear Fr\'echet sufficient dimension reduction, we adopt the trace correlation (\cite{tracecorr1998}) defined as $r^2=tr(P_\beta P_{\hat{\beta}})/d$ , where $P_\beta=\beta(\beta^T\beta)^{-1}\beta^T$.  To assess the performance of nonlinear Fr\'echet sufficient dimension reduction, we utilize the square distance correlation $\rho^2(f(X),\hat f(X))$ proposed by \cite{DC2007}. The square distance correlation can also be adapted to linear Fr\'echet sufficient dimension reduction as $\rho^2(\beta^T X,\hat\beta^T X)$. And larger values of $r^2$ or $\rho^2$ indicate better estimation.

We consider $n=100, 200, 300, 400$ and $p=10,20,30$. Treating $d$ as known, Table 1 and 2 summarize the { mean} values of $r^2$ and $\rho^2$ based on {100} repetitions with different combinations of $n$ and $p$. We can see from Table 1 that the original weighted inverse regression ensemble works well except for case (ii) of Model II and III with U-shape structure, which is consistent with our theoretical anticipation. As an effective remedy, the nonlinear weighted inverse regression produces a satisfying result as seen from Table 2, in which the tuning parameter is simply set as $\varepsilon_n=0.001$ and Gaussian kernel  $\kappa_X(X,X')=\exp\{-\|X-X'\|^2/(2\sigma^2_\kappa)\}$ is adopted with $\sigma_\kappa=0.1$. The results for order determination are presented in Table 3, where the entries are the number of correct estimation of $d$ out of $100$ repetitions. Table 3 shows that the ladle estimator in combination with weighted inverse regression ensemble works well, with percentage of correct estimation reaching as high as $100\%$ for most cases.

\begin{table}[h]\label{ta1}
	\def~{\hphantom{0}}
	\caption{\it The Averages of $r^2$ and $\rho^2$  for the estimation of $\mathcal{S}_{Y|X}$ based on $100$ simulation runs. }{%
		\begin{tabular}{cccccccccccccc}
			&&\multicolumn{4}{c}{Model I }&\multicolumn{4}{c}{Model II}&\multicolumn{4}{c}{Model III}\\[5.68pt]
			&$(p,n)$    &	100      &   200	  &	  300	     & 	400	    &	   	 100     &   200	  &   300	    &  400	&	   	 100     &   200	  &   300	    &  400    \\
			\multirow{6}{*}{Case i}&10         &	0.979	 &   0.988   &   0.993   	 &  0.995	&	   	0.988	 &   0.996   &   0.997   	 &  0.998	&	0.987	 &   0.995   &   0.996   	 &  0.997\\
			&           &	0.973	 &   0.984   &   0.990  	 &  0.993	&	  	0.988	 &   0.994   &   0.996  	 &  0.967	&	0.987	 &   0.994   &   0.995  	 &  0.997\\
			&20         &	0.947	 &   0.976   &   0.976	     &  0.989	&	   	0.971	 &   0.991   &   0.994	     &  0.995   &	0.966	 &   0.986   &   0.991	     &  0.994\\
			&           &	0.943	 &   0.970   &   0.970	     &  0.985	&	 	0.970	 &   0.989   &   0.992	     &  0.993   &	0.873	 &   0.986   &   0.990	     &  0.993\\
			&30	       &	0.908    &   0.960   &   0.976       &  0.982   &	  	0.961    &   0.987   &   0.990       &  0.992   &	0.945    &   0.977   &   0.986       &  0.989\\
			&           &	0.917	 &   0.955   &   0.970	     &  0.976 	&	  	0.964	 &   0.985   &   0.988	     &  0.990   &	0.959	 &   0.978   &   0.985	     &  0.989\\ [5.68pt]
			%\multicolumn{6}{c}{Model II (case i)}&\multicolumn{6}{c}{Model II (case ii)}\\[5.68pt]
			%$(p,n)$    &	100      &   200	  &	  300	     & 	400	    &	 100     &   200	  &   300	   &  400&	 100     &   200	  &   300	   &  400	    \\
			\multirow{6}{*}{Case ii}&	       10   &	0.990	 &   0.995    &   0.997   	&  0.998 &	0.255	 &   0.253    &   0.234    &  0.278	&	0.379 	 &   0.369     &   0.382  &  0.377\\
			&        &	0.989	 &   0.994    &   0.996   	&  0.997 &	0.122	 &   0.092    &   0.078    &  0.085	&	0.231 	 &   0.178     &   0.166  &  0.156\\
			&     20	&	0.976	 &   0.988    &   0.993	    &  0.995  &	0.136	 &   0.125    &   0.123    &  0.134 &	0.177 	 &   0.188     &   0.188  &  0.196  \\
			&	        &	0.978	 &   0.989    &   0.992	    &  0.994 &	0.092	 &   0.053    &   0.038    &  0.039 &	0.163 	 &   0.097     &   0.078  &  0.068  \\
			&    30   &	0.956    & 	 0.982    &   0.988     &  0.991 &	0.094    & 	 0.085    &   0.084    &  0.094 &	0.118   & 	 0.122     &   0.126  &  0.134  \\
			&    	    &	0.967	 &   0.982    &   0.987	    &  0.991 &  0.085	 &   0.048    &   0.032    &  0.028 &  0.147 	 &   0.081     &   0.060  &  0.049  \\																												
	\end{tabular}}
\caption*{ The average $r^2$ and $\rho^2$ are listed in the first and second rows for each $p$.}
\end{table}

 \begin{table}[h]\label{t2}
	\def~{\hphantom{0}}
	\caption{\it The Averages of $\rho^2$ for the estimation of $\mathcal{G}_{Y|X}$ based on $100$ simulation runs. }
%\newsavebox{\tablebox}
%\begin{lrbox}{\tablebox}
\centering
		\begin{tabular}{cccccccccc}
			&\multicolumn{4}{c}{Model II (case ii)}&&\multicolumn{4}{c}{Model III (case ii)}\\[5.68pt]
			$(p,n)$    &	100      &   200	  &	  300	     & 	400	    &   $(p,n)$    &	100      &   200	  &	  300	     & 	400	      	  \\
			10         &	0.948	 &   0.952   &   0.952   	 &  0.952	&	10         &	0.818	 &   0.828   &   0.827   	 &  0.828   \\
			20         &	0.952	 &   0.952   &   0.951	 &  0.952	&	20         &	0.823	 &   0.834   &   0.828	 &  0.829	  \\
			30	       &	0.952   &   0.951   &   0.952     &  0.952  &	30	       &	0.827   &   0.826   &   0.827     &  0.827   \\				
	\end{tabular}
%\end{lrbox}
%\scalebox{0.7}{\usebox{\tablebox}}
\end{table}

\begin{table}[h]\label{t2}
	\def~{\hphantom{0}}
	\caption{\it The number of correctly estimation for $d$ among $100$ simulation runs. }{%
		\begin{tabular}{cccccccccccccccccccccccc}	
			&\multicolumn{4}{c}{Model I (case i)}&\multicolumn{4}{c}{Model I (case ii)}&\multicolumn{4}{c}{Model II (case i)}&\multicolumn{4}{c}{Model III (case i)}\\[5.68pt]
			$(p,n)$    &	100      &   200	  &	  300	     & 	400	    &	   	 100     &   200	   &   300	        &  400	 &	100      &   200	  &	  300	     & 	400     &	100      &   200	  &	  300	     & 	400   \\
			10         &	100  	 &   100      &   100   	 &  100	    &	   	 100	 &   100       &   100   	    &  100   &	100  	 &   100      &   100   	 &  100&	100  	 &   100      &   100   	 &  100	\\
			20         &	100	     &   100      &   100	     &  100	    &	   	 100	 &   100       &   100	        &  100    &	100  	 &   100      &   100   	 &  100&	100  	 &   100      &   100   	 &  100  \\
			30	       &	100      &   100      &   100        &  100     &	     100     & 	 100       &   100          &  100    &	100  	 &   100      &   100   	 &  100&	99  	 &   100      &   100   	 &  100  \\
	\end{tabular}}
\end{table}

To better illustrate the performance of our nonlinear Fr\'echet sufficient dimension reduction method, we in Figure 2 present the 2-D scatter
plots for the nonlinear sufficient predictors from case (ii) of Model III versus their sample estimates obtained by the nonlinear weighted inverse regression ensemble method with $n=100$ and $p=10$.  The left panel is the 2-D scatter plots for  the first
nonlinear sufficient predictor $f_1(X)=x_1^2+x_2^2$ versus its estimate $\hat f_1(X)$; the right panel is the 2-D scatter plots for the second
nonlinear sufficient predictor  $f_2(X)=x_{p-1}^2+x_{p}^2$  versus its estimate $\hat f_2(X)$.  Figure 2 shows a strong relationship between $f_i$ and $\hat f_i$ for $i=1,2$. $\hat f_i$ behaves like a measurable function of $f_i$, which is consistent with our theoretical development as our focus on nonlinear Fr\'echet sufficient dimension reduction is the $\sigma$-field generated by $f_1$ and $f_2$ rather than $f_1$ and $f_2$ themselves. As $f^{1/3}_i$ is a measurable function of $f_i$, then $f_i^{1/3}$ can also be regarded as the nonlinear sufficient predictor. We in Figure 3 present the 2-D scatter
plots for the nonlinear sufficient predictors $f_1^{1/3}$ and  $f_2^{1/3}$ versus
$\hat f_1$ and  $\hat f_2$. We can observe a strong linear pattern between $f^{1/3}_i$ and $\hat f_i$, which again verify that our proposed nonlinear weighted inverse regression ensemble method is effect in recovering the central class $\mathcal{G}_{Y|X}$ with responses $Y$ being metric space valued random objects.

\begin{figure}[h]
	\centering
	%\caption{The scatter diagram for nonlinear weighted inverse regression.}
	\includegraphics[width=4.68in,height=2.68in]{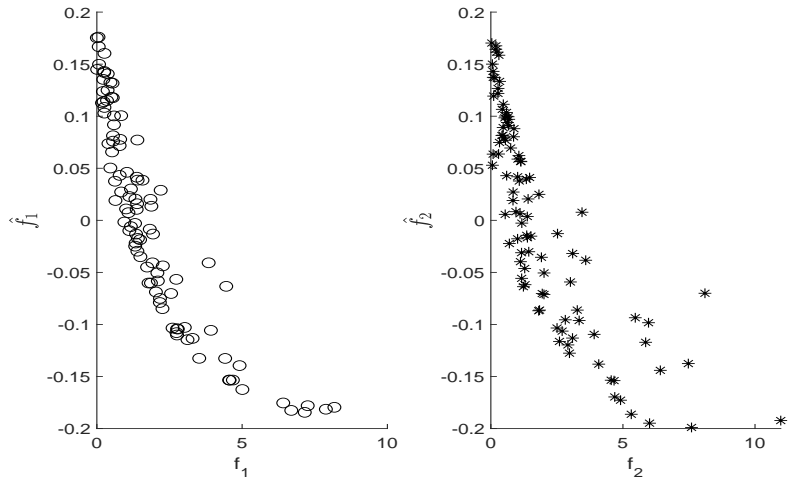}
	\caption{Scatter plots of nonlinear sufficient predictors $f_i$'s versus their estimates $\hat f_i$'s.}
\end{figure}

\begin{figure}[h]
	\centering
	%\caption{The scatter diagram for nonlinear weighted inverse regression.}
	\includegraphics[width=4.68in,height=2.68in]{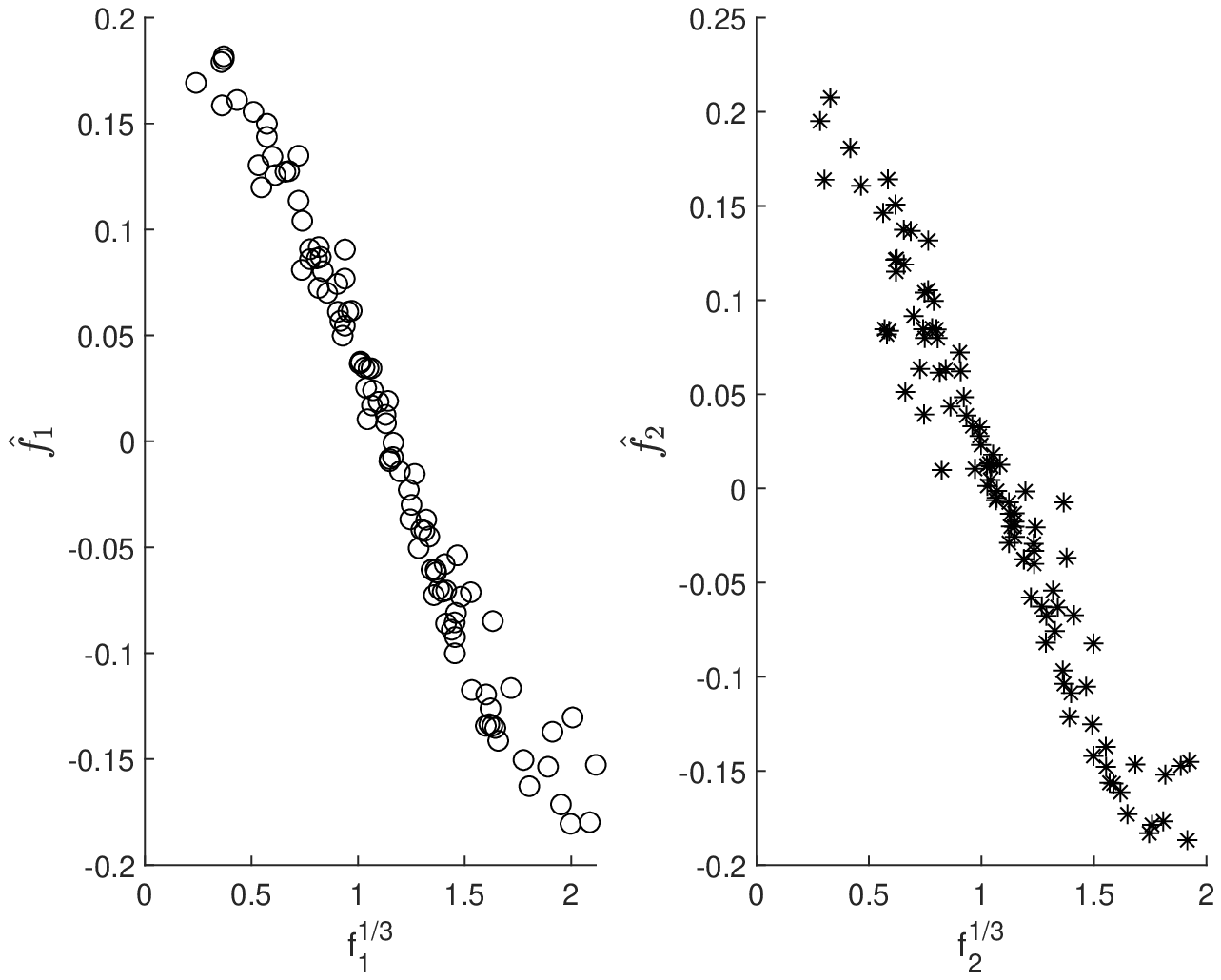}
	\caption{Scatter plots of nonlinear sufficient predictors $f^{1/3}_i$'s versus their estimates $\hat f_i$'s.}
\end{figure}

\section{Handwritten Digits Data}

\begin{figure}[!h]
\centering
\setlength{\abovecaptionskip}{0.cm}
\subfigure{
\begin{minipage}[c]{0.3\textwidth}
\centering
\includegraphics[height=2in,width=1.8in]{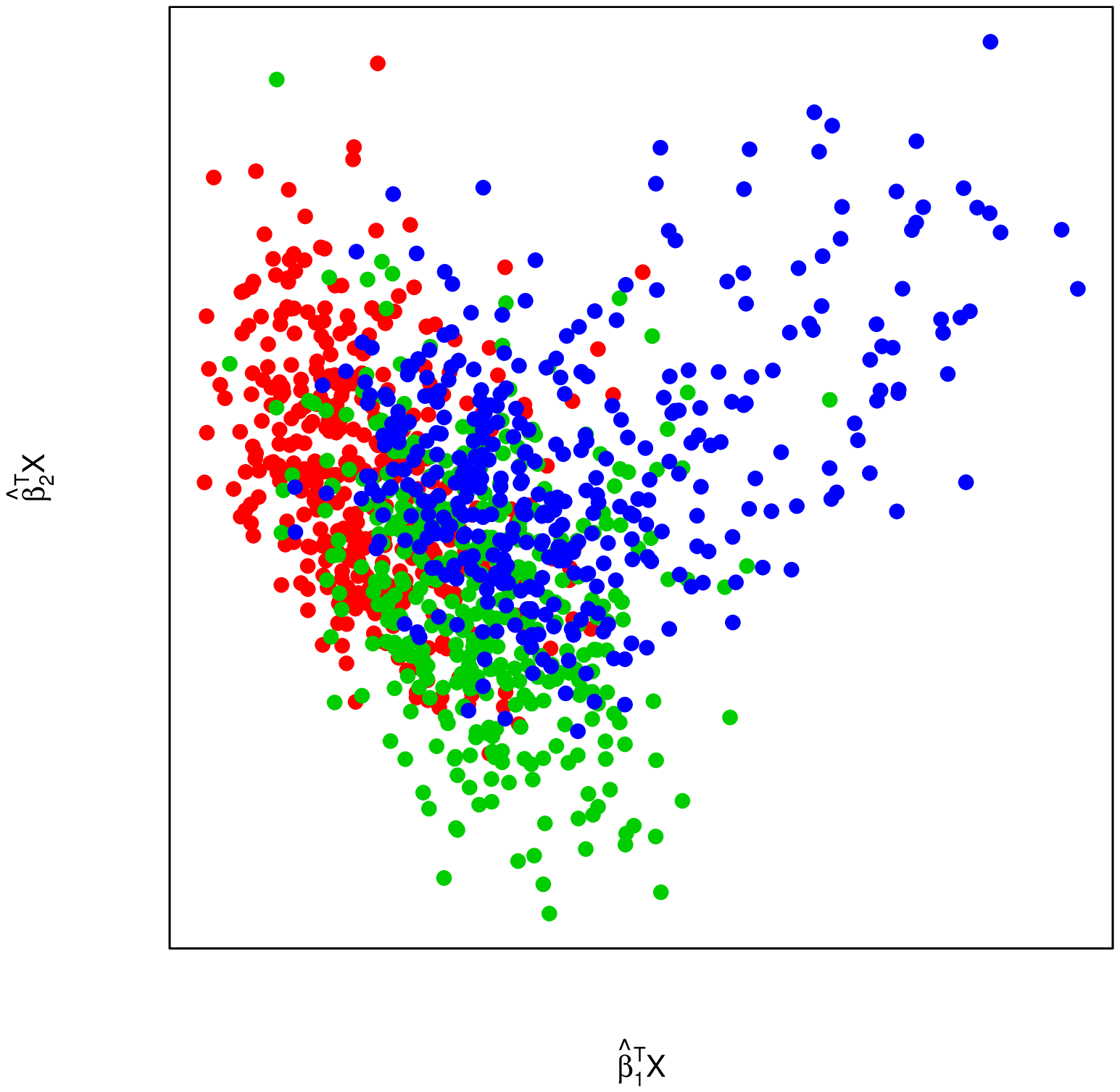}
\end{minipage}
\begin{minipage}[c]{0.3\textwidth}
\centering
\includegraphics[height=2in,width=1.8in]{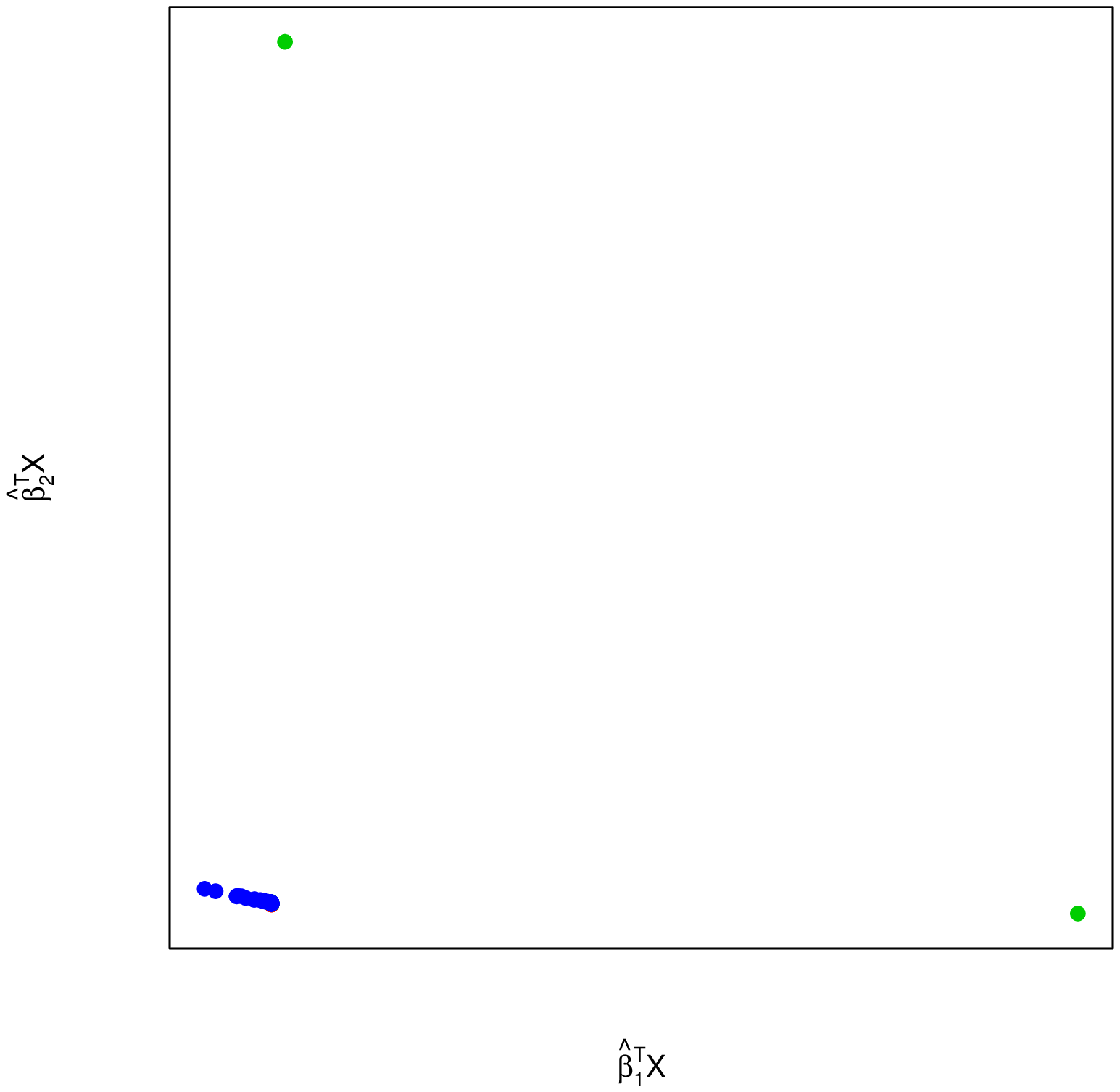}
\end{minipage}
\begin{minipage}[c]{0.3\textwidth}
\centering
\includegraphics[height=2in,width=1.8in]{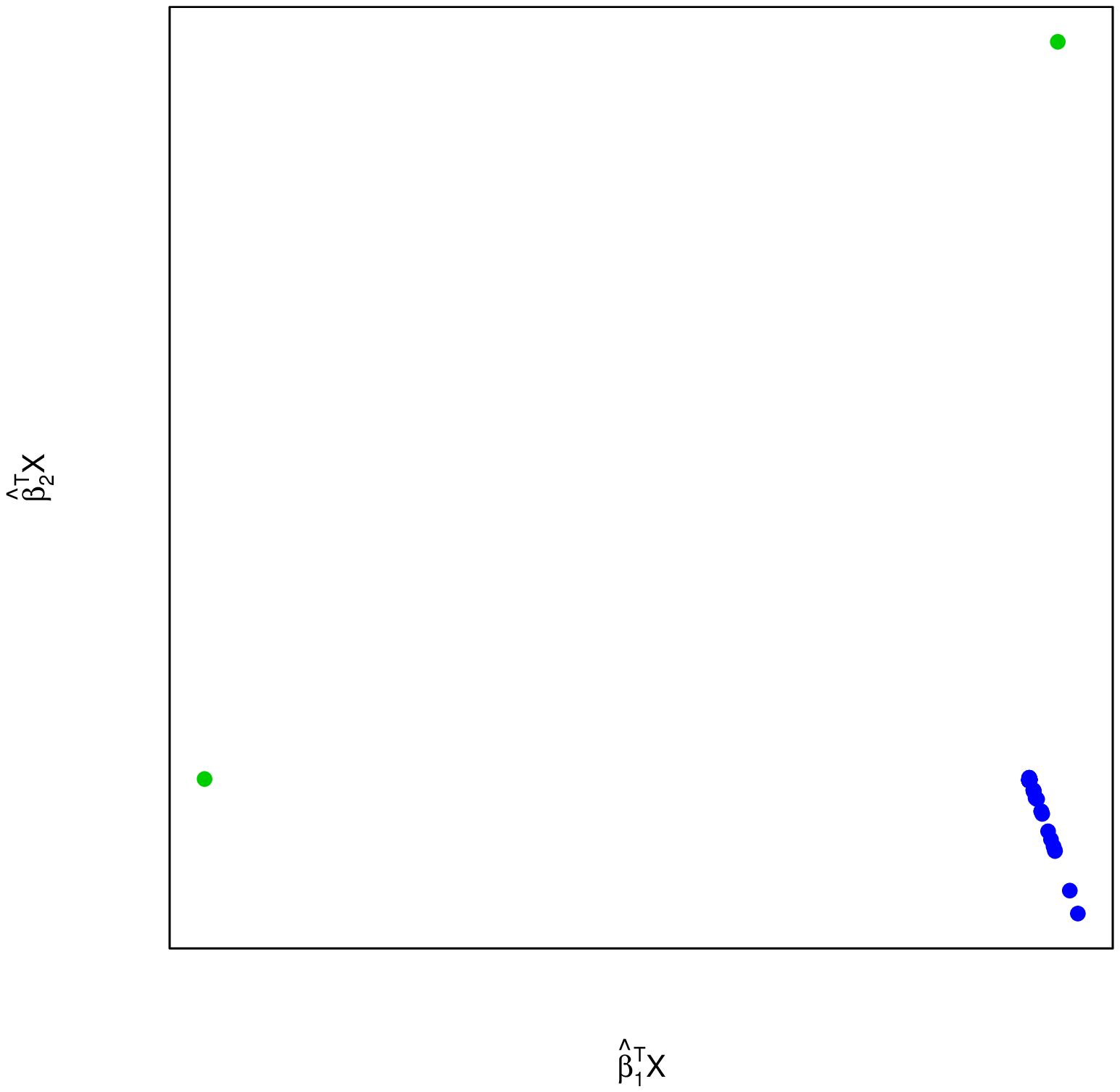}
\end{minipage}
}
\centering
\subfigure{
\begin{minipage}[c]{0.3\textwidth}
\centering
\includegraphics[height=2in,width=1.8in]{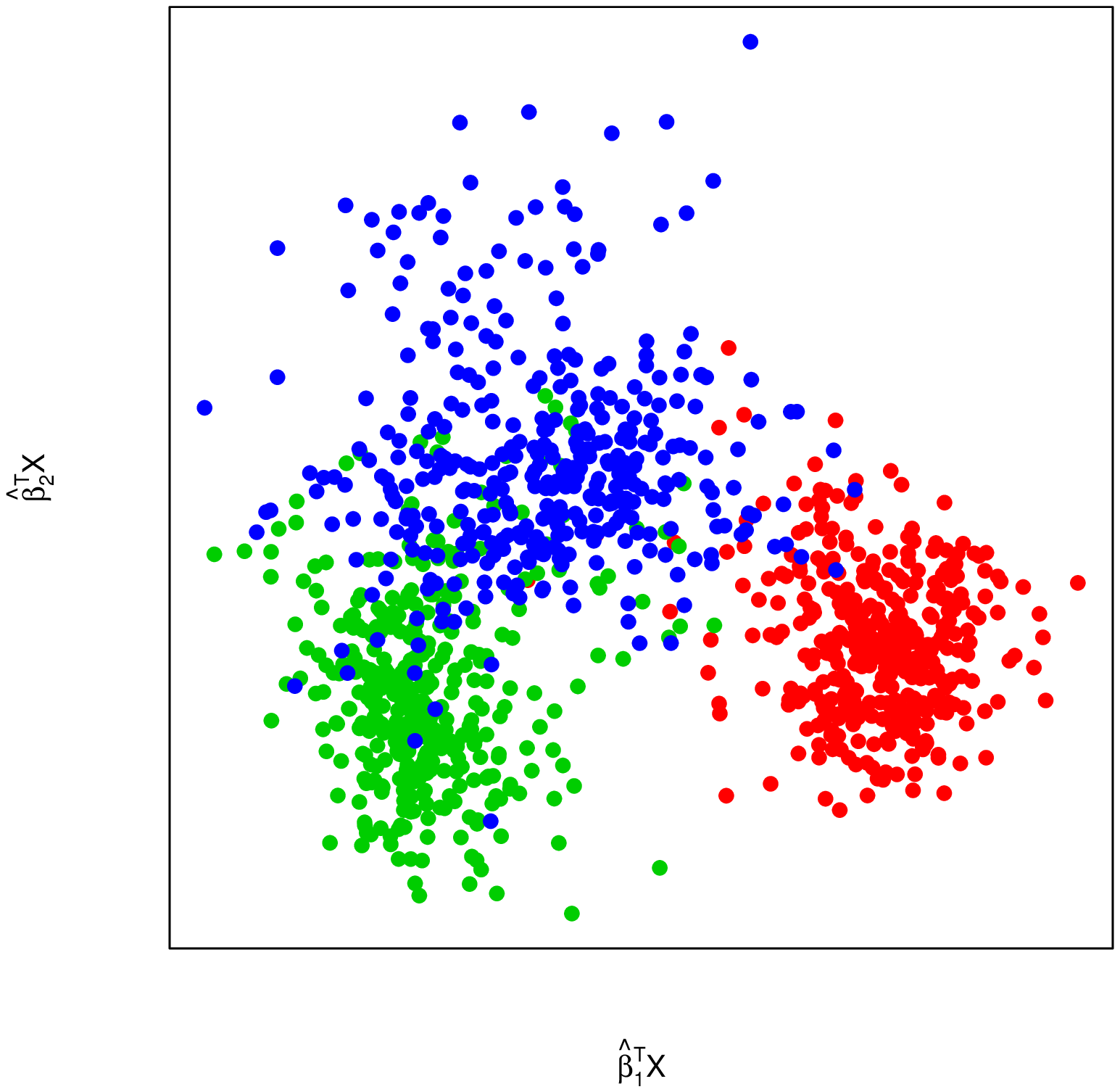}
\end{minipage}
\begin{minipage}[c]{0.3\textwidth}
\centering
\includegraphics[height=2in,width=1.8in]{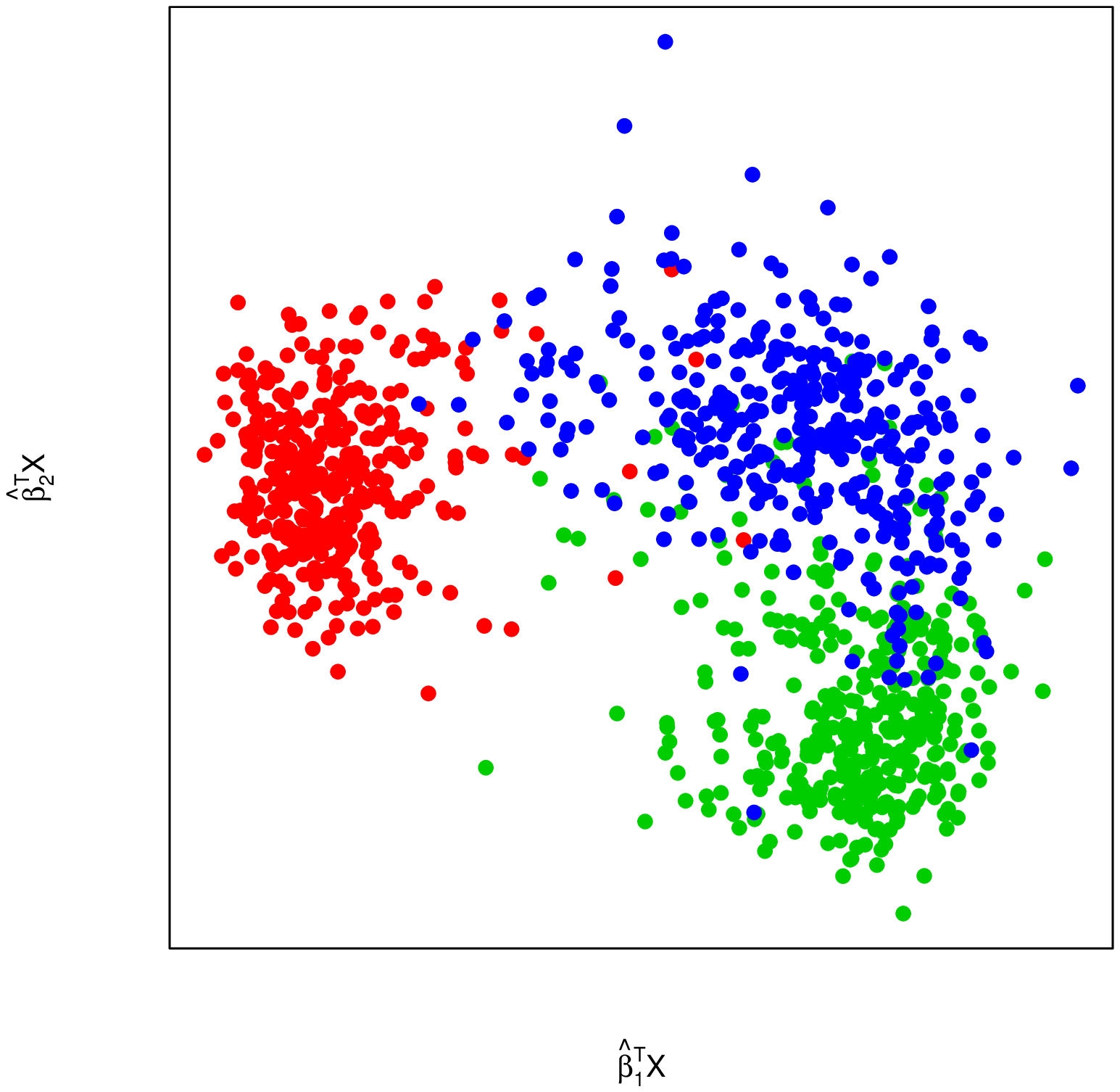}
\end{minipage}
\begin{minipage}[c]{0.3\textwidth}
\centering
\includegraphics[height=2in,width=1.8in]{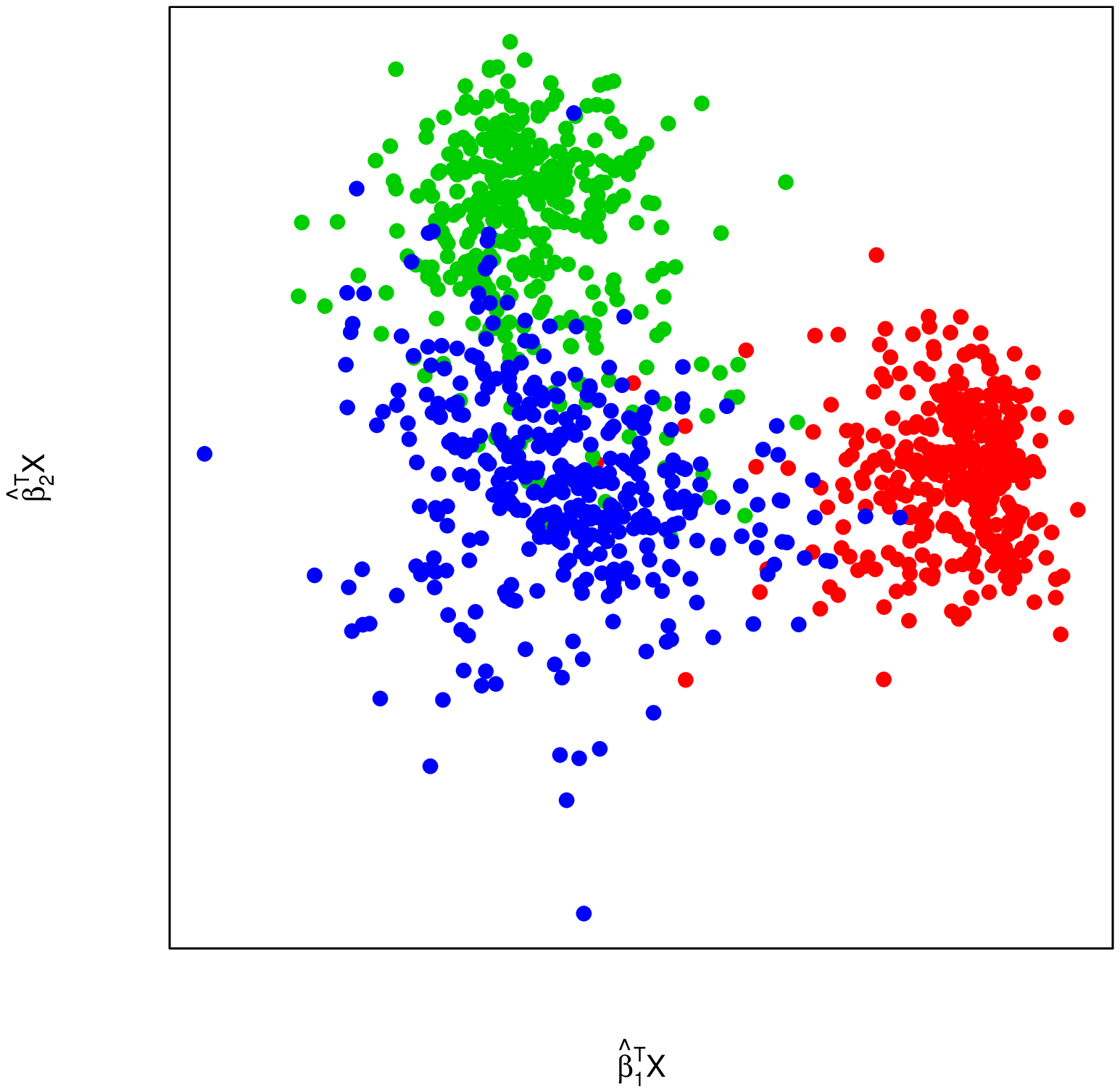}
\end{minipage}
}
\centering
\subfigure{
\centering
\begin{minipage}[c]{0.3\textwidth}
\centering
\includegraphics[height=2in,width=1.8in]{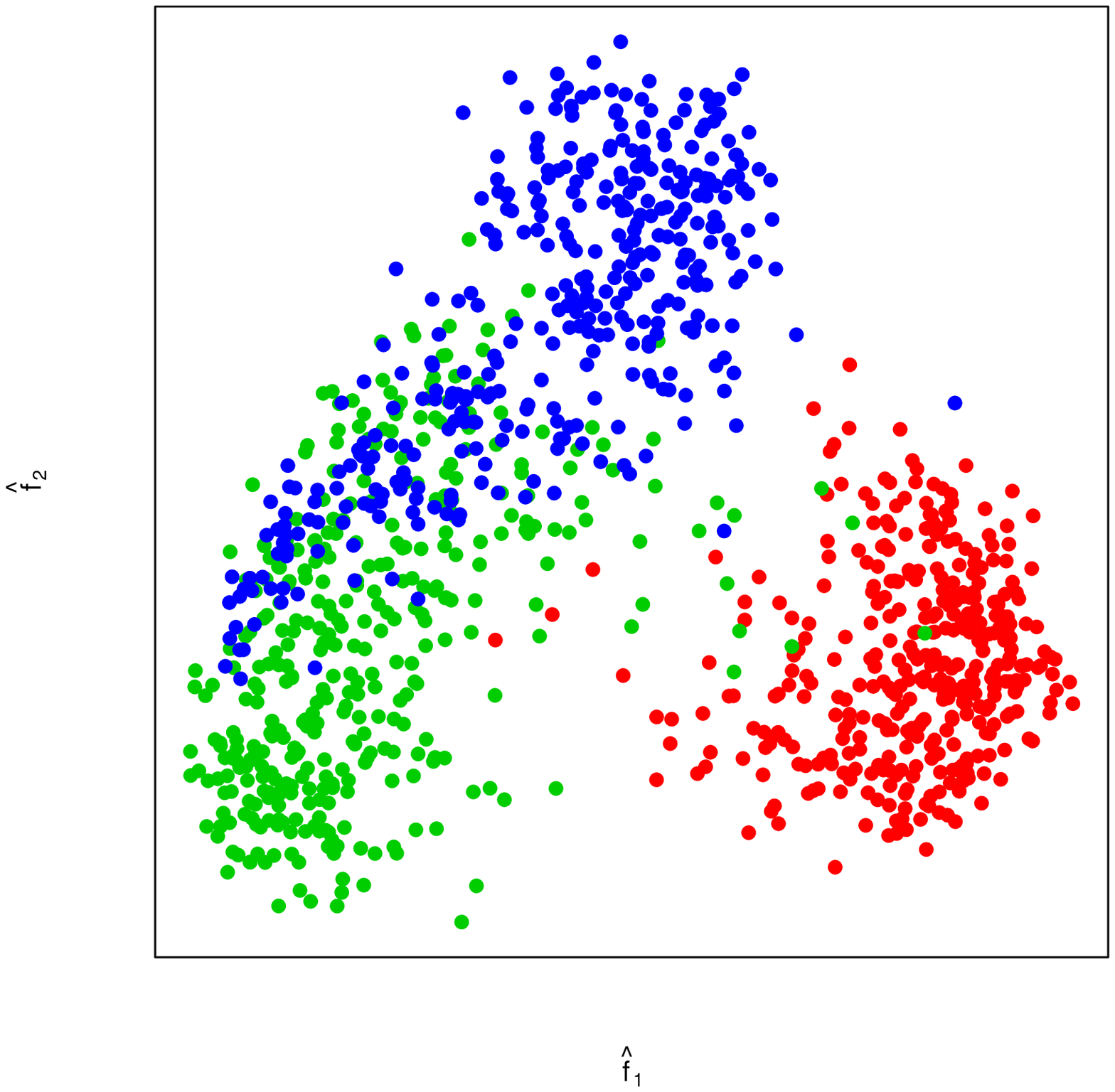}
\end{minipage}
\begin{minipage}[c]{0.3\textwidth}
\centering
\includegraphics[height=2in,width=1.8in]{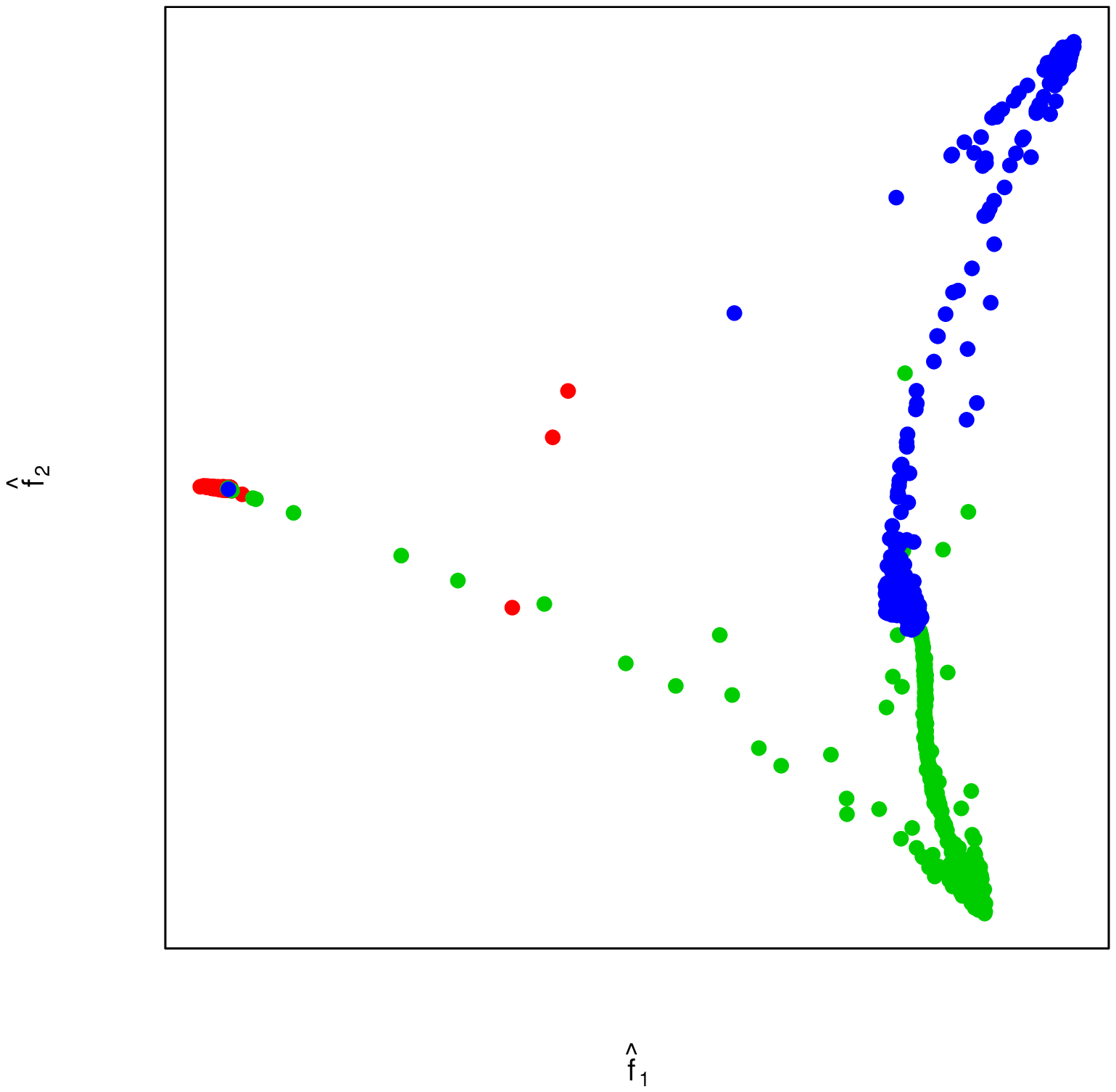}
\end{minipage}
\begin{minipage}[c]{0.3\textwidth}
\centering
\includegraphics[height=2in,width=1.8in]{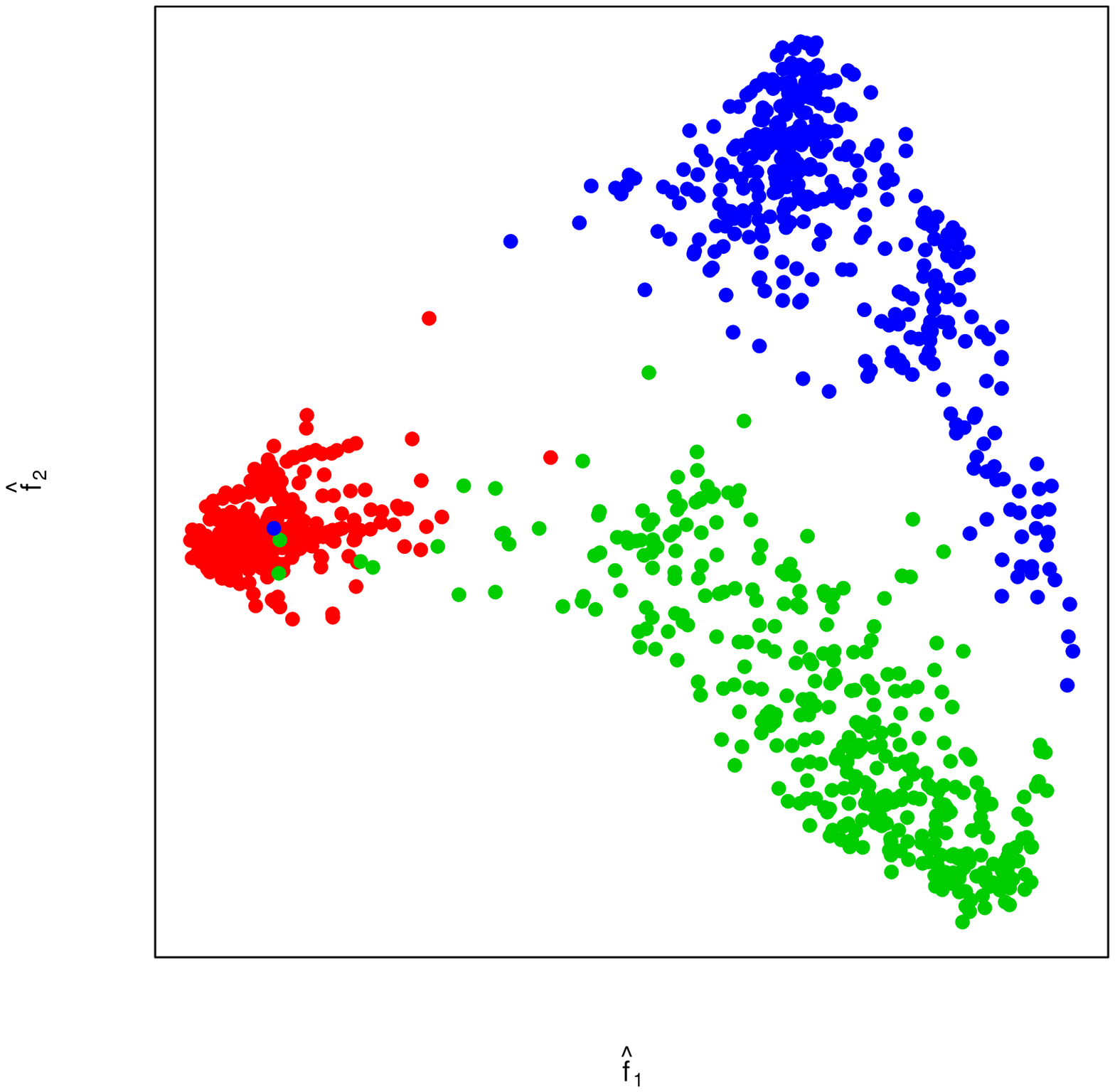}
\end{minipage}
}
\caption{The first row consists of the scatter plots for the training data via projective resampling based Slice Inverse Regression(SIR), Slice Average Variance Estimation(SAVE), Directional Regression(DR), respectively. The second row consists of the scatter plots based on our linear proposal with Euclidean distance, Locally Linear Embedding (LLE) and Isomap, respectively. The third row consists of the scatter plots based on our nonlinear proposal with Euclidean distance, LLE and Isomap, respectively. (red: $0$; green: $8$;blue: $9$.)}
\end{figure}

\begin{figure}
\centering
\setlength{\abovecaptionskip}{0.cm}
\subfigure{
\begin{minipage}[c]{0.3\textwidth}
\centering
\includegraphics[height=2in,width=1.8in]{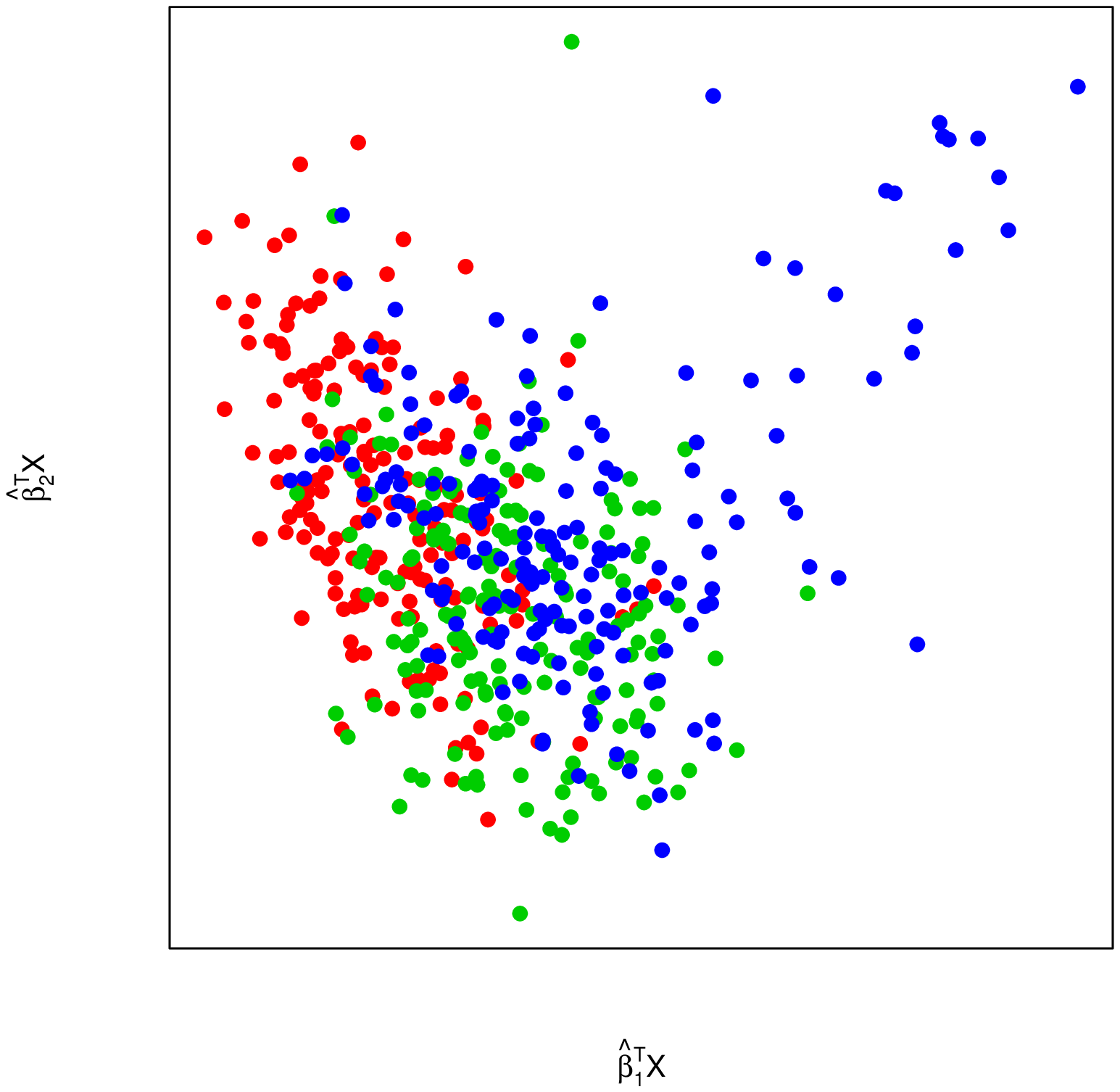}
\end{minipage}
\begin{minipage}[c]{0.3\textwidth}
\centering
\includegraphics[height=2in,width=1.8in]{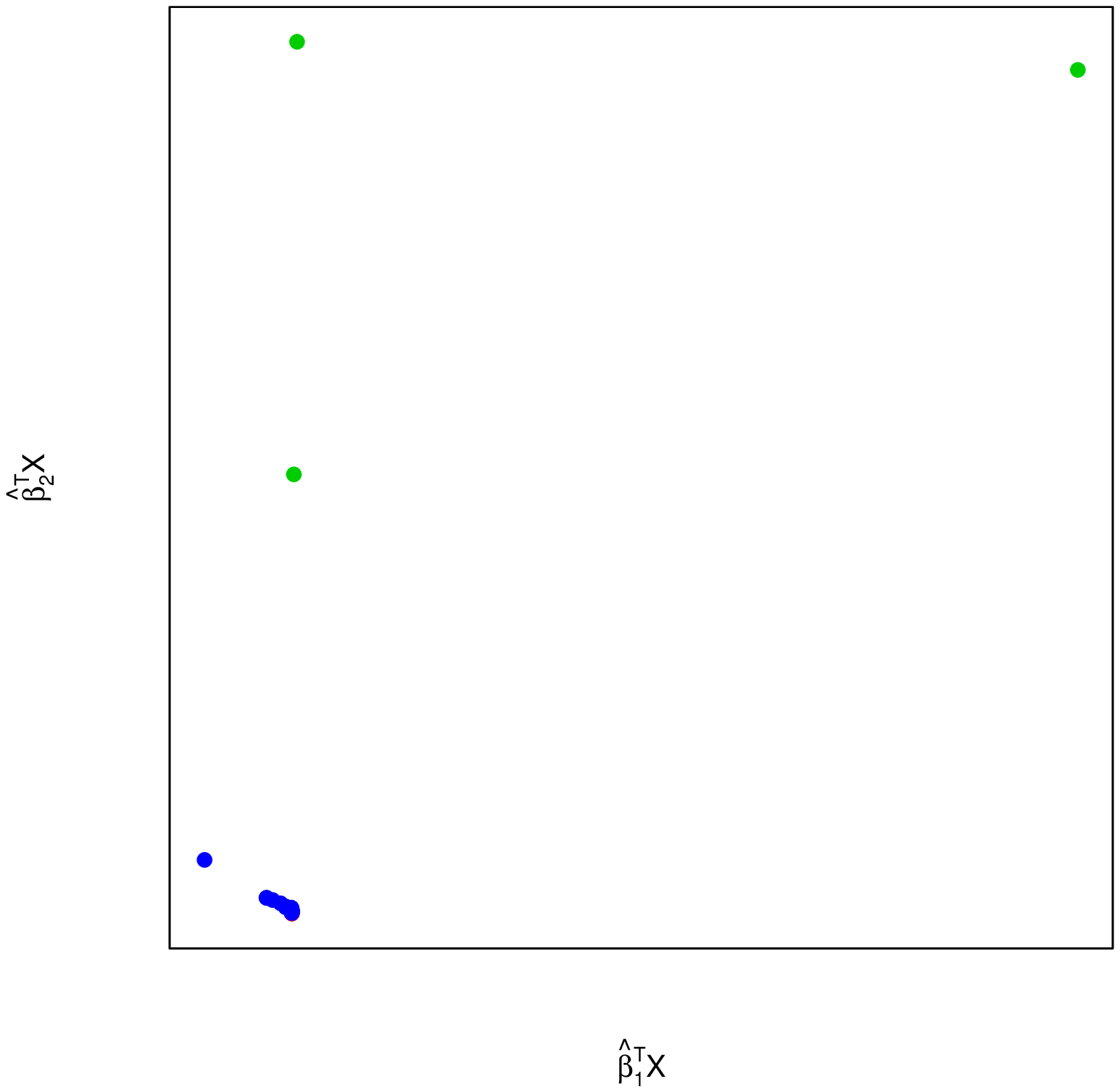}
\end{minipage}
\begin{minipage}[c]{0.3\textwidth}
\centering
\includegraphics[height=2in,width=1.8in]{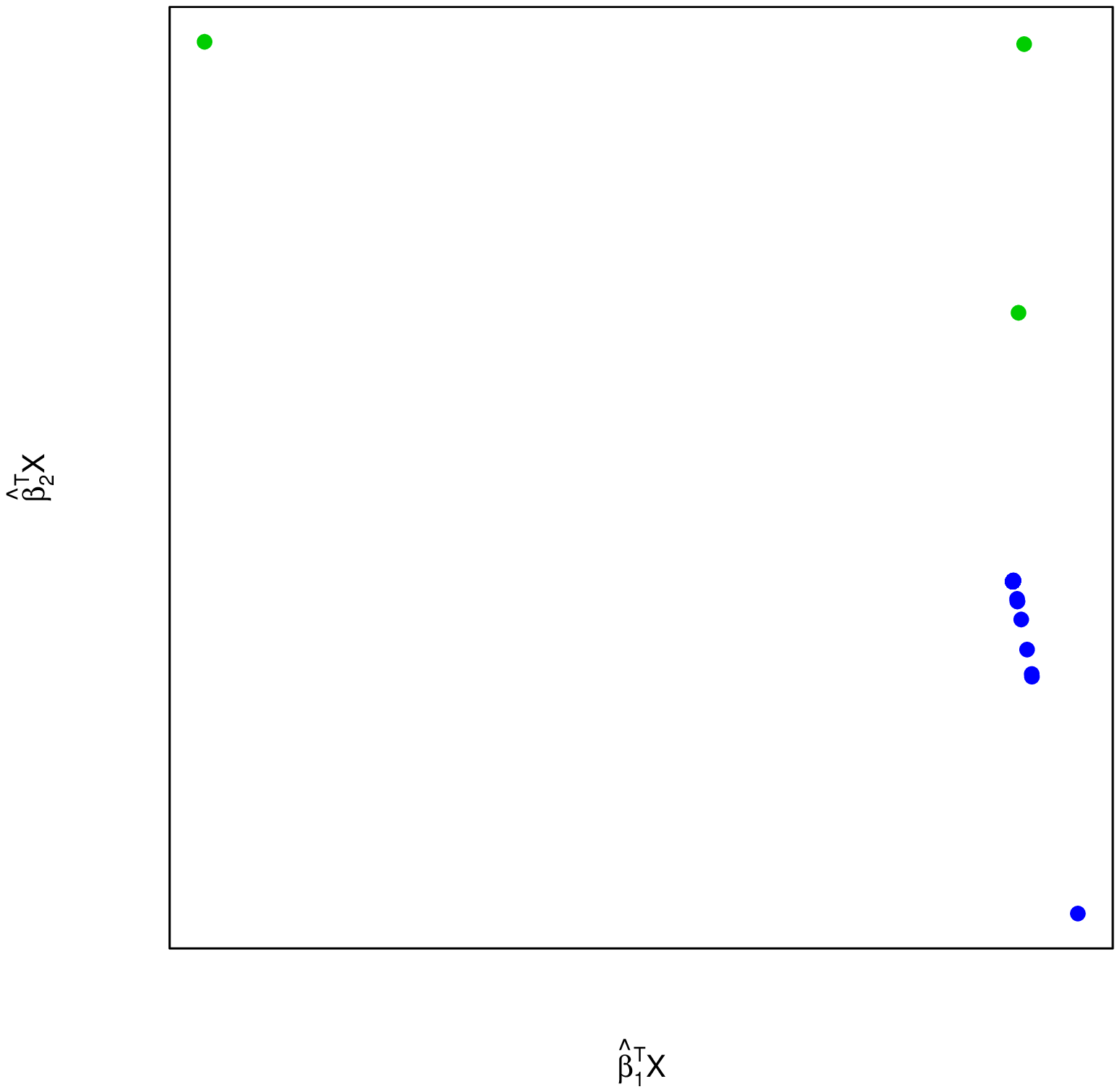}
\end{minipage}
}
\centering
\subfigure{
\begin{minipage}[c]{0.3\textwidth}
\centering
\includegraphics[height=2in,width=1.8in]{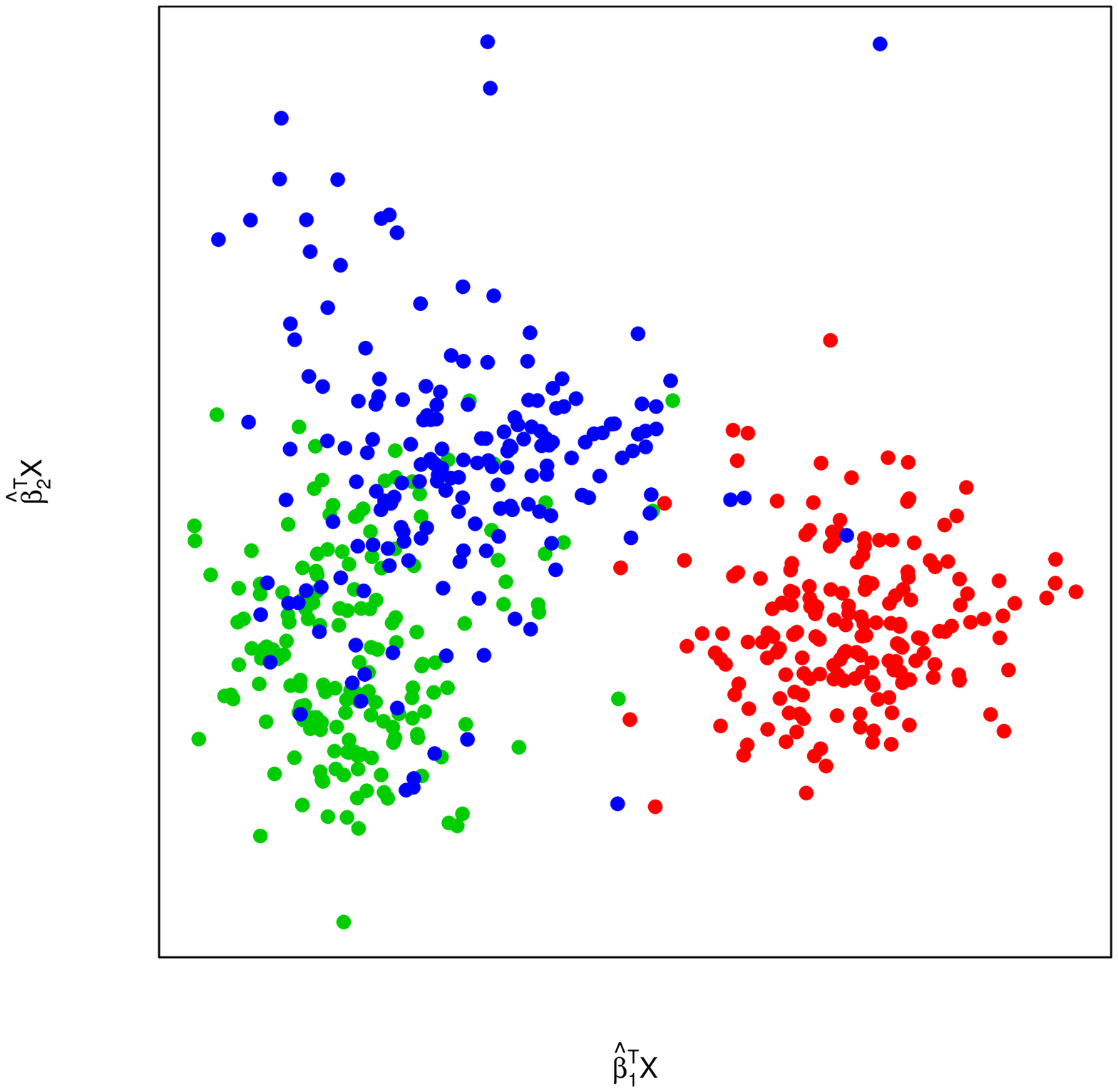}
\end{minipage}
\begin{minipage}[c]{0.3\textwidth}
\centering
\includegraphics[height=2in,width=1.8in]{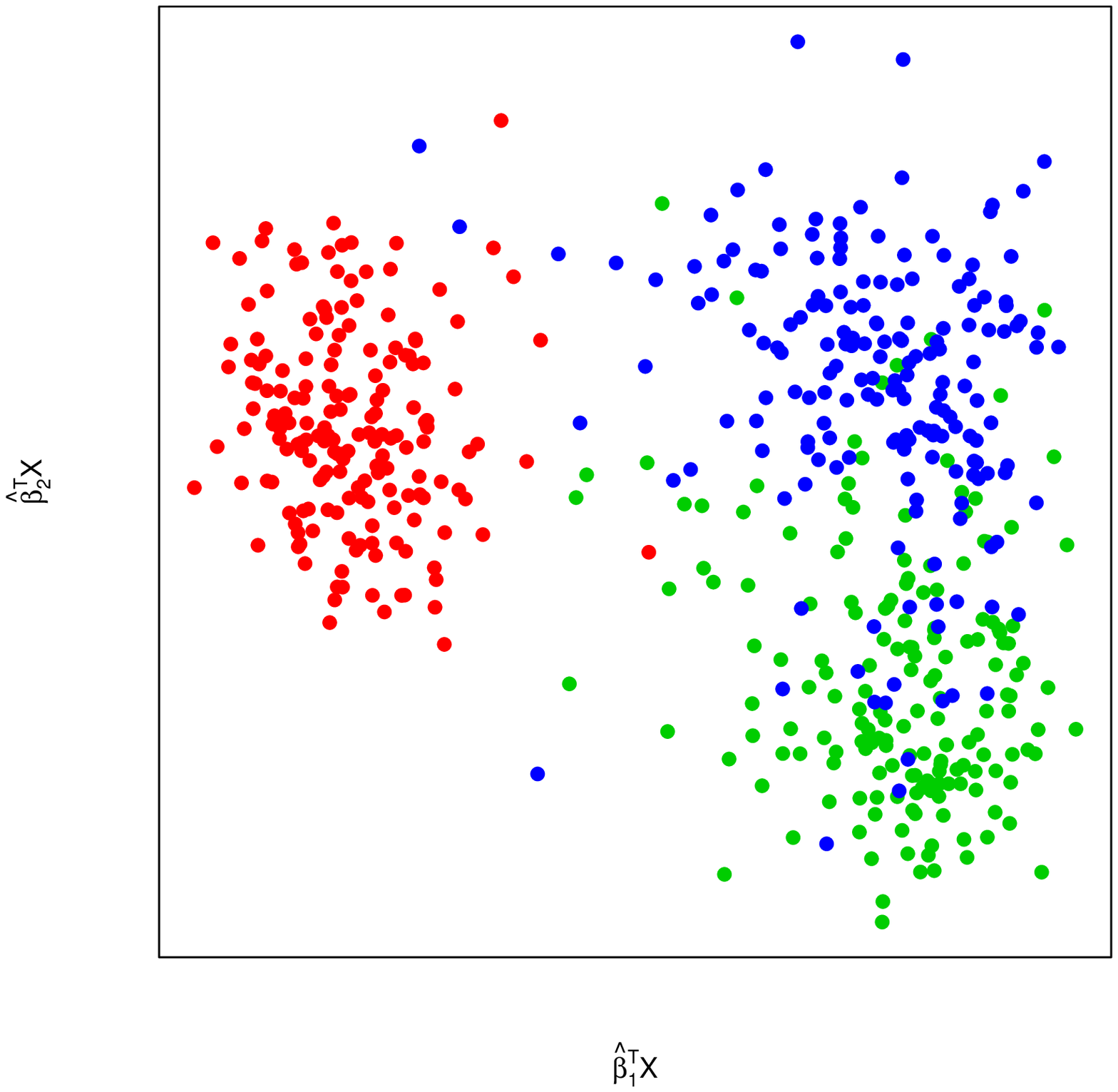}
\end{minipage}
\begin{minipage}[c]{0.3\textwidth}
\centering
\includegraphics[height=2in,width=1.8in]{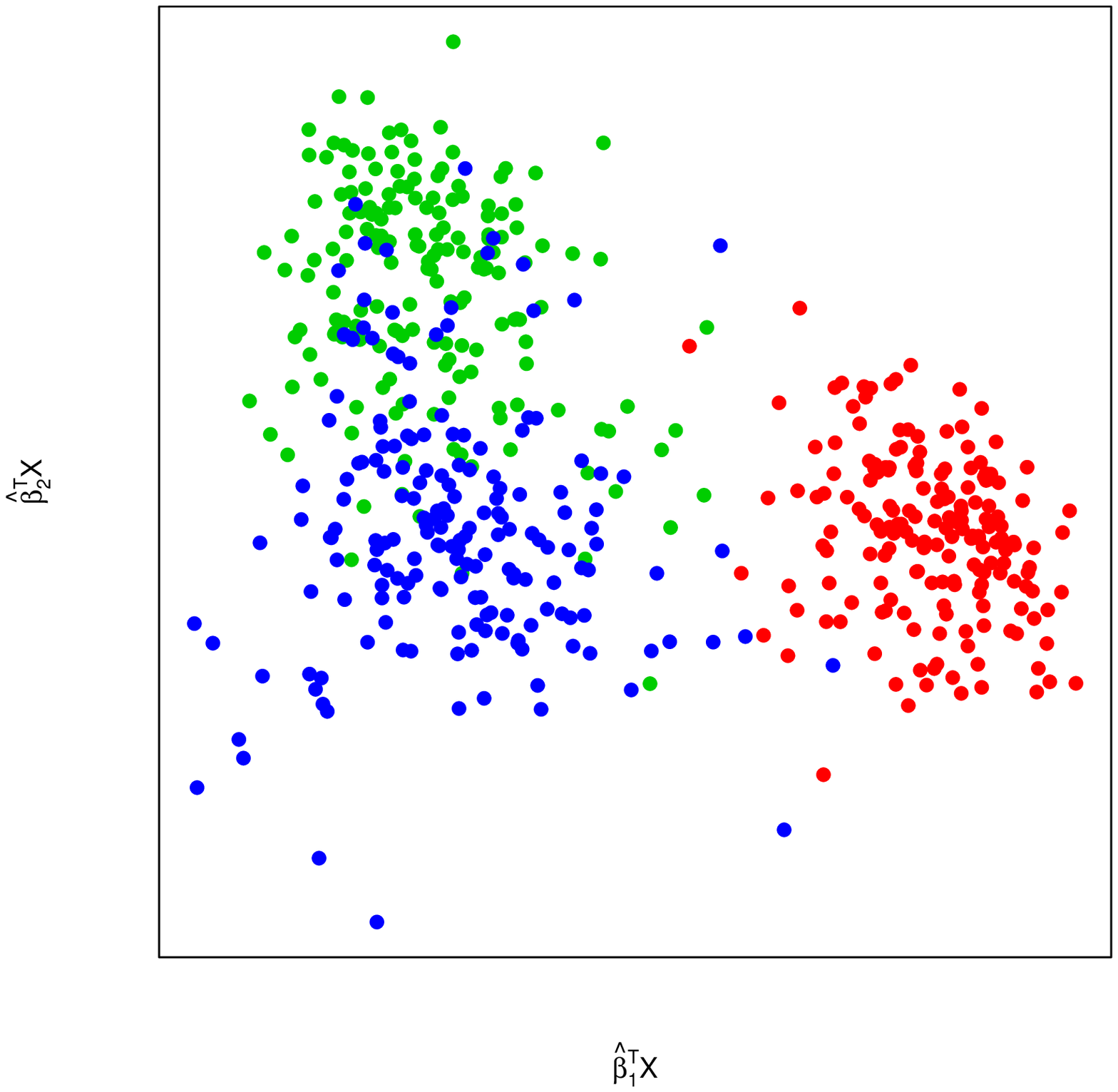}
\end{minipage}
}
\centering
\subfigure{
\begin{minipage}[c]{0.3\textwidth}
\centering
\includegraphics[height=2in,width=1.8in]{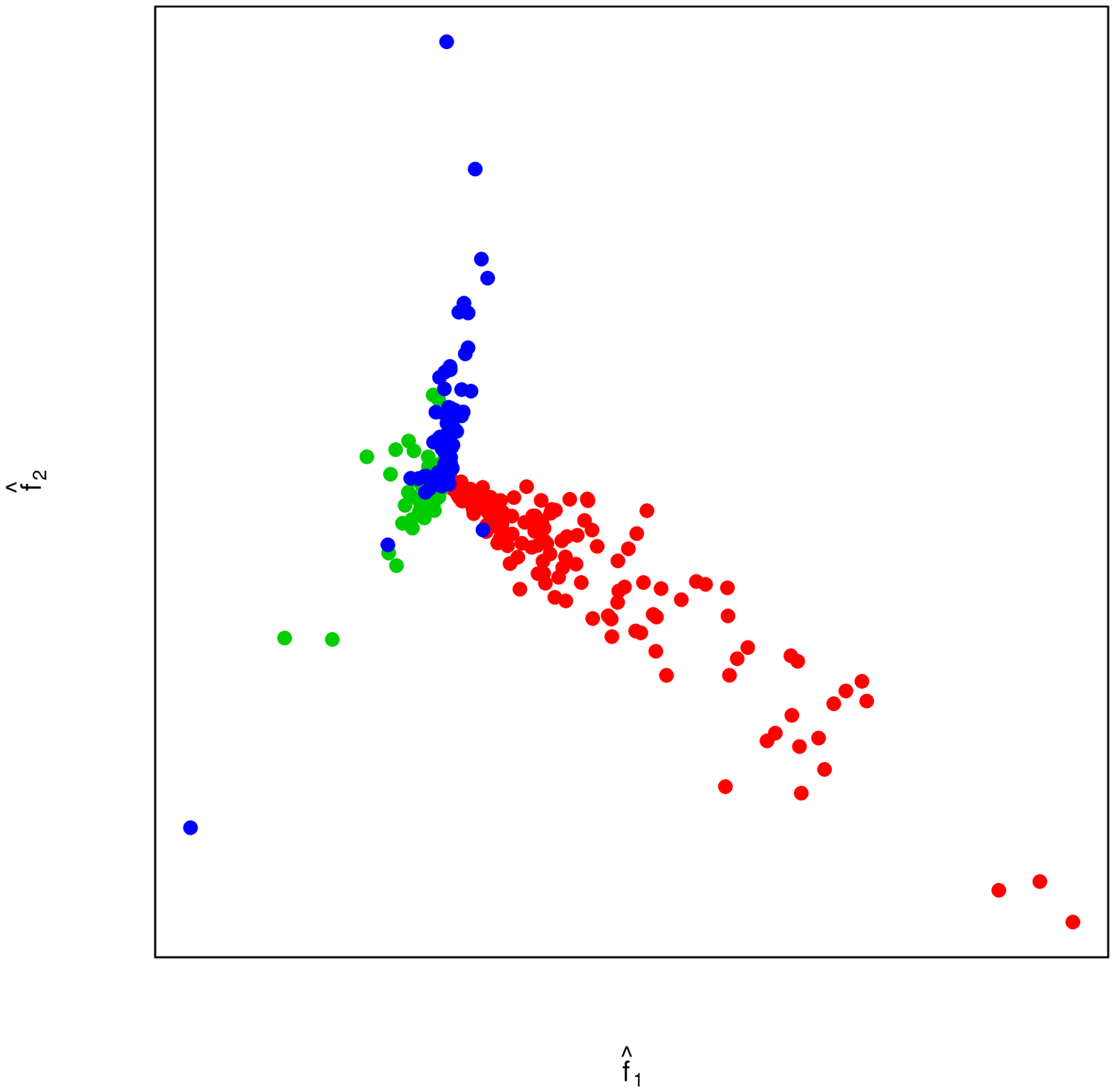}
\end{minipage}
\begin{minipage}[c]{0.3\textwidth}
\centering
\includegraphics[height=2in,width=1.8in]{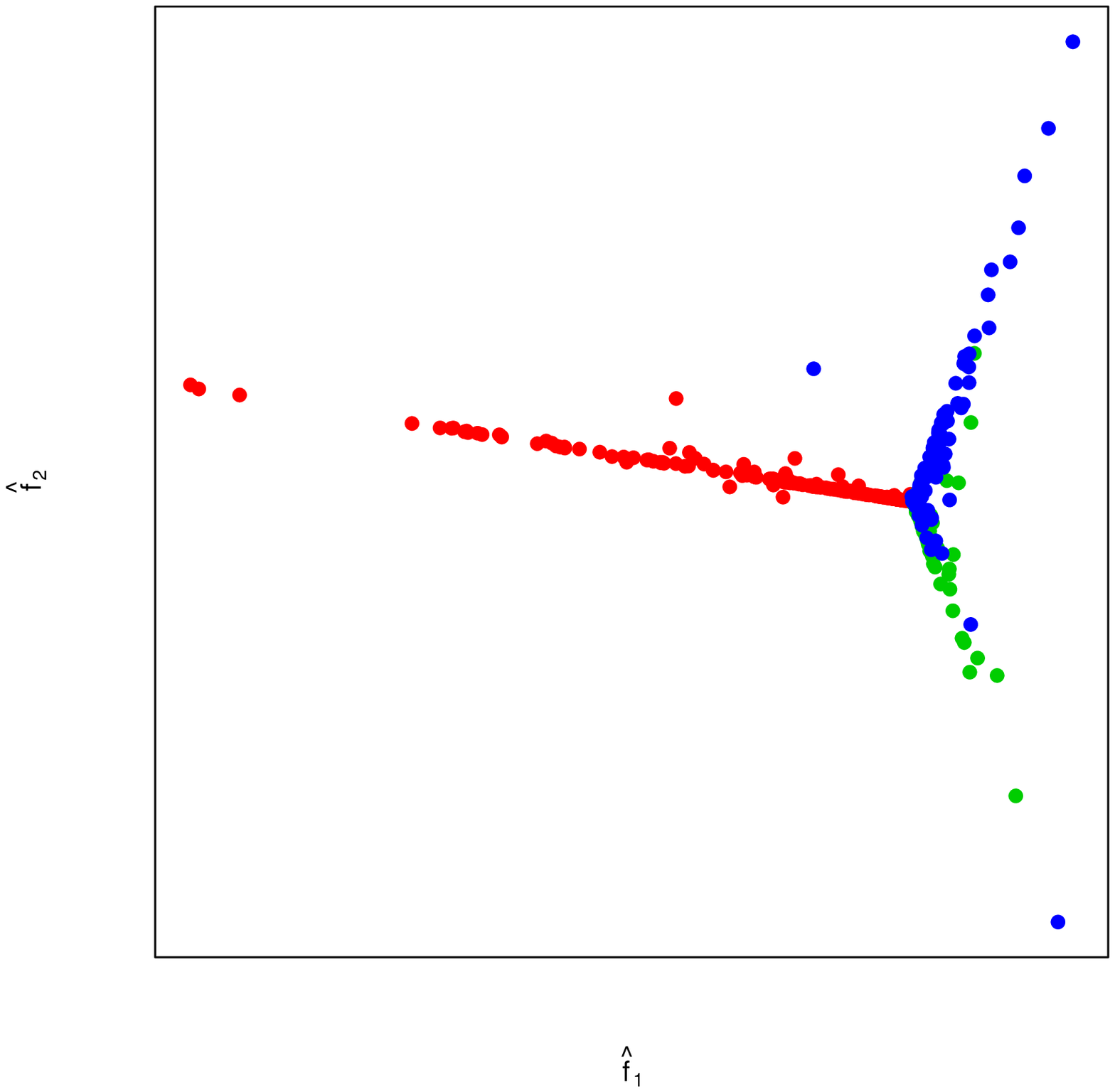}
\end{minipage}
\begin{minipage}[c]{0.3\textwidth}
\centering
\includegraphics[height=2in,width=1.8in]{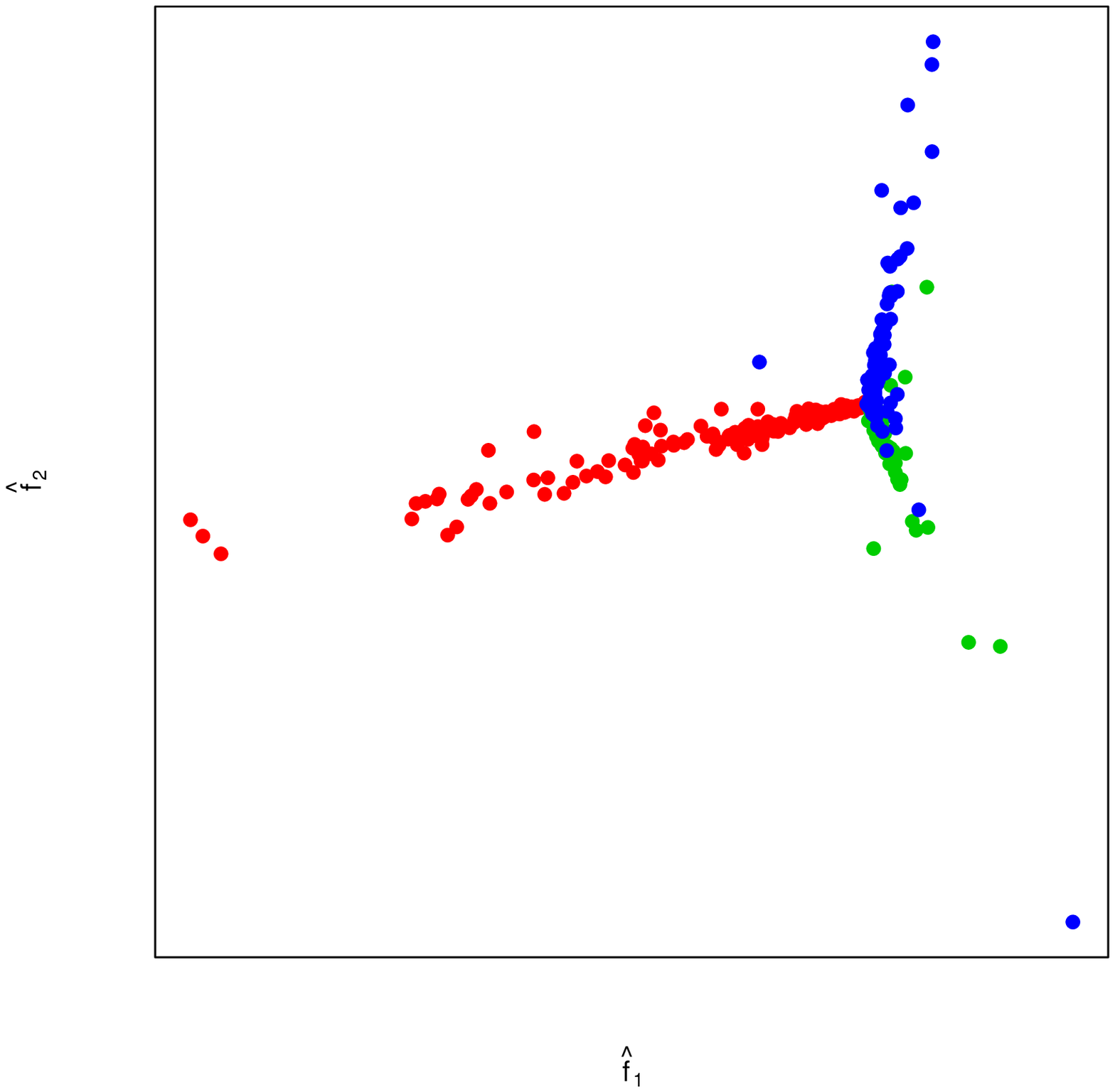}
\end{minipage}
}
\caption{The first row consists of the perspective plots for the first two sufficient predictors for the testing data by projective resampling based SIR, SAVE and DR, respectively. The second row consists of the perspective plots for the testing data based on our linear proposal with Euclidean distance, LLE and Isomap, respectively. The third row consists of the perspective plots for our nonlinear proposal with Euclidean distance, LLE and Isomap. (red: $0$; green: $8$;blue: $9$.)}
\end{figure}

To further investigate the performance of our proposals and demonstrate its use in real applications, we now extract $1670$ gray-scale images of three handwritten digit classes $\{0,8,9\}$, from the UCI Machine Learning Repository. This dataset contains a training group of size $1138$ and a testing group of size $532$. Each digit was represented by an $8\times 8$ pixel image. The $4\times 8$ upper part of each image was taken as the predictors $X$, and the $4\times 8$ bottom half was set as the responses $Y$.

We focus on sufficient dimension reduction of the $32$-dimensional
feature vectors $X$ for the training set, which serves as a preparatory step for further clustering or classification. Because the response $Y$ is also $32$-dimensional, we include the projective resampling approach (\cite{Li2008}) for comparisons, as it is the state of the art sufficient dimension reduction paradigm for multivariate response data. To be specific, we consider projective resampling approach in combination with three classical methods; sliced inverse regression, sliced average variance estimation and directional regression. And we adopt three different distance metrics for our proposals: the Euclidean distance, the distance metric learned by the Local Linear Embedding ({\cite{Roweis2000}}), the distance metric learned by the Isomap approach ({\cite{Tenenbaum2000}}).

For the training data, Figure 4 presents the scatter plots of the first two sufficient predictors estimated by projective resampling based three classical methods, as well as our proposed linear and nonlinear weighted inverse regression ensemble methods, with the cases for digits 0, 8 and 9 represented by red, green and blue dots respectively. Figure 4 shows that the linear weighted inverse regression ensemble method performs much better than the classical methods. We also observe that our nonlinear weighted inverse regression ensemble method based on the distance metric induced by the Isomap approach provides better separation both in location and variation, which should be useful for further classification.

Figure 5 presents the perspective plots for the testing data. Similar to our previous findings, our linear and nonlinear weighted inverse regression ensemble approaches again do a much better job in separating the three digit classes compared to the classical methods.  The upper parts of digits $8$ and $9$ are generally difficult to distinguish. However, our proposals provide valid and useful information for classification as seen from the scatter plots.

\section{Discussion}

When the predictor dimension is excessively large, we may consider sparse Fr\'echet dimension reduction with response as random objects and ultrahigh dimensional predictor. The proposed weighted inverse regression ensemble method can be further extended for model free variable selection and screening  (\cite{Yin2015, MSIR2016}) and minimax estimation of $\mathcal{S}_{Y|X}$ (\cite{MinimaxSIR}). The full potential of sparse Fr\'echet dimension reduction will be further explored in future research.

\clearpage
\appendix
\renewcommand{\theequation}{\thesection.\arabic{equation}}
\textbf{\Large{
\title{Supplement to ``Fr\'{e}chet Sufficient Dimension Reduction for Random Objects"}}}
\maketitle

\begin{abstract}
This Supplementary Material includes following topics: A. Additional results of simulation examples; B. Additional results for the application to the  handwritten digits data; C. Detailed proofs of the technical results.
\end{abstract}

\section{Additional Simulation Studies}

In addition to models I-III adopted in the main paper, we consider a new model here.

\vspace{0.2cm}
{Model IV.}
\[Y= (\cos(\varepsilon_1) \sin(f_1(X)) \sin(f_2(X)),\cos(\varepsilon_1)\sin(f_1(X))\cos(f_2(X)),\cos(\varepsilon_1)\cos(f_1(X)),\sin(\varepsilon_1)).\]

For Model IV, the response is a 4-dimensional vector and the fourth dimension can be viewed as a noise term which does not contain any valid information. Moreover, $Y\in \Omega$ where $\Omega$ is the unit sphere in $\mathbb{R}^4$. And the metric equipped with $\Omega$ is the geodesic distance $d(Y,Y')=\arccos(Y^T Y')$. Generate $\varepsilon_1$ from $N(0,0.1^2)$. We consider case (i) where $f_1(X)=\beta_1^T X$ and $f_2(X)=\beta_2^T X$ with $\beta_1=(0.5,0.5,0,\ldots,0)^T$ and $\beta_2=(0,\ldots,0,0.5,0.5)^T$ and cased (ii) where $f_1(X)=0.5(x_1^2+x_2^2)^{1/2}$ and $f_1(X)=0.5(x_{p-1}^2+x_p^2)^{1/2}$ for both linear and nonlinear Fr\'echet sufficient dimension reduction.

\vspace{0.2cm}
We design the following scenarios for the predictors for models I-IV.
\vspace{0.2cm}

{Scenario 1.} $X$ is generated from the multivariate normal distribution $N(\alpha,I_p)$, where $\alpha\sim U[0,1]^p$. The results for the four models, including linear Fr\'echet sufficient dimension reduction, nonlinear Fr\'echet sufficient dimension reduction and order determination, are presented in Table A.1-A.3.

\vspace{0.2cm}
{Scenario 2.} $X$ is generated from the multivariate normal distribution $N(\alpha,\Sigma)$, where $\alpha\sim U[0,1]^p$ and $\Sigma=\{(\sigma_{ij})_{p\times p}:\sigma_{ij}=0.2^{|i-j|}\}$. The results of the four models under scenario 2 are summarized in Table A.4-A.6.

\vspace{0.2cm}
{Scenario 3.} $x_i$ is generated from the poisson distribution $P_\lambda$ with $\lambda=1$ for $i=1,\ldots,p$. $x_i$ and $x_j$ are independent  of each other. The results of the four models under scenario 3 are presented in Table A.7 -A.9.

\vspace{0.2cm}
{Scenario 4.} $x_i$ is generated from the exponential distribution distribution $\text{Exp}(\lambda)$ with $\lambda=1$ for $i=1,\ldots,p$. $x_i$ and $x_j$ are independent  of each other. The results of the four models under scenario 4 are presented in Table A.10-A.12.

\vspace{0.2cm}
We see from these tables our proposal gives quite promising results for Fr\'chet sufficient dimension reduction and order determination. When the weighted inverse regression ensemble method fail to work with symmetric regression function, its nonlinear extension always make a good remedy. Our proposed methods along with the asymptotic theories are  robust to different model settings, except for
order determination with case (ii) of Model I under scenario 2.

\begin{table}[ht]\label{ta1}
	\def~{\hphantom{0}}
	\caption{\it The means of $r^2$ and $\rho^2$  for estimating $\mathcal{S}_{Y|X}$ among repetitions with scenario 1.  }{%
		\begin{tabular}{cccccccccccc}
			\\	
			\multicolumn{6}{c}{Model I (case i)}&\multicolumn{6}{c}{Model I (case ii)}\\[6pt]
			$(p,n)$    &	100      &   200	 &	  300	     & 	400	    &	$(p,n)$   &	 100     &   200	  &   300	&  400	    \\
			10         &	0.912	 &   0.938   &   0.955   	 &  0.967	&	    10    &	0.869	 &   0.923    &  0.954  &  0.961	\\
			           &	0.891	 &   0.917   &   0.938  	 &  0.952	&	          &	0.872	 &   0.917    &  0.947  &  0.955	\\
			20         &	0.790	 &   0.876   &   0.910	     &  0.927	&	    20	  &	0.771	 &   0.869    &  0.911  &  0.919	\\
			           &	0.787	 &   0.851   &   0.884	     &  0.905	&	          &	0.798	 &   0.862    &  0.904  &  0.909   \\
			30	       &	0.707    &   0.817   &   0.858       &  0.888   &	    30    &	0.688    &   0.830    &  0.854  &  0.894   \\
			           &	0.731	 &   0.799   &   0.827	     &  0.861	&	          &	0.752	 &   0.839    &  0.849  &  0.884  \\
			[8pt]																		
			\multicolumn{6}{c}{Model II(case i)}&\multicolumn{6}{c}{Model II(case ii)}\\[6pt]
			$(p,n)$    &	100      &   200	 &	  300	     & 	400	    &	$(p,n)$   &	 100     &   200	  &   300	 &  400	  \\
			10         &	0.915	 &   0.963   &   0.974   	 &  0.982	&	    10    &	0.471	 &   0.510    &   0.529  &  0.531	\\
			           &	0.915	 &   0.953   &   0.966  	 &  0.974	&	          &	0.375	 &   0.392    &   0.421  &  0.392	\\
			20         &	0.827	 &   0.918   &   0.946	     &  0.961	&	    20	  &	0.387	 &   0.439    &   0.451  &  0.467   \\
			           &	0.866	 &   0.914   &   0.936	     &  0.951	&	          &	0.372	 &   0.361    &   0.338  &  0.381   \\
			30	       &	0.766    &   0.874   &   0.912       &  0.937   &	    30    &	0.314   & 	 0.382    &   0.413  &  0.424   \\
			           &	0.850	 &   0.886   &   0.909	     &  0.929	&	          & 0.328	 &   0.337    &   0.339  &  0.326   \\
			[8pt]		
			\multicolumn{6}{c}{Model III (case i)}&\multicolumn{6}{c}{Model III (case ii)}\\[6pt]
			$(p,n)$    &	100      &   200	 &	  300	     & 	400	    &	$(p,n)$   &	 100     &   200	   &   300	  &  400	  \\
			10         &	0.891	 &   0.946   &   0.971   	 &  0.976	&	    10    &	0.538 	 &   0.577     &   0.586  &  0.584 	\\
			           &	0.895	 &   0.941   &   0.966  	 &  0.973	&	          &	0.476 	 &   0.508     &   0.512  &  0.514 	\\
			20         &	0.810	 &   0.898   &   0.924	     &  0.941	&	    20	  &	0.401 	 &   0.464     &   0.463  &  0.493    \\
			           &	0.843	 &   0.901   &   0.919	     &  0.934	&	          &	0.453 	 &   0.466     &   0.445  &  0.465    \\
			30	       &	0.735    &   0.856   &   0.900       &  0.922   &	    30    &	0.333    & 	 0.391     &   0.427  &  0.444   \\
			           &	0.813	 &   0.870   &   0.902	     &  0.919	&	          & 0.424 	 &   0.426     &   0.446  &  0.438    \\
            [8pt] 	
			\multicolumn{6}{c}{Model IV (case i)}&\multicolumn{6}{c}{Model IV (case ii)}\\[6pt]
			$(p,n)$    &	100      &   200	 &	  300	     & 	400	    &	$(p,n)$   &	 100     &   200	  &   300	  &  400	  \\
			10         &	0.920	 &   0.949   &   0.966   	 &  0.978	&	    10    &	0.591	 &   0.607    &   0.622   &  0.627	\\
			           &	0.921	 &   0.945   &   0.962 	     &  0.974	&	          &	0.430	 &   0.433    &   0.451   &  0.448	\\
			20         &	0.803	 &   0.877   &   0.934	     &  0.944	&	    20	  &	0.431	 &   0.484    &   0.501   &  0.520   \\
			           &	0.832	 &   0.879   &   0.930	     &  0.938	&	          &	0.387	 &   0.379    &   0.373   &  0.388   \\
			30	       &	0.733    &   0.840   &   0.888       &  0.911   &	    30    &	0.358    & 	 0.419    &   0.449   &  0.454   \\
			           &	0.806	 &   0.856   &   0.890       &  0.907	&	          & 0.365	 &   0.356    &   0.371   &  0.366   \\
	\end{tabular}}
	\caption*{
		The average $r^2$ and $\rho^2$ are listed in the first and second rows for each $p$.
	}
\end{table}

% \begin{table}[h]\label{t2}
%	\def~{\hphantom{0}}
%	\tbl{\it The Averages of $\rho^2$ for the estimation of $\mathcal{G}_{Y|X}$ based on $100$ simulation runs.}{%
%		\begin{tabular}{cccccccccccc}
%			\\	
%			\multicolumn{6}{c}{Model II(case 2)}&\multicolumn{6}{c}{Model III(case 2)}\\[6pt]
%			$(p,n)$    &	100      &   200	 &	  300	     & 	400	    &   $(p,n)$   &	 100     &   200	  &   300    &  400	  \\
%			10         &	0.986	 &   0.957   &   0.957   	 &  0.957	&	    10    &	0.940	 &   0.940    &   0.942  &  0.941	  \\
%			20         &	0.954	 &   0.958   &   0.956	     &  0.956	&	    20	  &	0.941	 &   0.942    &   0.942  &  0.943   \\
%			30	       &	0.956    &   0.956   &   0.955       &  0.955   &	    30    &	0.940    & 	 0.942    &   0.943  &  0.941   \\
%			[8pt]	
%			\multicolumn{6}{c}{Model IV(case 1)}&\\[6pt]
%			$(p,n)$    &	100      &   200	  &	  300	     & 	400	      \\
%			10         &	0.843	 &   0.840    &   0.839   	 &  0.844	 \\
%			20         &	0.843	 &   0.835    &   0.837	     &  0.839	  \\
%			30	       &	0.840    &   0.838    &   0.834      &  0.838   \\
%			[8pt]																	
%	\end{tabular}}
%\end{table}

 \begin{table}[ht]\label{t3}
	\def~{\hphantom{0}}
	\caption{\it The means of $\rho^2$ for estimating $\mathcal{G}_{Y|X}$ among $100$ repetitions with scenario 1. }{%
		\begin{tabular}{ccccccccccccc}
			\\	
			&\multicolumn{4}{c}{Model II (case ii)}&\multicolumn{4}{c}{Model III (case ii)}&\multicolumn{4}{c}{Model IV (case ii)}\\[6pt]
			$(p,n)$ &  100  &  200  &  300  &  400  &  100  &  200  &  300  &  400  &  100  &  200  &  300  &  400\\
			  10    & 0.986	& 0.957 & 0.957 & 0.957 & 0.843 & 0.840 & 0.839 & 0.844 & 0.940 & 0.940 & 0.942 & 0.941 \\
			  20    & 0.954 & 0.958 & 0.956 & 0.956 & 0.843 & 0.835 & 0.837 & 0.839 & 0.941 & 0.942 & 0.942 & 0.943 \\
			  30    & 0.956 & 0.956 & 0.955 & 0.955 & 0.840 & 0.838 & 0.834 & 0.838 & 0.940 & 0.942 & 0.943 & 0.941 \\																\end{tabular}}
\end{table}

\begin{table}[h]\label{t4}
	\def~{\hphantom{0}}
	\caption{\it The number of correctly estimation for $d$ among $100$ repetitions with scenario 1. }{%
		\begin{tabular}{ccccccccccccccc}
			\\	
			\multicolumn{5}{c}{Model I (case i)}&\multicolumn{5}{c}{Model I (case ii)}&\multicolumn{5}{c}{Model II (case i)}\\[6pt]
			$(p,n)$    &	100      &   200	  &	  300	     & 	400	    &	$(p,n)$   &	 100   &   200	  &   300	    &  400	 &$(p,n)$    &	100      &   200	  &	  300	     & 	400   \\
			10         &	100  	 &   100      &   100   	 &  100	    &	    10    &	41	   &   59     &   63   	    &  70   &10         &	85  	 &   96       &   99   	     &  99	\\
			20         &	100	     &   100      &   100	     &  100	    &	    20	  &	22	   &   38     &   40	    &  53   &20         &	78  	 &   96       &   100   &  100 \\
			30	       &	100      &   100      &   100        &  100     &	    30    &	18     &   31     &   36        &  43   &30	       &	23  	 &   70       &   94    &  100  \\
			[6pt]																		
			\multicolumn{5}{c}{Model III (case i)}&\multicolumn{5}{c}{Model IV (case i)}&\\[6pt]
			$(p,n)$   &	 100   &   200	   &   300	    &  400	    &$(p,n)$    &	100    &   200	   &   300	    & 	400	 \\
                10    &	38	   &   45      &   48   	&  60       & 10        &	 74    & 	90     &   87       &  90       \\
                20    &	10	   &   21      &   29	    &  35       &20	        &	 49    & 	72     &   82       &  83   	\\
                30	  &	2      &   16      &   27       &  32       &30         &	 37    & 	65     &   70       &  67     	\\
\end{tabular}}
\end{table}

\begin{table}[ht]\label{ta5}
	\def~{\hphantom{0}}
	\caption{\it The means of $r^2$ and $\rho^2$  for estimating $\mathcal{S}_{Y|X}$ among $100$ repetitions with scenario 2.  }{%
		\begin{tabular}{cccccccccccc}
			\\	
			\multicolumn{6}{c}{Model I (case i)}&\multicolumn{6}{c}{Model I (case ii)}\\[6pt]
			$(p,n)$    &	100      &   200	 &	  300	     & 	400	    &	$(p,n)$   &	 100     &   200	  &   300	&  400	    \\
			10         &	0.854	 &   0.913   &   0.933   	 &  0.947	&	    10    &	0.840	 &   0.880    &  0.931  &  0.937	\\
			           &	0.862	 &   0.909   &   0.929  	 &  0.943	&	          &	0.872	 &   0.895    &  0.939  &  0.945	\\
			20         &	0.720	 &   0.823   &   0.848	     &  0.881	&	    20	  &	0.711	 &   0.781    &  0.846  &  0.881	\\
			           &	0.764	 &   0.830   &   0.848	     &  0.878	&	          &	0.791	 &   0.820    &  0.866  &  0.896   \\
			30	       &	0.617    &   0.766   &   0.795       &  0.837   &	    30    &	0.584    &   0.767    &  0.794  &  0.836   \\
			           &	0.689	 &   0.783   &   0.807	     &  0.839	&	          &	0.716	 &   0.816    &  0.826  &  0.856  \\
			[8pt]																		
			\multicolumn{6}{c}{Model II (case i)}&\multicolumn{6}{c}{Model II (case ii)}\\[6pt]
			$(p,n)$    &	100      &   200	 &	  300	     & 	400	    &	$(p,n)$   &	 100     &   200	  &   300	&  400	  \\
			10         &	0.853	 &   0.920   &   0.948   	 &  0.963	&	    10    &	0.249	 &   0.258    &   0.247 &  0.250	\\
			           &	0.888	 &   0.930   &   0.947  	 &  0.962	&	          &	0.131	 &   0.101    &   0.084 &  0.088	\\
			20         &	0.707	 &   0.853   &   0.900	     &  0.921	&	    20	  &	0.128	 &   0.141    &   0.121 &  0.158   \\
			           &	0.814	 &   0.881   &   0.908	     &  0.925	&	          &	0.097	 &   0.059    &   0.042 &  0.051   \\
			30	       &	0.603    &   0.753   &   0.842       &  0.883   &	    30    &	0.086    & 	 0.086    &   0.085 &  0.099   \\
			           &	0.778	 &   0.827   &   0.873   	 &  0.897	&	          & 0.100	 &   0.049    &   0.032 &  0.029   \\
			[8pt]	
			\multicolumn{6}{c}{Model III(case i)}&\multicolumn{6}{c}{Model III (case ii)}\\[6pt]
			$(p,n)$    &	100      &   200	  &	  300	    & 	400	    &	$(p,n)$   &	 100     &   200	  &   300	&  400	    \\
			10         &	0.864	 &   0.916    &  0.943      &  0.957	&	    10    &	0.365	 &   0.383    &  0.373  &  0.353\\
			           &	0.896	 &   0.930    &  0.950      &  0.961	&	          &	0.243	 &   0.192    &  0.173  &  0.154\\
			20         &	0.749	 &   0.830    &  0.898      &  0.904	&	    20    &	0.181	 &   0.189    &  0.184  &  0.182\\
			           &	0.819	 &   0.862    &  0.915      &  0.918    &	          &	0.160	 &   0.106    &  0.078  &  0.069\\
			30	       &	0.687    &   0.794    &  0.842      &  0.862    &	    30    &	0.114	 &   0.123    &  0.118  &  0.122\\
			           &	0.799	 &   0.840    &  0.873      &  0.886    &	          &	0.144	 &   0.086    &  0.063  &  0.050\\					
			[8pt]																		
			\multicolumn{6}{c}{Model IV (case i)}&\multicolumn{6}{c}{Model IV (case ii)}\\[6pt]
			$(p,n)$    &	100      &   200	 &	  300	     & 	400	    &	$(p,n)$   &	 100     &   200	  &   300	&  400	  \\
			10         &	0.844	 &   0.915   &   0.946   	 &  0.959	&	    10    &	0.595	 &   0.613    &   0.631 &  0.644	\\
			           &	0.877	 &   0.930   &   0.953  	 &  0.964	&	          &	0.504	 &   0.478    &   0.510 &  0.526	\\
			20         &	0.731	 &   0.843   &   0.881	     &  0.916	&	    20	  &	0.424	 &   0.472    &   0.505 &  0.523   \\
			           &	0.808	 &   0.873   &   0.901	     &  0.927	&	          &	0.439	 &   0.435    &   0.444 &  0.473  \\
			30	       &	0.660    &   0.803   &   0.828       &  0.877   &	    30    &	0.339    & 	 0.412    &   0.445 &  0.462   \\
			           &	0.783	 &   0.853   &   0.864	     &  0.898	&	          & 0.409	 &   0.415    &   0.411 &  0.443   \\
\end{tabular}}
%	\begin{tabnote}
%		The average $r^2$ and $\rho^2$ are listed in the first and second row for each $p$.
%	\end{tabnote}
\end{table}

% \begin{table}[h]\label{t2}
%	\def~{\hphantom{0}}
%	\tbl{\it The Averages of $\rho^2$ for the estimation of $\mathcal{G}_{Y|X}$ based on $100$ simulation runs. }{%
%		\begin{tabular}{cccccccccccc}
%			\\	
%			\multicolumn{6}{c}{Model II(case 2)}&\multicolumn{6}{c}{Model III(case 2)}\\[6pt]
%			$(p,n)$    &	100      &   200	 &	  300	     & 	400	    &   $(p,n)$   &	 100     &   200	  &   300    &  400	  \\
%			10         &	0.954	 &   0.951   &   0.951   	 &  0.951	&	    10    &	0.933	 &   0.939    &   0.938  &  0.941	  \\
%			20         &	0.952	 &   0.951   &   0.952	     &  0.952	&	    20	  &	0.937	 &   0.939    &   0.941  &  0.939   \\
%			30	       &	0.950    &   0.950   &   0.951       &  0.952   &	    30    &	0.935    & 	 0.939    &   0.940  &  0.939   \\
%			[8pt]	
%			\multicolumn{6}{c}{Model IV(case 2)}&\\[6pt]
%			$(p,n)$    &	100      &   200	 &	  300	     & 	400	      \\
%			10         &	0.824	 &   0.820   &   0.825   	 &  0.826	 \\
%			20         &	0.817	 &   0.825   &   0.825	     &  0.823	 \\
%			30	       &	0.836    &   0.821   &   0.825       &  0.821  \\
%			[8pt]																
%	\end{tabular}}
%\end{table}
 \begin{table}[ht]\label{t6}
	\def~{\hphantom{0}}
	\caption{\it The means of $\rho^2$ for estimating $\mathcal{G}_{Y|X}$ among $100$ repetitions with scenario 2. }{%
		\begin{tabular}{ccccccccccccc}
			\\	
			&\multicolumn{4}{c}{Model II (case ii)}&\multicolumn{4}{c}{Model III (case ii)}&\multicolumn{4}{c}{Model IV (case ii)}\\[6pt]
			$(p,n)$    &	100      &   200	 &	  300	     & 	400	       &	 100     &   200	  &   300    &  400	 &	100      &   200	 &	  300	     & 	400	 \\
			10         &	0.954	 &   0.951   &   0.951   	 &  0.951	   &	0.824	 &   0.820    &   0.825  &  0.826&	0.933	 &   0.939    &   0.938  &  0.941	  \\
			20         &	0.952	 &   0.951   &   0.952	     &  0.952	   &	0.817	 &   0.825    &   0.825	 &  0.823&	0.937	 &   0.939    &   0.941  &  0.939   \\
			30	       &	0.950    &   0.950   &   0.951       &  0.952      &	0.836    &   0.821    &   0.825  &  0.821&	0.935    & 	 0.939    &   0.940  &  0.939   \\
\end{tabular}}
\end{table}

%\begin{table}[h]\label{t2}
%	\def~{\hphantom{0}}
%	\tbl{\it The number of correctly estimation for $d$ among $100$ simulation runs. }{%
%		\begin{tabular}{cccccccccccc}
%			\\	
%			\multicolumn{6}{c}{Model I(case 1)}&\multicolumn{6}{c}{Model I(case 2)}\\[6pt]
%			$(p,n)$    &	100      &   200	  &	  300	     & 	400	    &	$(p,n)$   &	 100   &   200	  &   300	    &  400	    \\
%			10         &	100  	 &   100      &   100   	 &  100	    &	    10    &	38	   &   38     &   41    	&  45   	\\
%			20         &	100	     &   100      &   100	     &  100	    &	    20	  &	12	   &   15     &   22	    &  35      \\
%			30	       &	100      &   100      &   100        &  100     &	    30    &	10     &   13     &   17        &  18      \\
%			[8pt]																		
%			\multicolumn{6}{c}{Model II(case 1)}&\multicolumn{3}{c}{Model III(case 1)}\\[6pt]
%			$(p,n)$    &	100      &   200	  &	  300	     & 	400	    &	$(p,n)$   &	 100   &   200	   &   300	    &  400	    \\
%			10         &	73  	 &   97       &   100  	     &  100  	&	    10    &	73	   &   80      &   81   	&  84   	\\
%			20         &	45  	 &   93       &   98         &  100  	&	    20	  &	46	   &   55      &   64   	&  68   	\\
%			30	       &	6  	     &   53       &   62         &  96      &	    30    &	31	   &   44      &   62   	&  64     	\\
%			[8pt]
%			\multicolumn{6}{c}{Model IV(case 1)}&\\[6pt]
%			$(p,n)$    &	100      &   200	  &	  300	     & 	400	      \\
%			10         &	68	 &   81   &   85   	 &  87	 \\
%			20         &	49	 &   60   &   66	 &  67	 \\
%			30	       &	43   &   53   &   58     &  59  \\
%\end{tabular}}
%\end{table}
\begin{table}[ht]\label{t7}
	\def~{\hphantom{0}}
	\caption{\it The number of correctly estimation for $d$ among $100$ repetitions with scenario 2.  }{%
		\begin{tabular}{ccccccccccccccc}
			\\	
			\multicolumn{5}{c}{Model I (case i)}&\multicolumn{5}{c}{Model I (case ii)}&\multicolumn{5}{c}{Model II (case i)}\\[6pt]
			$(p,n)$    &	100      &   200	  &	  300	     & 	400	    &	$(p,n)$   &	 100   &   200	  &   300	    &  400	&   $(p,n)$    &	100      &   200	  &	  300	     & 	400	 \\
			10         &	100  	 &   100      &   100   	 &  100	    &	    10    &	38	   &   38     &   41    	&  45   &	10         &	73  	 &   97       &   100  	     &  100  \\
			20         &	100	     &   100      &   100	     &  100	    &	    20	  &	12	   &   15     &   22	    &  35   &   20         &	45  	 &   93       &   98         &  100  \\
			30	       &	100      &   100      &   100        &  100     &	    30    &	10     &   13     &   17        &  18   &   30	       &	6  	     &   53       &   62         &  96    \\
			[3pt]																		
\multicolumn{5}{c}{Model III (case i)}& \multicolumn{5}{c}{Model IV (case i)}&\\[6pt]
			$(p,n)$   &	 100   &   200	   &   300	    &  400	&   $(p,n)$    &	100    &   200	  &	  300	 & 	400	 \\
			10        &	73	   &   80      &   81   	&  84   &    10        &	68	   &   81     &   85   	 &  87\\
			20        &	49	   &   60      &   66	    &  67   &    20	       &	46	   &   55     &   64     &  68   \\
			30	      &	43     &   53      &   58       &  59   &    30        &	31	   &   44     &   62   	 &  64   \\
\end{tabular}}
\end{table}

\begin{table}[ht]\label{ta8}
	\def~{\hphantom{0}}
	\caption{\it The means of $r^2$ and $\rho^2$  for estimating $\mathcal{S}_{Y|X}$ among $100$ repetitions with scenario 3. }{%
		\begin{tabular}{cccccccccccc}
			\\	
			\multicolumn{6}{c}{Model I (case i)}&\multicolumn{6}{c}{Model I (case ii)}\\[6pt]
			$(p,n)$    &	100      &   200	  &	  300	     & 	400	    &	$(p,n)$   &	 100     &   200	  &   300	&  400	    \\
			10         &	0.852	 &   0.903   &   0.918   	 &  0.927	&	    10    &	0.942	 &   0.952    &  0.961  &  0.972	\\
			           &	0.821	 &   0.874   &   0.889  	 &  0.899	&	          &	0.937	 &   0.942    &  0.953  &  0.965	\\
			20         &	0.756	 &   0.843   &   0.846	     &  0.874	&	    20	  &	0.869	 &   0.915    &  0.928  &  0.945	\\
			           &	0.734	 &   0.802   &   0.805	     &  0.834	&	          &	0.877	 &   0.908    &  0.922  &  0.935   \\
			30	       &	0.659    &   0.757   &   0.805       &  0.832   &	    30    &	0.846    &   0.880    &  0.898  &  0.910   \\
			           &	0.669	 &   0.720   &   0.758	     &  0.786	&	          &	0.870	 &   0.882    &  0.890  &  0.899  \\
			[8pt]																		
			\multicolumn{6}{c}{Model II (case i)}&\multicolumn{6}{c}{Model II (case ii)}\\[6pt]
			$(p,n)$    &	100      &   200	 &	  300	     & 	400	    &	$(p,n)$   &	 100     &   200	  &   300	&  400	  \\
			10         &	0.932	 &   0.965   &   0.974   	 &  0.981	&	    10    &	0.562	 &   0.562    &   0.582 &  0.570	\\
			           &	0.925	 &   0.957   &   0.966  	 &  0.975	&	          &	0.584	 &   0.568    &   0.571 &  0.552	\\
			20         &	0.852	 &   0.927   &   0.950	     &  0.967	&	    20	  &	0.504	 &   0.523    &   0.517 &  0.528   \\
			           &	0.877	 &   0.922   &   0.940	     &  0.958	&	          &	0.566	 &   0.572    &   0.566 &  0.548   \\
			30	       &	0.776    &   0.888   &   0.926       &  0.945   &	    30    &	0.447    & 	 0.502    &   0.513 &  0.513  \\
			           &	0.853	 &   0.892   &   0.919	     &  0.936	&	          & 0.565	 &   0.565    &   0.555 &  0.553   \\
			[8pt]	
			\multicolumn{6}{c}{Model III (case i)}&\multicolumn{6}{c}{Model III (case ii)}\\[6pt]
			$(p,n)$    &	100      &   200	  &	  300	   & 	400	    &	$(p,n)$   &	 100     &   200	  &   300	&  400	  \\
			10         &	0.981	 &   0.990    &   0.995    &  0.996	    &	    10    &	0.568	 &   0.577    &   0.578 &  0.581	 \\
			           &	0.981	 &   0.989    &   0.993    &  0.995	    &	          &	0.668	 &   0.653    &   0.646 &  0.649	\\
			20         &	0.954	 &   0.979    &   0.986    &  0.991     &	    20    &	0.497	 &   0.522    &   0.520 &  0.525	\\
			           &	0.962	 &   0.979    &   0.985    &  0.989     &	          &	0.663	 &   0.638    &   0.637 &  0.653	\\
			30	       &	0.928   & 	 0.969    &   0.979    &  0.984     &	    30    &	0.476	 &   0.499    &   0.512 &  0.507	\\
			           &    0.951	 &   0.971    &   0.978    &  0.983     &	          &	0.657	 &   0.647    &   0.648 &  0.639	\\
			[8pt]																		
			\multicolumn{6}{c}{Model IV (case i)}&\multicolumn{6}{c}{Model IV (case ii)}\\[6pt]
			$(p,n)$    &	100      &   200	 &	  300	     & 	400	    &	$(p,n)$   &	 100     &   200	  &   300	&  400	  \\
			10         &	0.984	 &   0.992   &   0.995   	 &  0.996	&	    10    &	0.642	 &   0.657    &   0.659 &  0.654	\\
			           &	0.984	 &   0.991   &   0.994  	 &  0.996	&	          &	0.619	 &   0.605    &   0.605 &  0.594	\\
			20         &	0.959	 &   0.982   &   0.987	     &  0.991	&	    20	  &	0.562	 &   0.574    &   0.571 &  0.569   \\
			           &	0.966	 &   0.981   &   0.986	     &  0.990	&	          &	0.611	 &   0.598    &   0.581 &  0.569   \\
			30	       &	0.928    &   0.969   &   0.981       &  0.985   &	    30    &	0.517    & 	 0.532    &   0.542 &  0.536   \\
			           &	0.951	 &   0.971   &   0.980	     &  0.984	&	          & 0.596	 &   0.571    &   0.571 &  0.550   \\
\end{tabular}}
%	\begin{tabnote}
%		The average $r^2$ and $\rho^2$ are listed in the first and second row for each $p$.
%	\end{tabnote}
\end{table}

% \begin{table}[h]\label{t2}
%	\def~{\hphantom{0}}
%	\tbl{\it The Averages of $\rho^2$ for the estimation of $\mathcal{G}_{Y|X}$ based on $100$ simulation runs. }{%
%		\begin{tabular}{cccccccccccc}
%			\\	
%			\multicolumn{6}{c}{Model II(case 2)}&\multicolumn{6}{c}{Model III(case 2)}\\[6pt]
%			$(p,n)$    &	100      &   200	 &	  300	     & 	400	    &   $(p,n)$   &	 100     &   200	  &   300   &  400	  \\
%			10         &	0.886	 &   0.888   &   0.889   	 &  0.8911	&	    10    &	0.913	 &   0.915    &   0.918 &  0.914	  \\
%			20         &	0.881	 &   0.885   &   0.893	     &  0.8848	&	    20	  &	0.917	 &   0.915    &   0.918 &  0.916   \\
%			30	       &	0.901    &   0.892   &   0.886       &  0.8914  &	    30    &	0.915    & 	 0.918    &   0.918 &  0.918   \\
%			[8pt]	
%			\multicolumn{6}{c}{Model IV(case 2)}&\\[6pt]
%			$(p,n)$    &	100      &   200	  &	  300	     & 	400	    	  \\
%			10         &	0.773	 &   0.780    &   0.780   	 &  0.776	  \\
%			20         &	0.770	 &   0.774    &   0.777	     &  0.774	   \\
%			30	       &	0.777    &   0.773    &   0.774      &  0.777     \\
%			[8pt]																	
%	\end{tabular}}
%\end{table}
 \begin{table}[ht]\label{t9}
	\def~{\hphantom{0}}
	\caption{\it The means of $\rho^2$ for estimating $\mathcal{G}_{Y|X}$ among $100$ repetitions with scenario 3.  }{%
		\begin{tabular}{ccccccccccccc}	
			&\multicolumn{4}{c}{Model II (case ii)}&\multicolumn{4}{c}{Model III (case ii)}&\multicolumn{4}{c}{Model IV (case ii)}\\[6pt]
			$(p,n)$    &	100      &   200	 &	  300	     & 	400	    &    100     &   200	  &   300   &  400	&	100      &   200	  &	  300	     & 	400  \\
			10         &	0.886	 &   0.888   &   0.889   	 &  0.8911	&	0.773	 &   0.780    &   0.780 &  0.776&	0.913	 &   0.915    &   0.918 &  0.914\\
			20         &	0.881	 &   0.885   &   0.893	     &  0.8848	&	0.770	 &   0.774    &   0.777 &  0.774&	0.917	 &   0.915    &   0.918 &  0.916   \\
			30	       &	0.901    &   0.892   &   0.886       &  0.8914  &	0.777    &   0.773    &   0.774 &  0.777&	0.915    & 	 0.918    &   0.918 &  0.918   \\
	\end{tabular}}
\end{table}

\begin{table}[ht]\label{t10}
	\def~{\hphantom{0}}
	\caption{\it The number of correctly estimation for $d$ among $100$ repetitions with scenario 3.  }{%
		\begin{tabular}{ccccccccccccccc}
			\\	
			\multicolumn{5}{c}{Model I (case i)}&\multicolumn{5}{c}{Model I (case ii)}&\multicolumn{5}{c}{Model II (case i)}\\[6pt]
			$(p,n)$    &	100      &   200	  &	  300	     & 	400	    &	$(p,n)$   &	 100   &   200	  &   300	    &  400	& $(p,n)$    &	100      &   200	  &	  300	     & 	400	   \\
			10         &	100  	 &   100      &   100   	 &  100	    &	    10    &	100	   &   100    &   100   	&  100  &10         &	90  	 &   99       &   98   	     &  99  	\\
			20         &	100	     &   100      &   100	     &  100	    &	    20	  &	100	   &   100    &   100	    &  100  & 20         &	76  	 &   95       &   99	     &  100 \\
			30	       &	100      &   100      &   100        &  100     &	    30    &	97     &   98     &   100       &  100  & 30	       &	28  	 &   81       &   96         &  100  \\
			[8pt]																		
\multicolumn{5}{c}{Model III (case i)}&\multicolumn{5}{c}{Model IV (case i)}&\\[6pt]
				$(p,n)$   &	 100   &   200	&   300	    &  400	 &   $(p,n)$    &	100    &   200	  &	  300	    & 	400	\\
			        10    &	98	   &   99   &   100   	&  100   &     10       &	100	   &   100    &   100   	&  100   \\
			        20    &	98	   &   97   &   98	    &  100	 &     20	    &	100	   &   100    &   100   	&  100   \\
			        30	  &	96     &   96   &   99      &  99    &     30       &	100	   &   100    &   100   	&  100   \\
\end{tabular}}
\end{table}

\begin{table}[ht]\label{ta11}
	\def~{\hphantom{0}}
	\caption{\it The means of $r^2$ and $\rho^2$  for estimating $\mathcal{S}_{Y|X}$ among $100$ repetitions with scenario 4. }{%
		\begin{tabular}{cccccccccccc}
			\\	
			\multicolumn{6}{c}{Model I (case i)}&\multicolumn{6}{c}{Model I (case ii)}\\[6pt]
			$(p,n)$    &	100      &   200	 &	  300	     & 	400	    &	$(p,n)$   &	 100     &   200	  &   300	&  400	    \\
			10         &	0.792	 &   0.830   &   0.865   	 &  0.898	&	    10    &	0.908	 &   0.932    &  0.931  &  0.932	\\
			           &	0.775	 &   0.802   &   0.835  	 &  0.869	&	          &	0.909	 &   0.929    &  0.925  &  0.925	\\
			20         &	0.682	 &   0.755   &   0.758	     &  0.787	&	    20	  &	0.839	 &   0.878    &  0.884  &  0.890	\\
			           &	0.670	 &   0.713   &   0.710	     &  0.743	&	          &	0.859	 &   0.876    &  0.876  &  0.879   \\
			30	       &	0.608    &   0.687   &   0.706       &  0.739   &	    30    &	0.809    &   0.838    &  0.845  &  0.874   \\
			           &	0.618	 &   0.654   &   0.660	     &  0.688	&	          &	0.851	 &   0.842    &  0.836  &  0.863  \\
			[8pt]																		
			\multicolumn{6}{c}{Model II (case i)}&\multicolumn{6}{c}{Model II (case ii)}\\[6pt]
			$(p,n)$    &	100      &   200	 &	  300	     & 	400	    &	$(p,n)$   &	 100     &   200	  &   300	 &  400	  \\
			10         &	0.907	 &   0.947   &   0.966   	 &  0.975	&	    10    &	0.580	 &   0.589    &   0.585  &  0.585	\\
			           &	0.913	 &   0.942   &   0.957  	 &  0.967	&	          &	0.640	 &   0.616    &   0.610  &  0.601	\\
			20         &	0.808	 &   0.906   &   0.937	     &  0.955	&	    20	  &	0.505	 &   0.531    &   0.531  &  0.536   \\
			           &	0.853	 &   0.904   &   0.929	     &  0.944	&	          &	0.604	 &   0.613    &   0.612  &  0.598   \\
			30	       &	0.725    &   0.866   &   0.908       &  0.935   &	    30    &	0.459    & 	 0.500    &   0.510  &  0.524   \\
			           &	0.825	 &   0.879   &   0.904	     &  0.925	&	          & 0.581	 &   0.598    &   0.600  &  0.605   \\
			[8pt]	
			\multicolumn{6}{c}{Model III (case i)}&\multicolumn{6}{c}{Model III (case ii)}\\[6pt]
			$(p,n)$    &	100      &   200	  &	  300	     & 	400	    &	$(p,n)$   &	 100     &   200	  &   300	&  400	  \\
			10         &	0.970	 &   0.984   &   0.989   	 &  0.992	&	    10    &	0.620	 &   0.634    &   0.652 &  0.660	\\
			           &	0.974	 &   0.984   &   0.988  	 &  0.991	&	          &	0.710	 &   0.713    &   0.716 &  0.713	\\
			20         &	0.935	 &   0.969   &   0.981	     &  0.984	&	    20	  &	0.535	 &   0.551    &   0.561 &  0.567   \\
			           &	0.952	 &   0.971   &   0.980	     &  0.983	&	          &	0.711	 &   0.709    &   0.694 &  0.700   \\
			30	       &	0.910    &   0.954   &   0.972       &  0.978   &	    30    &	0.501    & 	 0.528    &   0.536 &  0.532   \\
			           &	0.946	 &   0.960   &   0.973	     &  0.977	&	          & 0.713	 &   0.707    &   0.693 &  0.686   \\
			[8pt]																		
			\multicolumn{6}{c}{Model IV (case i)}&\multicolumn{6}{c}{Model IV (case ii)}\\[6pt]
			$(p,n)$    &	100      &   200	 &	  300	     & 	400	    &	$(p,n)$   &	 100     &   200	  &   300	&  400	  \\
			10         &	0.99	 &   0.982   &   0.990   	 &  0.992	&	    10    &	0.687	 &   0.702    &   0.697 &  0.700	\\
			           &	0.974	 &   0.982   &   0.989  	 &  0.991	&	          &	0.675	 &   0.656    &   0.653 &  0.652	\\
			20         &	0.931	 &   0.972   &   0.981	     &  0.984	&	    20	  &	0.587	 &   0.606    &   0.599 &  0.612   \\
			           &	0.950	 &   0.974   &   0.981	     &  0.983	&	          &	0.665	 &   0.639    &   0.622 &  0.629   \\
			30	       &	0.909    &   0.956   &   0.971       &  0.979   &	    30    &	0.537    & 	 0.554    &   0.559 &  0.558   \\
			           &	0.945	 &   0.962   &   0.972	     &  0.978	&	          & 0.662	 &   0.634    &   0.622 &  0.607  \\
\end{tabular}}
%	\begin{tabnote}
%		The average $r^2$ and $\rho^2$ are listed in the first and second row for each $p$.
%	\end{tabnote}
\end{table}

% \begin{table}[h]\label{t2}
%	\def~{\hphantom{0}}
%	\tbl{\it The Averages of $\rho^2$ for the estimation of $\mathcal{G}_{Y|X}$ based on $100$ simulation runs. }{%
%		\begin{tabular}{cccccccccccc}
%			\\	
%			\multicolumn{6}{c}{Model II(case 2)}&\multicolumn{6}{c}{Model III(case 2)}\\[6pt]
%			$(p,n)$    &	100      &   200	 &	  300	     & 	400	    &   $(p,n)$   &	 100     &   200	  &   300   &  400	  \\
%			10         &	0.826	 &   0.835   &   0.826   	 &  0.828	&	    10    &	0.893	 &   0.897    &   0.899 &  0.902	  \\
%			20         &	0.825	 &   0.828   &   0.826	     &  0.831	&	    20	  &	0.892	 &   0.897    &   0.898 &  0.898   \\
%			30	       &	0.829    &   0.822   &   0.828       &  0.822   &	    30    &	0.893   & 	 0.897    &   0.899 &  0.899   \\
%			[8pt]
%			\multicolumn{6}{c}{Model IV(case 2)}&\\[6pt]
%			$(p,n)$    &	100      &   200	 &	  300	     & 	400	    	  \\
%			10         &	0.775	 &   0.771   &   0.779   	 &  0.774		  \\
%			20         &	0.777	 &   0.779   &   0.779	     &  0.777	   \\
%			30	       &	0.775    &   0.773   &   0.781       &  0.778    \\
%			[8pt]																		
%	\end{tabular}}
%\end{table}
 \begin{table}[ht]\label{t12}
	\def~{\hphantom{0}}
	\caption{\it The means of $\rho^2$ for estimating $\mathcal{G}_{Y|X}$ among $100$ repetitions with scenario 4.  }{%
		\begin{tabular}{ccccccccccccc}
			\\	
			&\multicolumn{4}{c}{Model II (case ii)}&\multicolumn{4}{c}{Model III (case ii)}&\multicolumn{4}{c}{Model IV (case ii)}\\[6pt]
			$(p,n)$    &	100      &   200	 &	  300	     & 	400	    &	 100     &   200	  &   300   &  400	&	100      &   200	 &	  300	     & 	400  \\
			10         &	0.826	 &   0.835   &   0.826   	 &  0.828	&	0.775	 &   0.771    &   0.779 &  0.774&	0.893	 &   0.897    &   0.899 &  0.902\\
			20         &	0.825	 &   0.828   &   0.826	     &  0.831	&	0.892	 &   0.897    &   0.898 &  0.898&	0.777	 &   0.779    &   0.779	     &  0.777   \\
			30	       &	0.829    &   0.822   &   0.828       &  0.822   &	0.775    &   0.773    &   0.781 &  0.778&	0.893    & 	 0.897    &   0.899 &  0.899\\												
	\end{tabular}}
\end{table}

\begin{table}[ht]\label{t2}
	\def~{\hphantom{0}}
	\caption{\it The number of correctly estimation for $d$ among $100$ repetitions with scenario 4. }{%
		\begin{tabular}{ccccccccccccccc}
			\\	
			\multicolumn{5}{c}{Model I (case i)}&\multicolumn{5}{c}{Model I (case ii)}&\multicolumn{5}{c}{Model II (case i)}\\[6pt]
			$(p,n)$    &	100      &   200	  &	  300	     & 	400	    &	$(p,n)$   &	 100   &   200	  &   300	    &  400	&  $(p,n)$    &	100      &   200	  &	  300	     & 	400  \\
			10         &	100  	 &   100      &   100   	 &  100	    &	    10    &	98	   &   99     &   100   	&  100  &	10         &	84  	 &   100      &   99  &100  	\\
			20         &	100	     &   100      &   100	     &  100	    &	    20	  &	98	   &   98     &   100	    &  100  &   20         &	80  	 &   99       &   100	&100    \\
			30	       &	100      &   100      &   100        &  100     &	    30    &	94     &   99     &   100       &  100  &   30	       &	29  	 &   92       &   99 &100\\
			[8pt]																		
\multicolumn{5}{c}{Model III (case i)}&\multicolumn{5}{c}{Model IV (case i)}&\\[6pt]
			$(p,n)$   &	 100   &   200	 &   300	 &  400	  &  $(p,n)$   &	100    &   200	  &	  300	    & 	400  \\
				10    &	100	   &   99    &   100   	 &  99    &    10      &	100	   &   100    &   100   	&  100\\
				20    &	96	   &   100   &   100	 &  100	  &    20	   &	99	   &   100    &   100   	&  100 \\
				30	  &	94     &   94    &   95      &  98    &    30      &	97	   &   100    &   100   	&  100  \\
\end{tabular}}
\end{table}

\clearpage
\section{Additional Results for the Handwritten Digits Data}
\subsection{Application to Handwritten Digit Classes $\{1,4,7\}$}

To further investigate the performance of our proposals and demonstrate its use in real applications, we now extract $1705$ gray-scale images of three handwritten digit classes $\{1,4,7\}$ in Figure 1, from the UCI Machine Learning Repository. This dataset contains a training group of size $1163$ and a testing group of size $542$. Each digit was represented by an $8\times 8$ pixel image. The $4\times 8$ bottom part of each image was taken as the predictors $X$, and the $4\times 8$ upper half was set as the responses $Y$.
\begin{figure}[htbp]
    \centering
    \subfigure{
         \includegraphics[height=1.3in,width=5in]{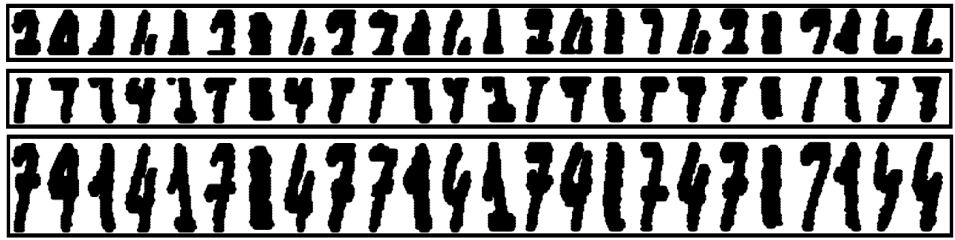}
    }
    \caption{The first row consists of the responses $Y$ which are the upper halves of the image digits $\{1,4,7\}$; The second row consists of the predictors $X$ which are the bottom halves of the image digits $\{1,4,7\}$; The third row consists of the whole image digits $\{1,4,7\}$.}
\end{figure}

\begin{figure}[ht]
\centering
\setlength{\abovecaptionskip}{0.cm}
\subfigure{
\begin{minipage}[c]{0.3\textwidth}
\centering
\includegraphics[height=2in,width=1.8in]{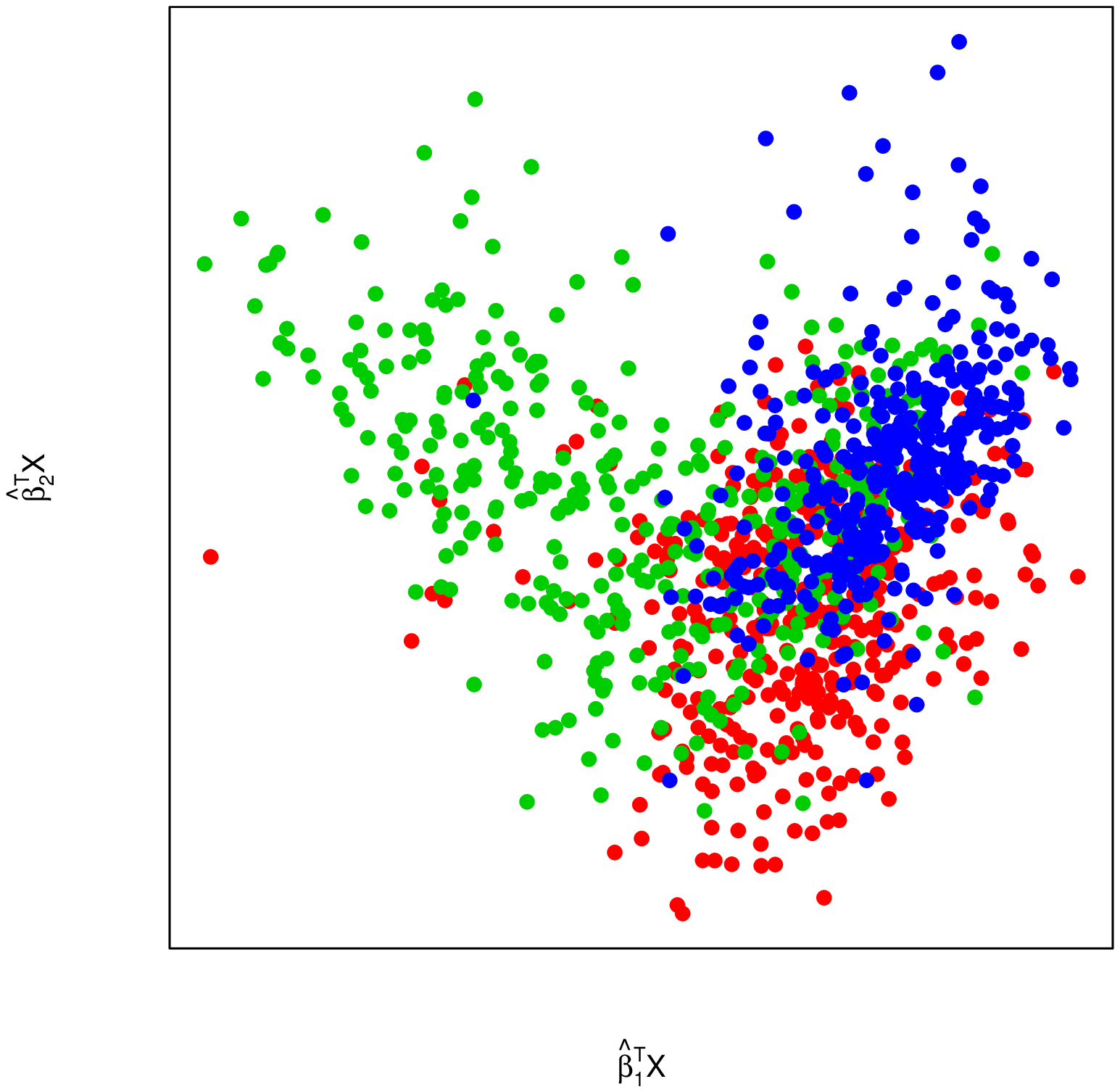}
\end{minipage}
\begin{minipage}[c]{0.3\textwidth}
\centering
\includegraphics[height=2in,width=1.8in]{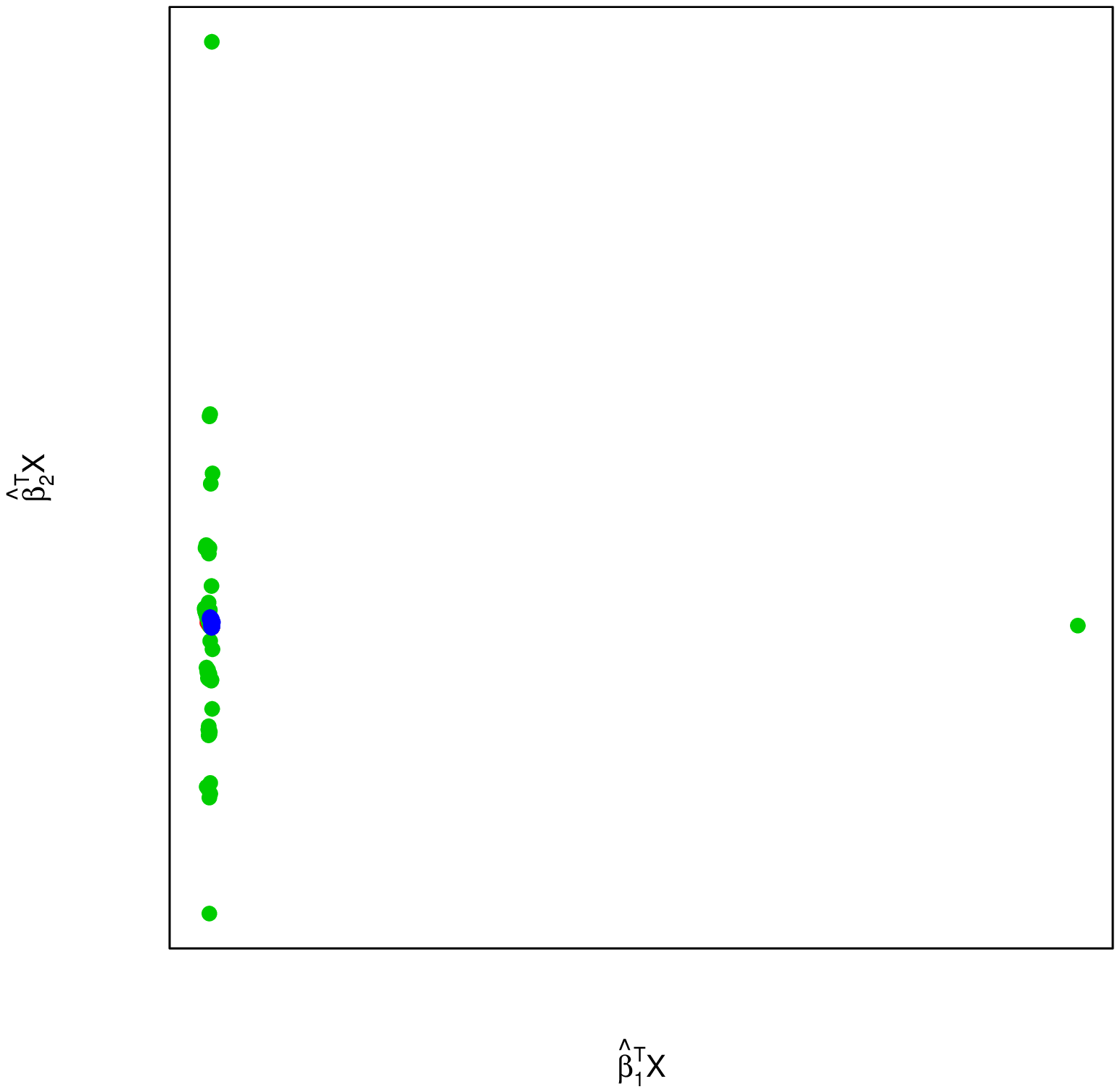}
\end{minipage}
\begin{minipage}[c]{0.3\textwidth}
\centering
\includegraphics[height=2in,width=1.8in]{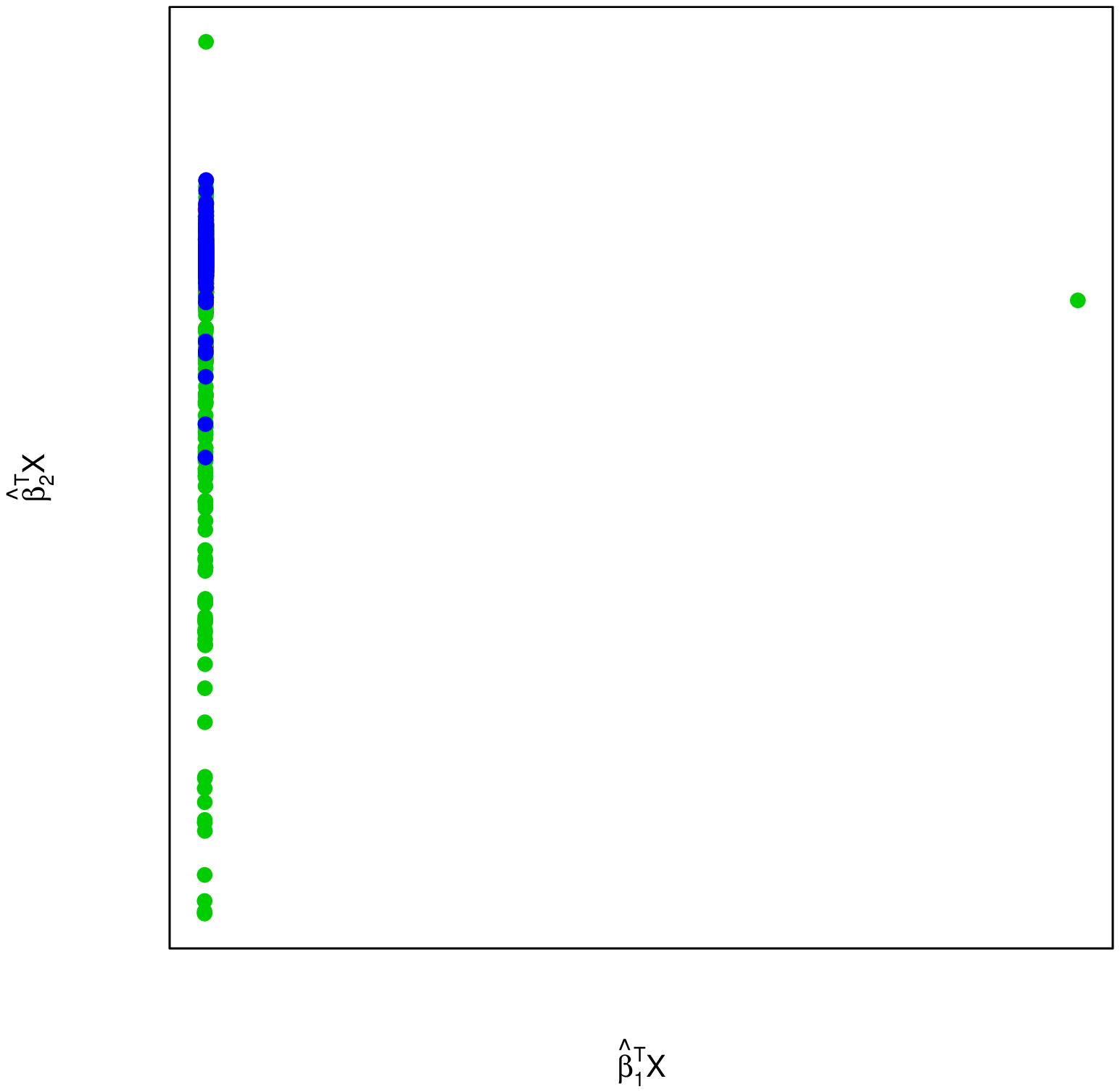}
\end{minipage}
}
\centering
\subfigure{
\begin{minipage}[c]{0.3\textwidth}
\centering
\includegraphics[height=2in,width=1.8in]{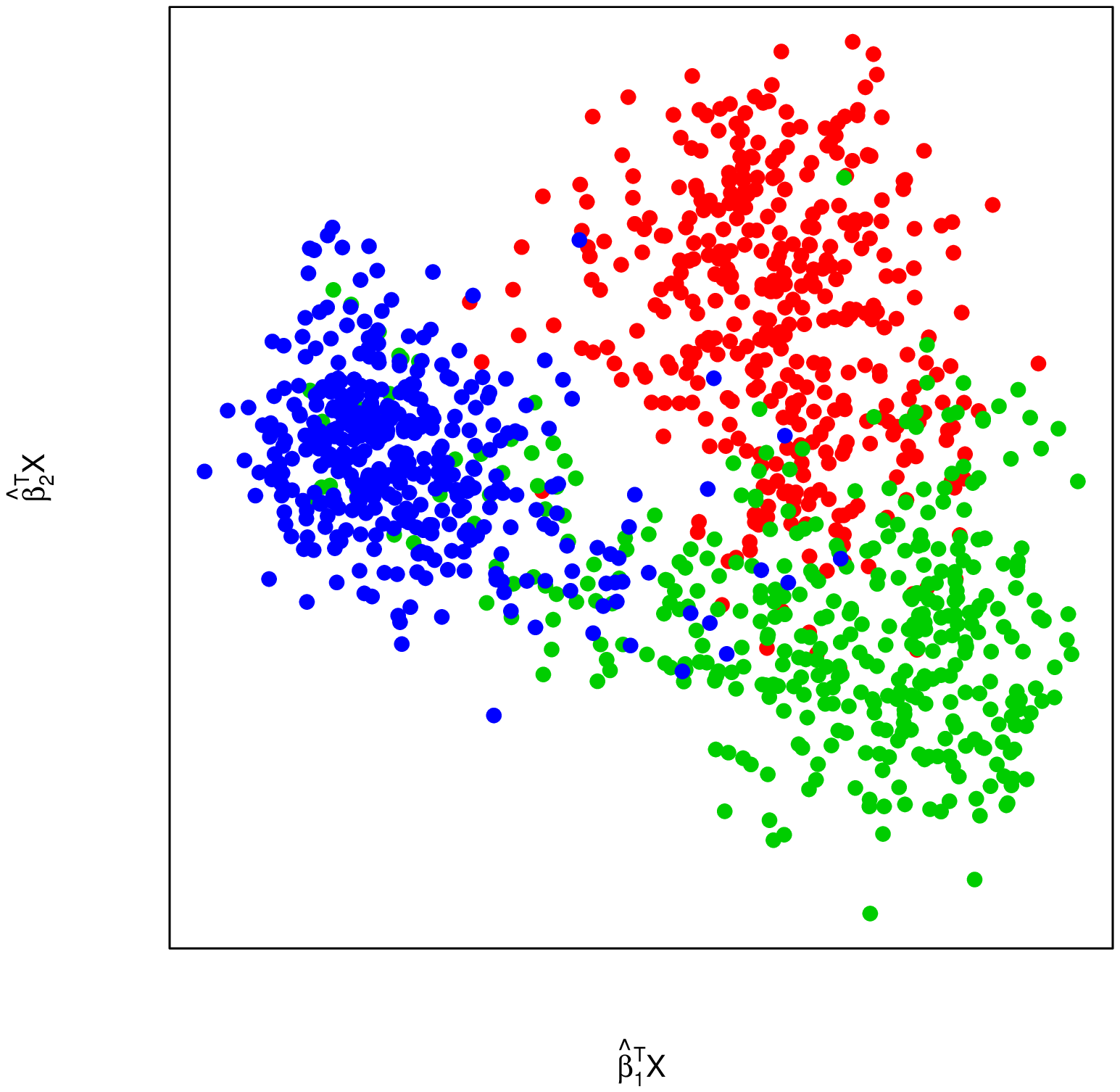}
\end{minipage}
\begin{minipage}[c]{0.3\textwidth}
\centering
\includegraphics[height=2in,width=1.8in]{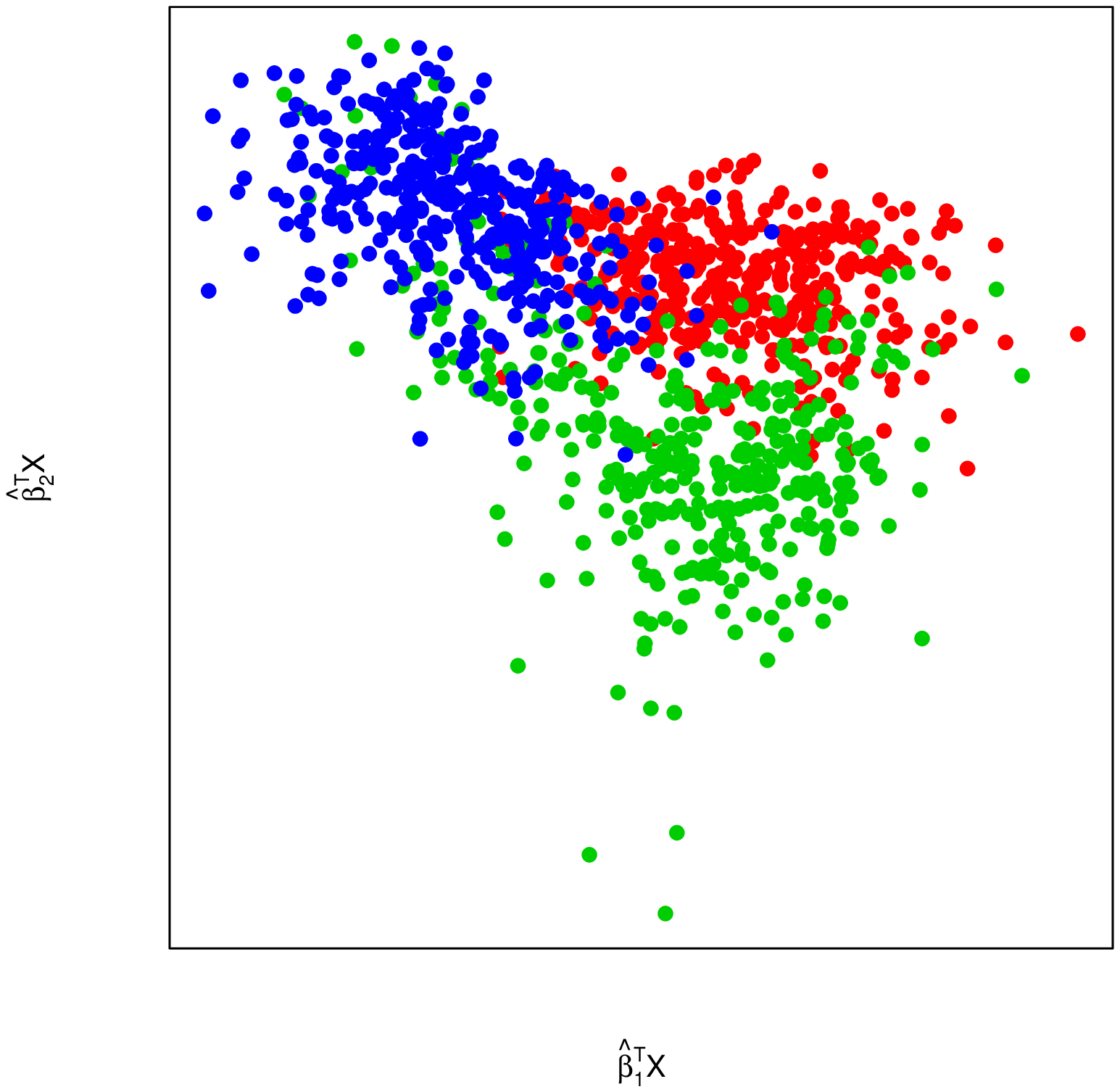}
\end{minipage}
\begin{minipage}[c]{0.3\textwidth}
\centering
\includegraphics[height=2in,width=1.8in]{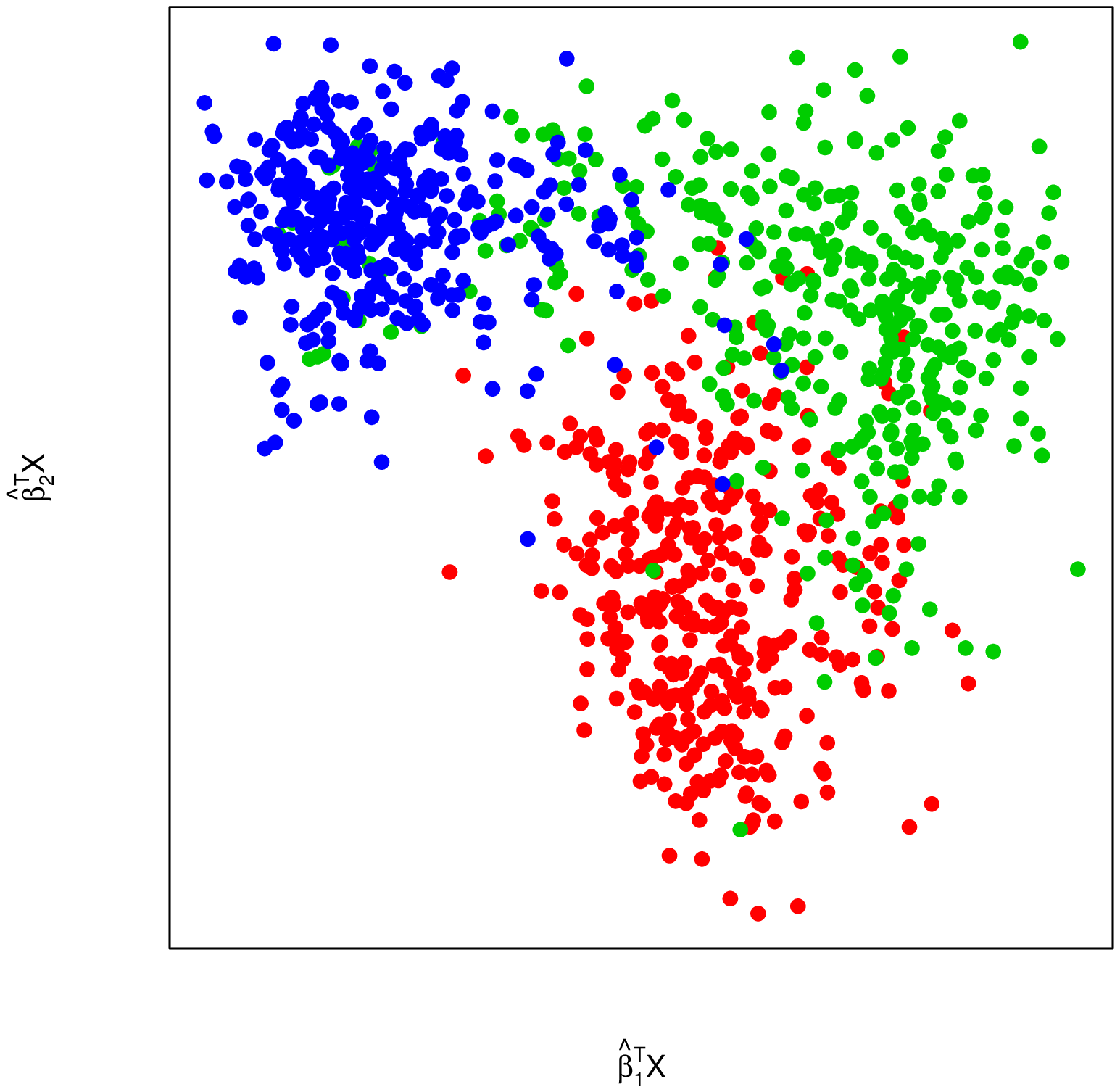}
\end{minipage}
}
\centering
\subfigure{
\begin{minipage}[c]{0.3\textwidth}
\centering
\includegraphics[height=2in,width=1.8in]{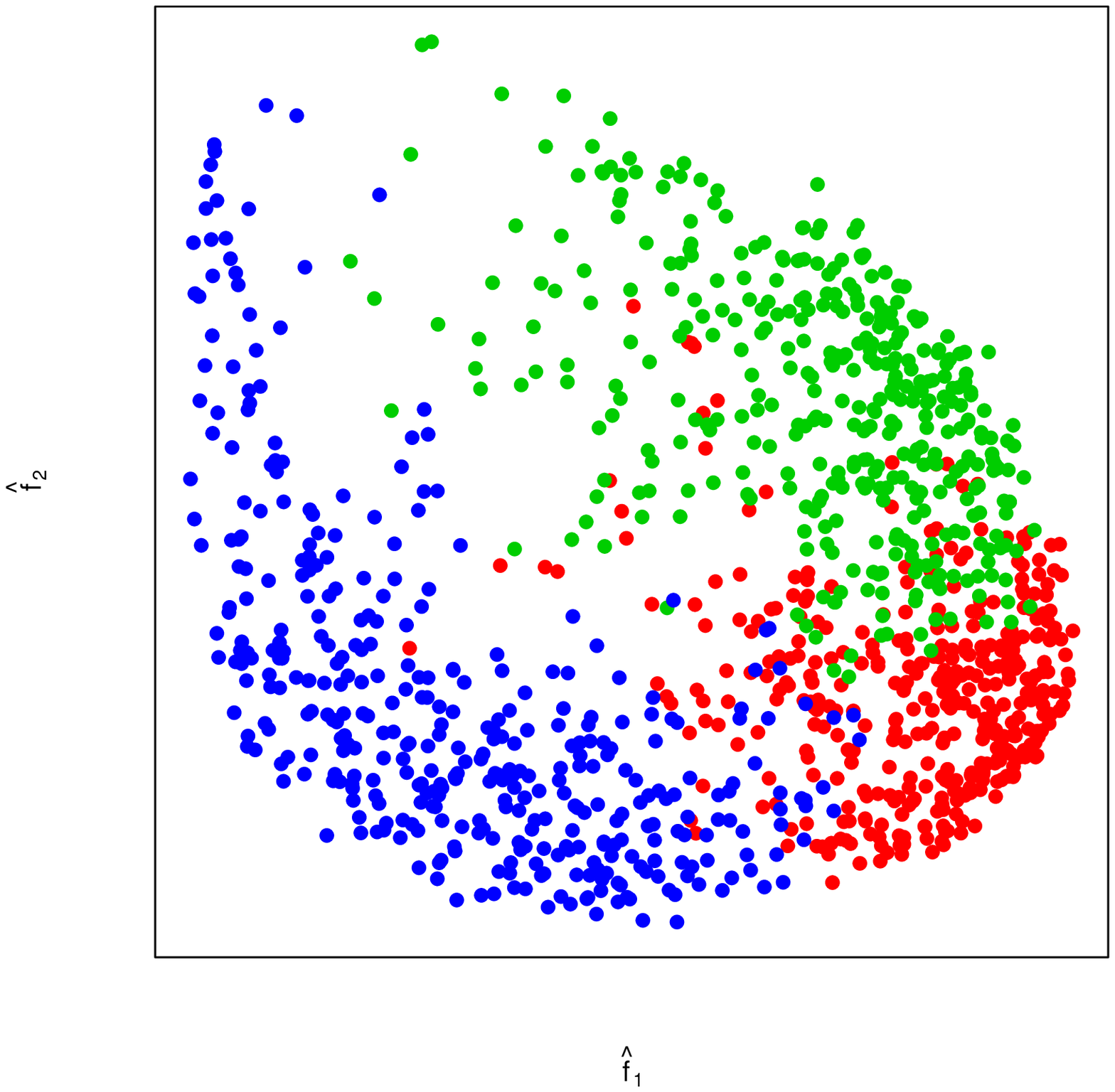}
\end{minipage}
\begin{minipage}[c]{0.3\textwidth}
\centering
\includegraphics[height=2in,width=1.8in]{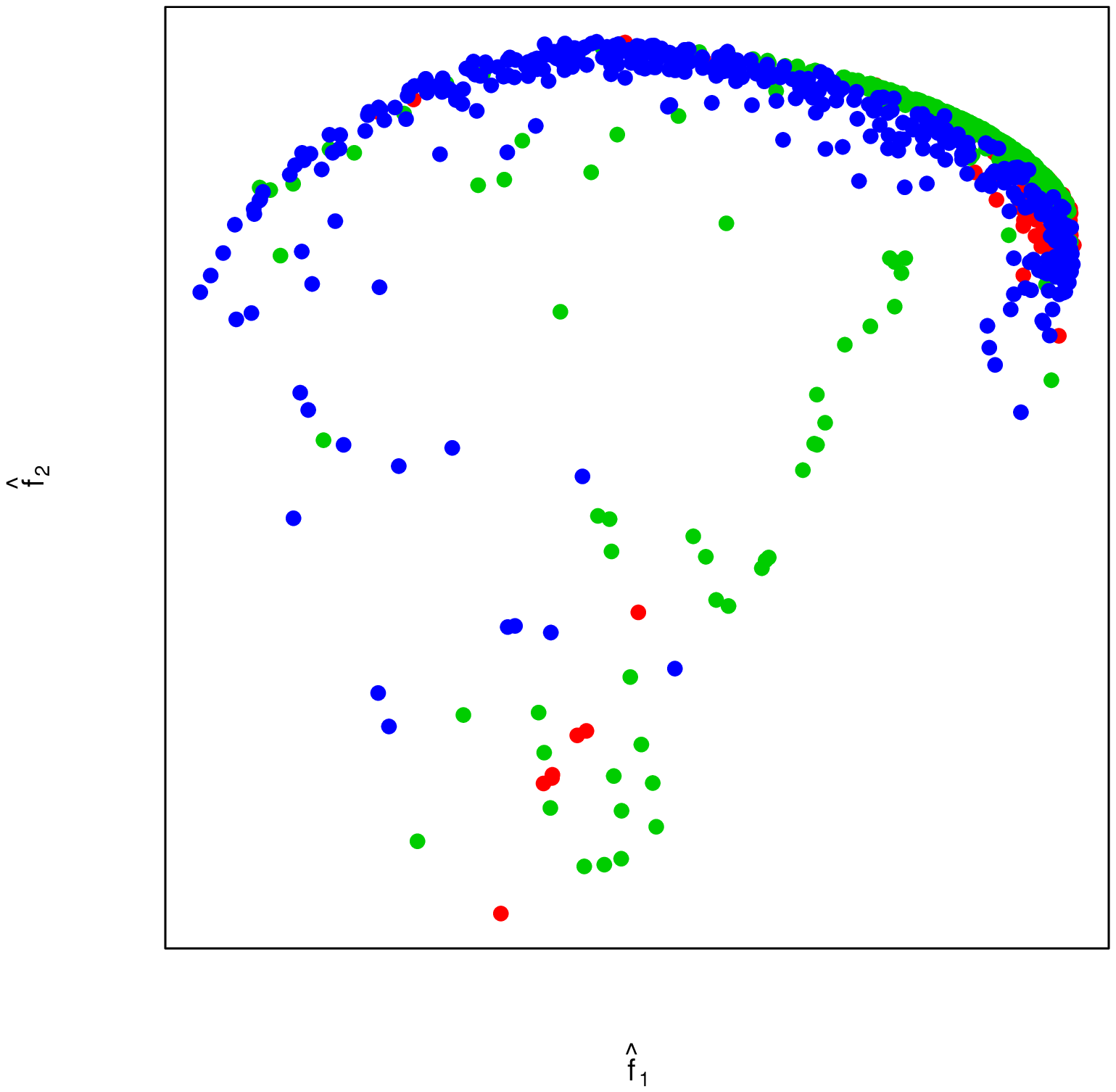}
\end{minipage}
\begin{minipage}[c]{0.3\textwidth}
\centering
\includegraphics[height=2in,width=1.8in]{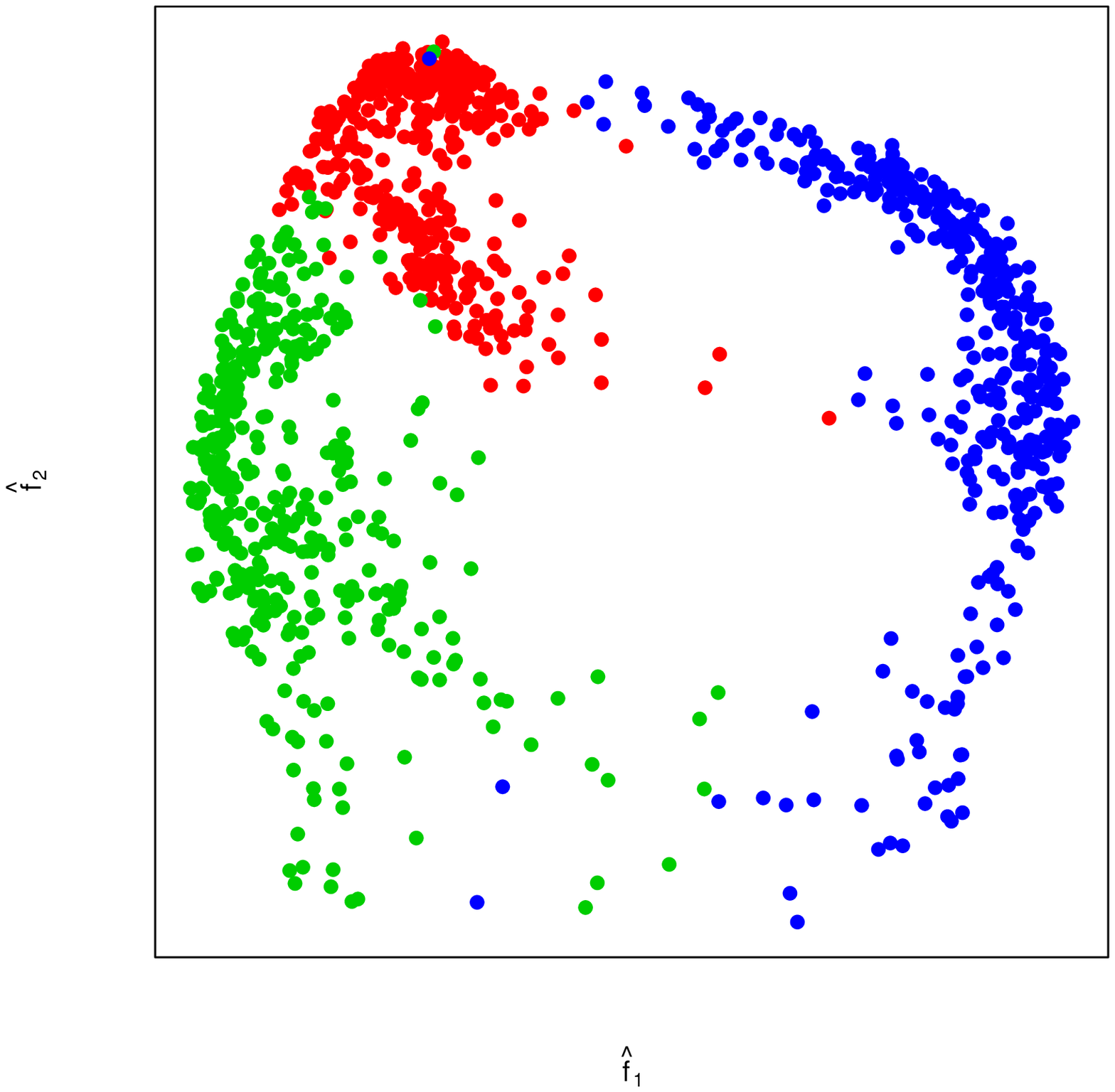}
\end{minipage}
}
\caption{The first row consists of the scatter plots for the training data via projective resampling based SIR, SAVE and DR, respectively. The second row consists of the scatter plots based on our linear proposal with Euclidean distance, LLE and Isomap, respectively. The third row consists of the scatter plots based on our nonlinear proposal with Euclidean distance, LLE and Isomap, respectively. (red: $1$; green: $4$;blue: $7$.)}
\end{figure}

%\begin{figure}
%\centering
%\begin{minipage}[c]{0.3\textwidth}
%\centering
%\includegraphics[height=2in,width=1.8in]{nonlinearEu1.eps}
%\end{minipage}
%\begin{minipage}[c]{0.3\textwidth}
%\centering
%\includegraphics[height=2in,width=1.8in]{nonlinearlle1.eps}
%\end{minipage}
%\begin{minipage}[c]{0.3\textwidth}
%\centering
%\includegraphics[height=2in,width=1.8in]{nonlineariso1.eps}
%\end{minipage}
%\caption{Perspective plots for the first two predictors estimated by our nonlinear proposals with Euclidean distance, LLE distance and Isomap distance, respectively.(red: 1; green: 4;blue: 7)}
%\end{figure}

\begin{figure}[ht]
\centering
\setlength{\abovecaptionskip}{0.cm}
\subfigure{
\begin{minipage}[c]{0.3\textwidth}
\centering
\includegraphics[height=2in,width=1.8in]{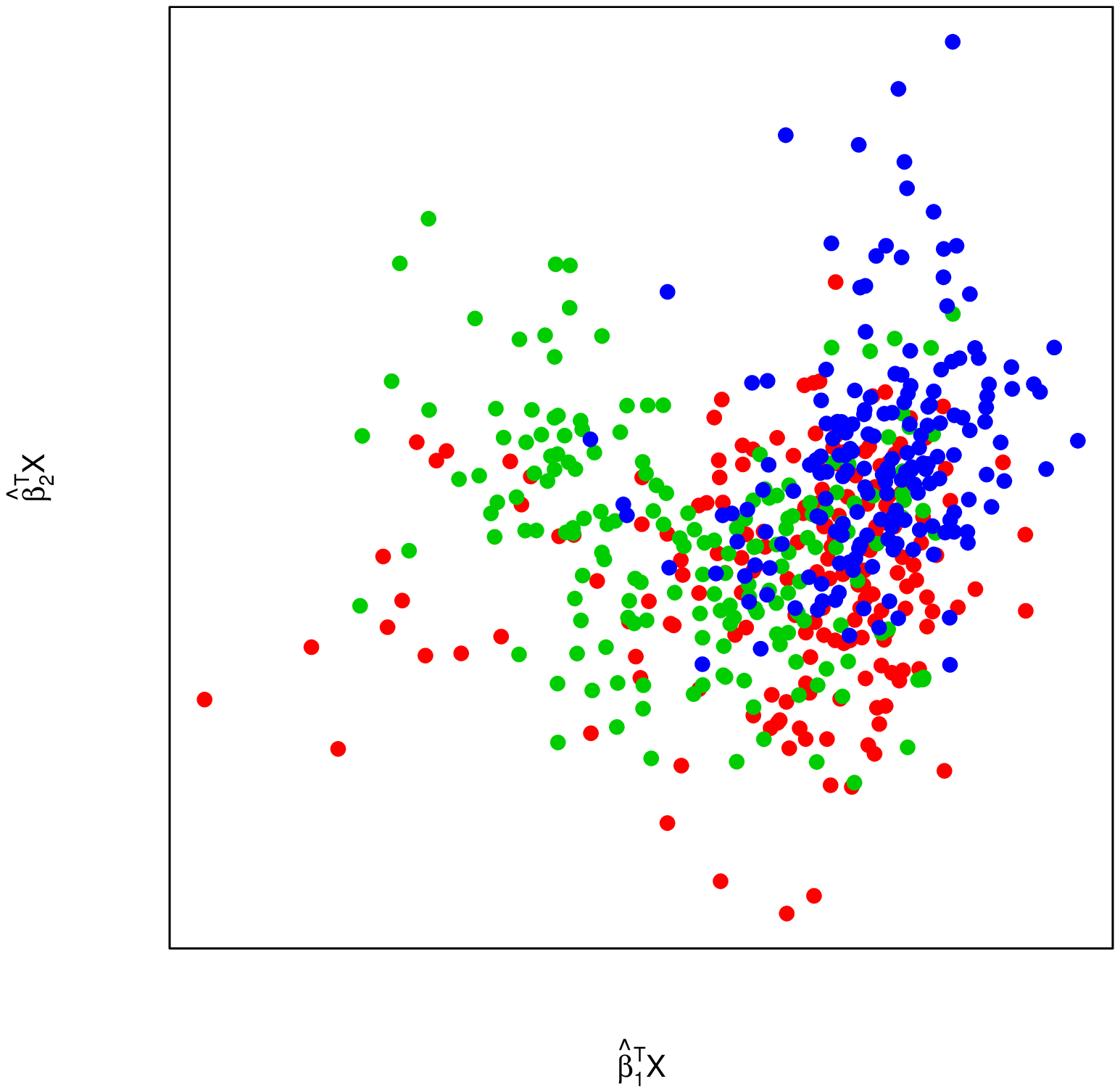}
\end{minipage}
\begin{minipage}[c]{0.3\textwidth}
\centering
\includegraphics[height=2in,width=1.8in]{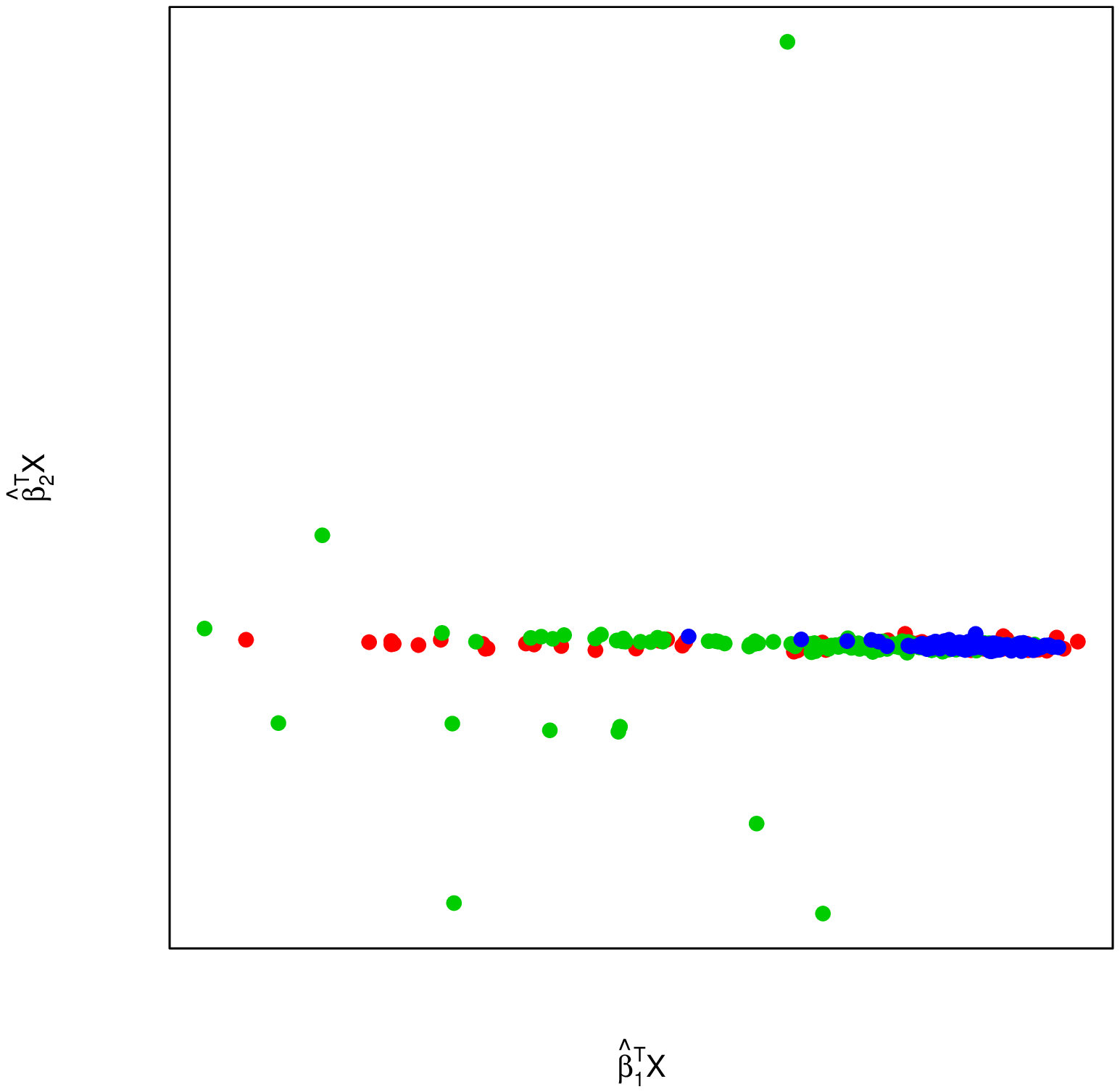}
\end{minipage}
\begin{minipage}[c]{0.3\textwidth}
\centering
\includegraphics[height=2in,width=1.8in]{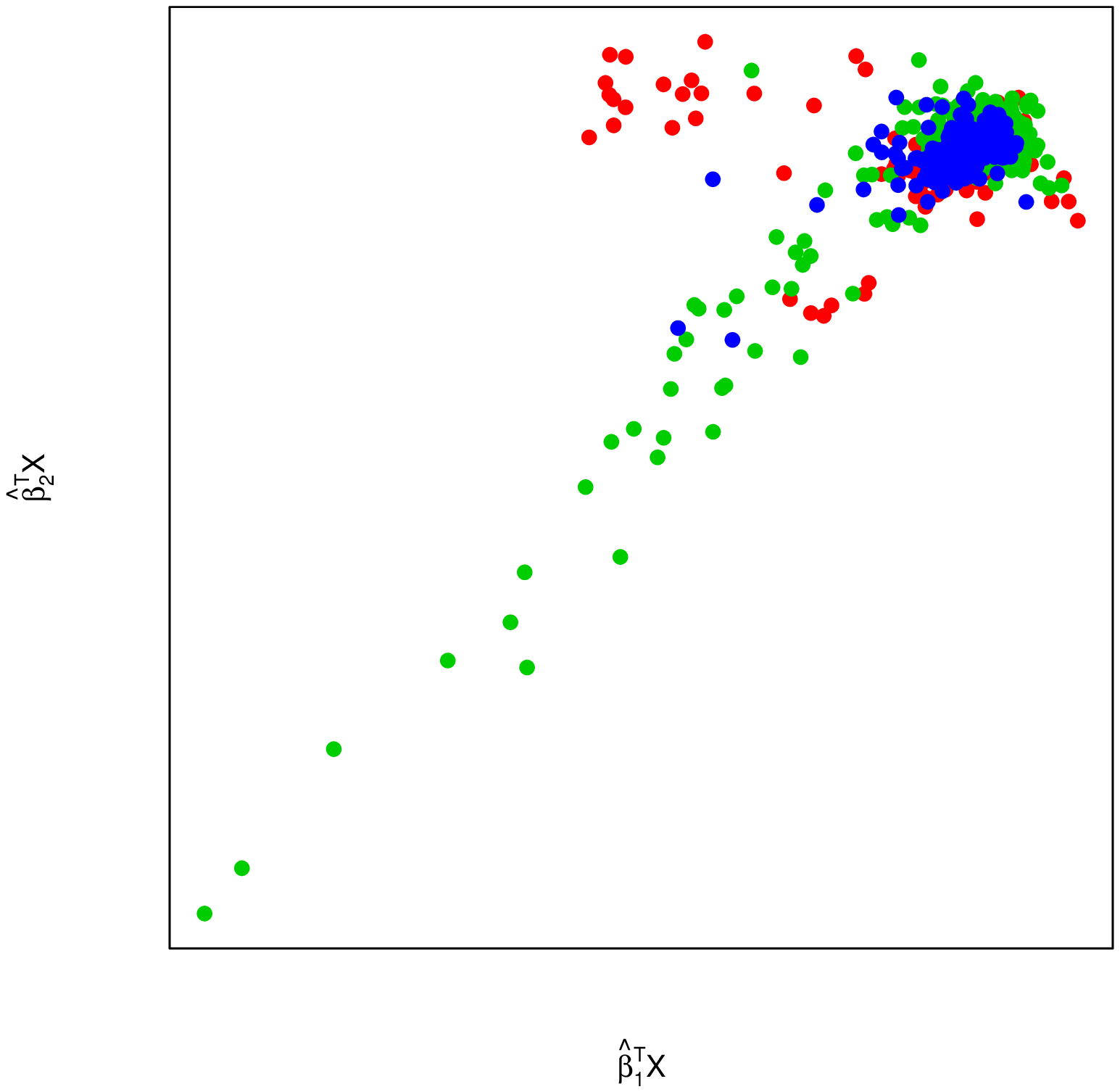}
\end{minipage}
}
\centering
\subfigure{
\begin{minipage}[c]{0.3\textwidth}
\centering
\includegraphics[height=2in,width=1.8in]{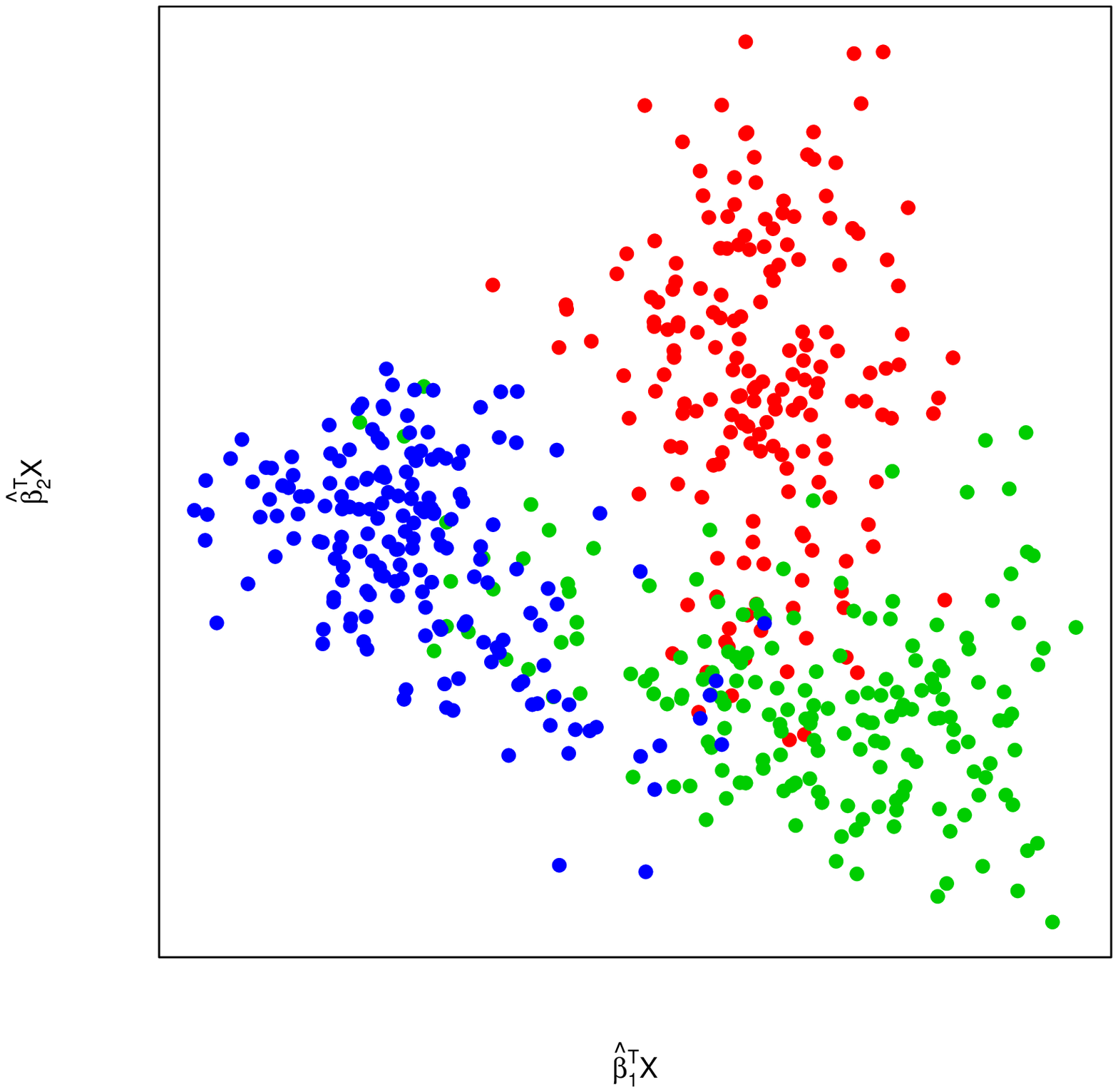}
\end{minipage}
\begin{minipage}[c]{0.3\textwidth}
\centering
\includegraphics[height=2in,width=1.8in]{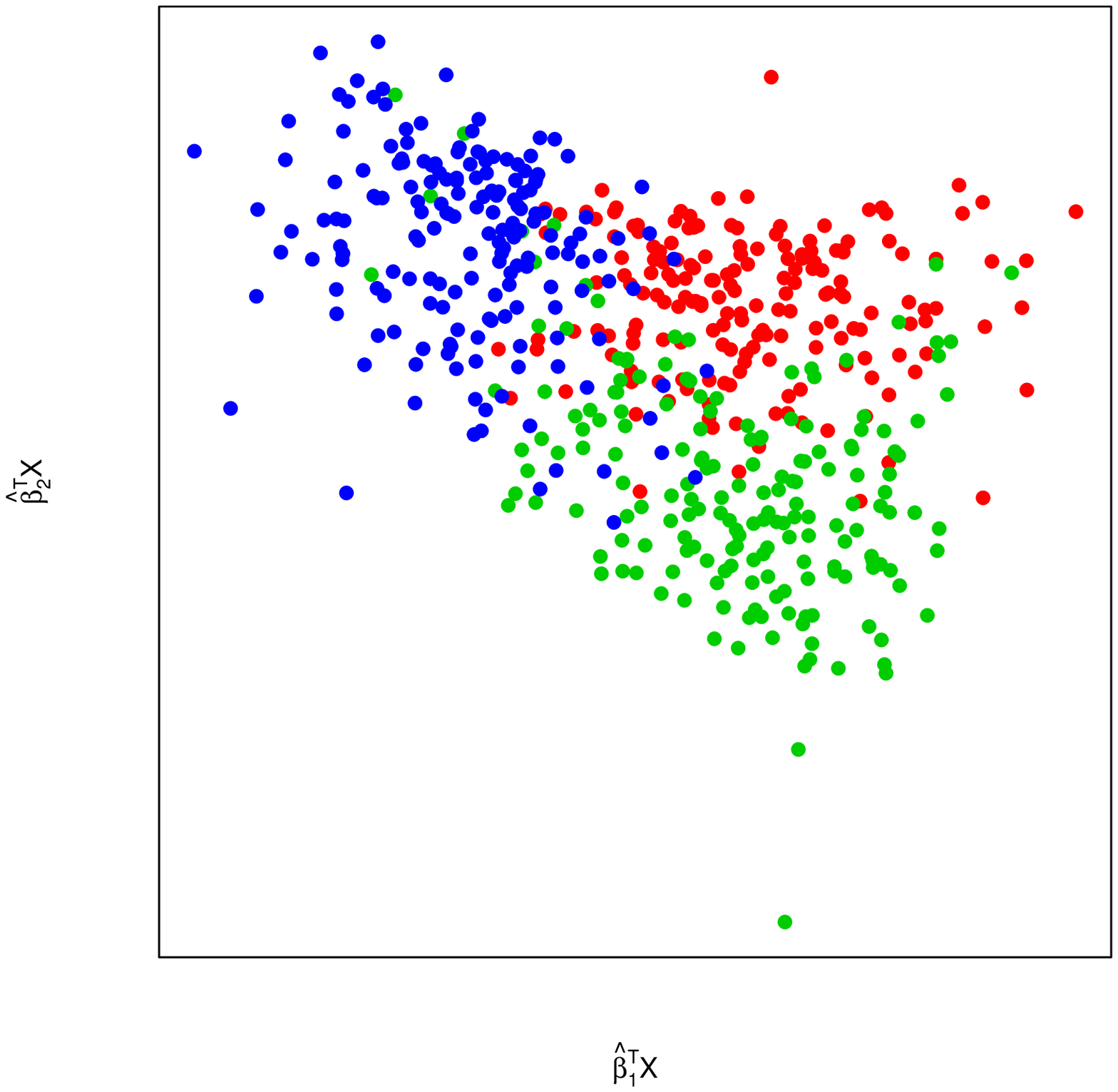}
\end{minipage}
\begin{minipage}[c]{0.3\textwidth}
\centering
\includegraphics[height=2in,width=1.8in]{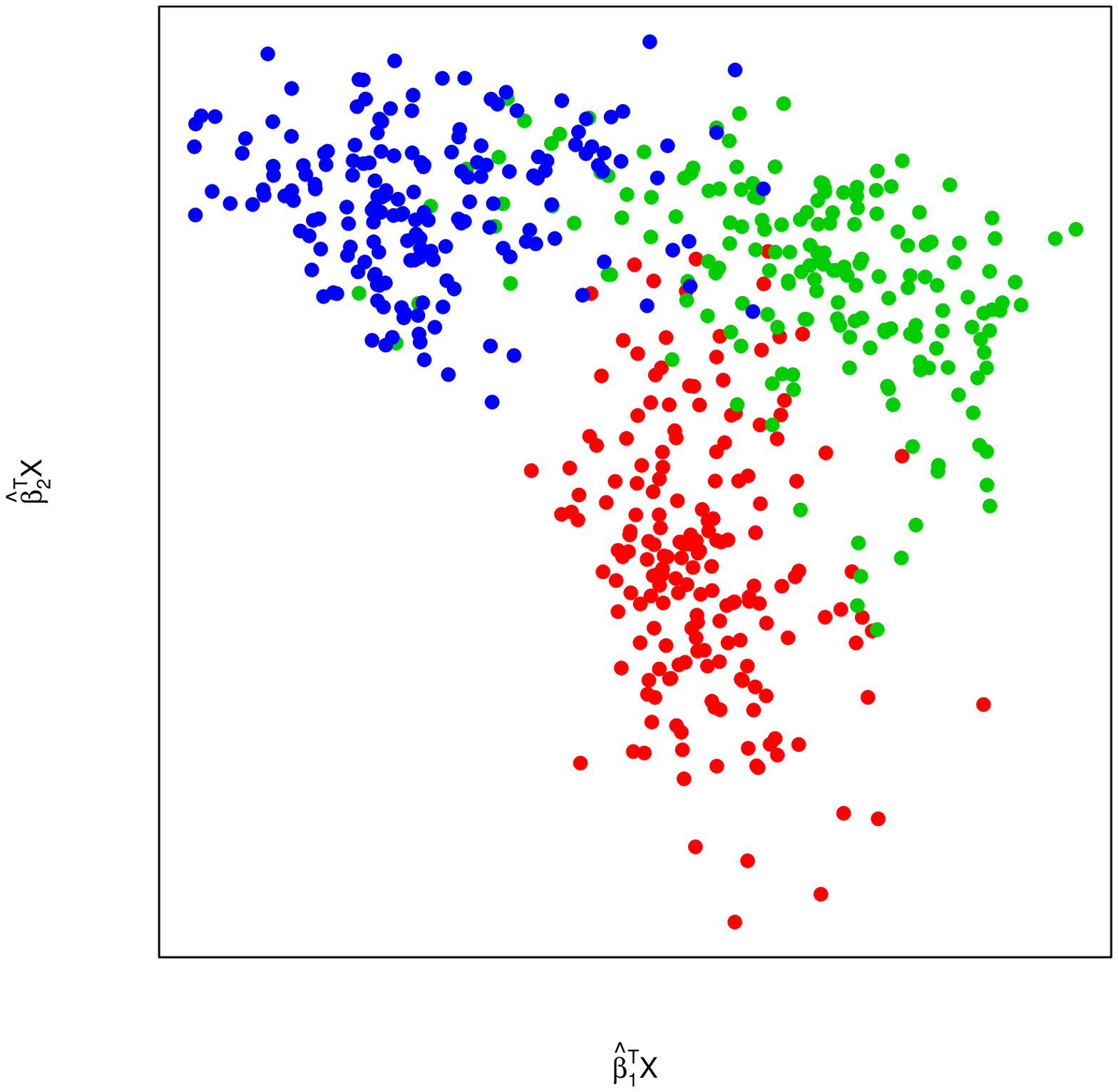}
\end{minipage}
}
\centering
\subfigure{
\begin{minipage}[c]{0.3\textwidth}
\centering
\includegraphics[height=2in,width=1.8in]{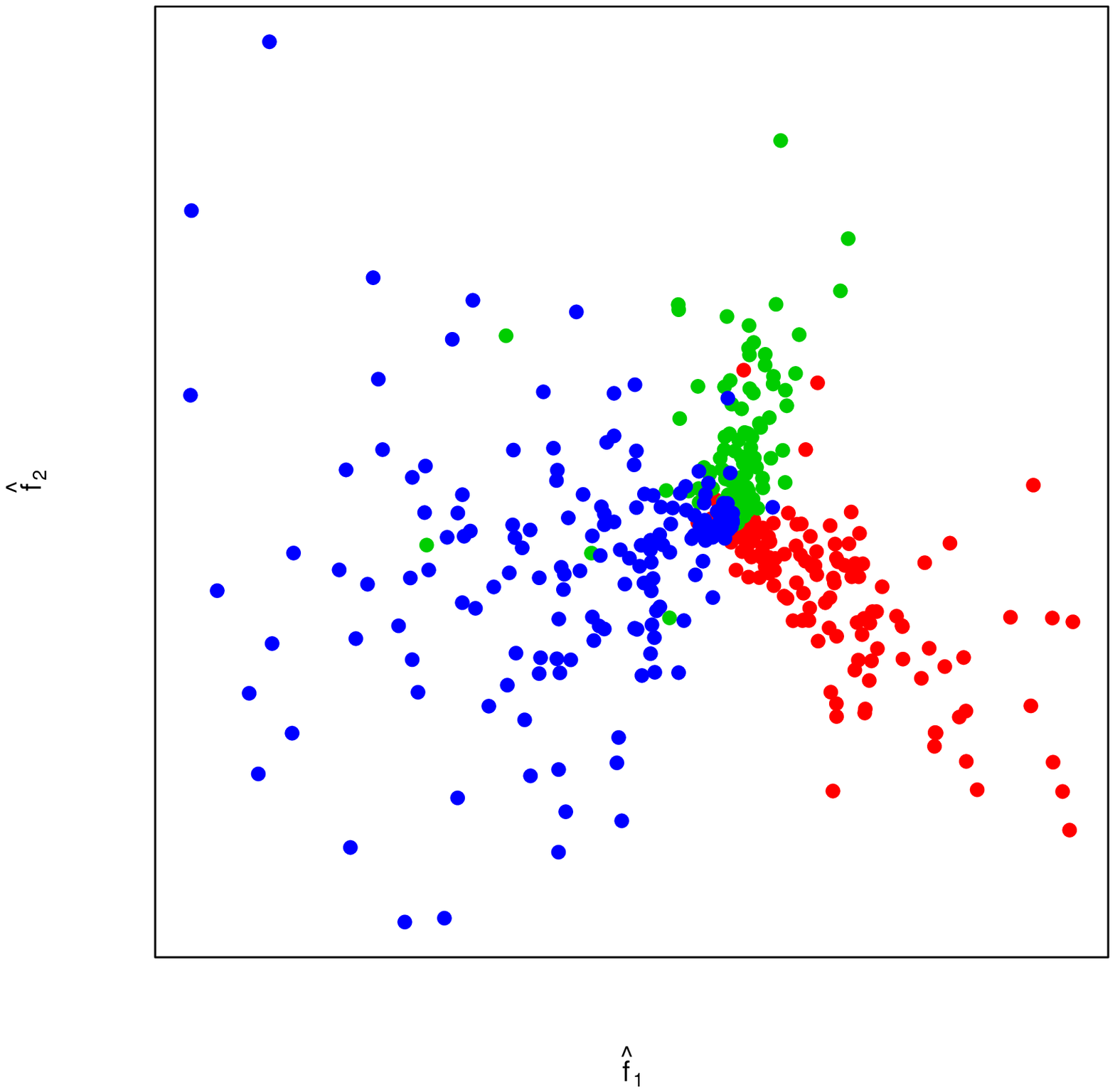}
\end{minipage}
\begin{minipage}[c]{0.3\textwidth}
\centering
\includegraphics[height=2in,width=1.8in]{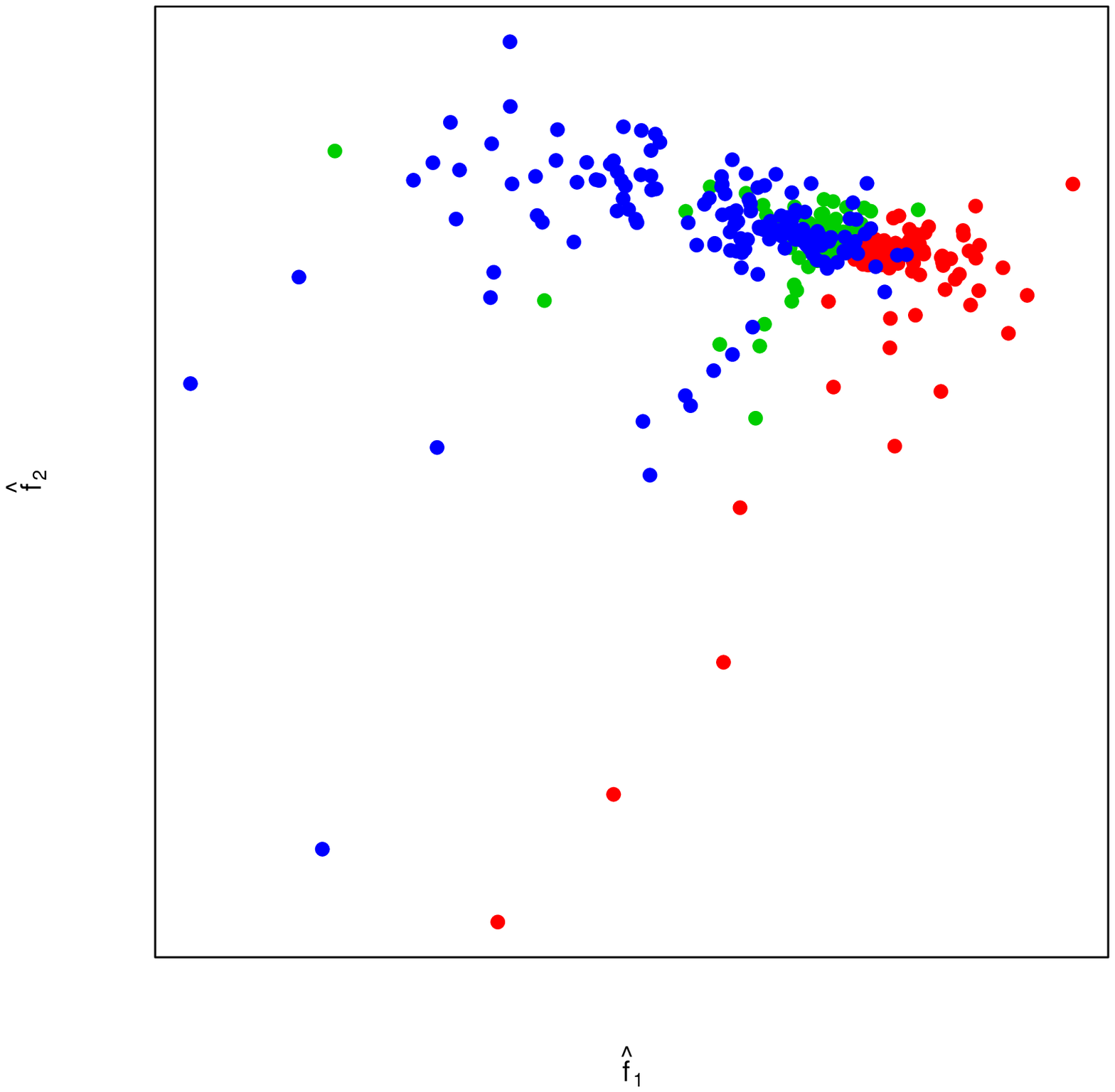}
\end{minipage}
\begin{minipage}[c]{0.3\textwidth}
\centering
\includegraphics[height=2in,width=1.8in]{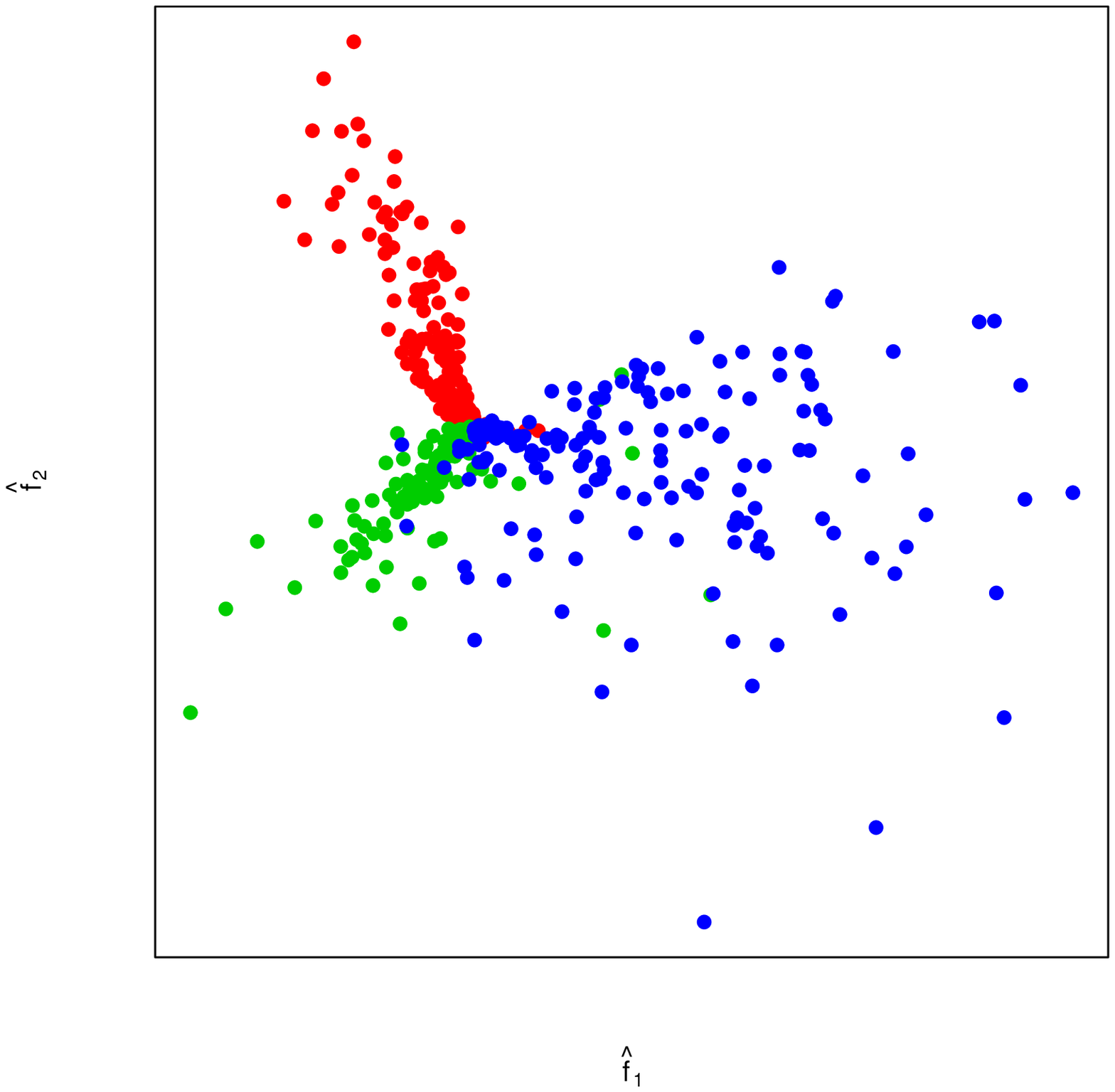}
\end{minipage}
}
\caption{The first row consists of the perspective plots for the first two sufficient predictors for the testing data by SIR, SAVE and DR, respectively. The second row consists of the perspective plots for the testing data based on our linear proposal with Euclidean distance, LLE and Isomap, respectively. The third row consists of the perspective plots for our nonlinear proposal with Euclidean distance, LLE and Isomap, respectively. (red: $1$; green: $4$;blue: $7$.)}
\end{figure}

For image digits $\{1,4,7\}$, we also include projective resampling approach in combination with three classical methods, sliced inverse regression, sliced average variance estimation and directional regression for comparisons. And we adopt three different distance metrics for our proposals: the Euclidean distance, the distance metric learned by the Local Linear Embedding (\cite{Roweis2000}), the distance metric learned by the Isomap approach
(\cite{Tenenbaum2000}). Similar to the conclusion drawn from the application to image digits $\{0,8,9\}$, we again find that our proposals provide valid and useful information for classification as seen from Figure 2 and 3, especially for the nonlinear approach in combination with Isomap.

\clearpage
\subsection{Structural Dimension Determination}
\begin{figure}[htbp]
    \setlength{\abovecaptionskip}{0.cm}
    \centering
    \caption{The vertical axis in the panel (a) and (b) represents a combination of the measures about eigenvalues and eigenvectors, $g_n(k)$, for digits groups: \{0,8,9\} and \{1,4,7\}, respectively.}
    \subfigure[The ladle plot]{
        \includegraphics[width=0.4\textwidth]{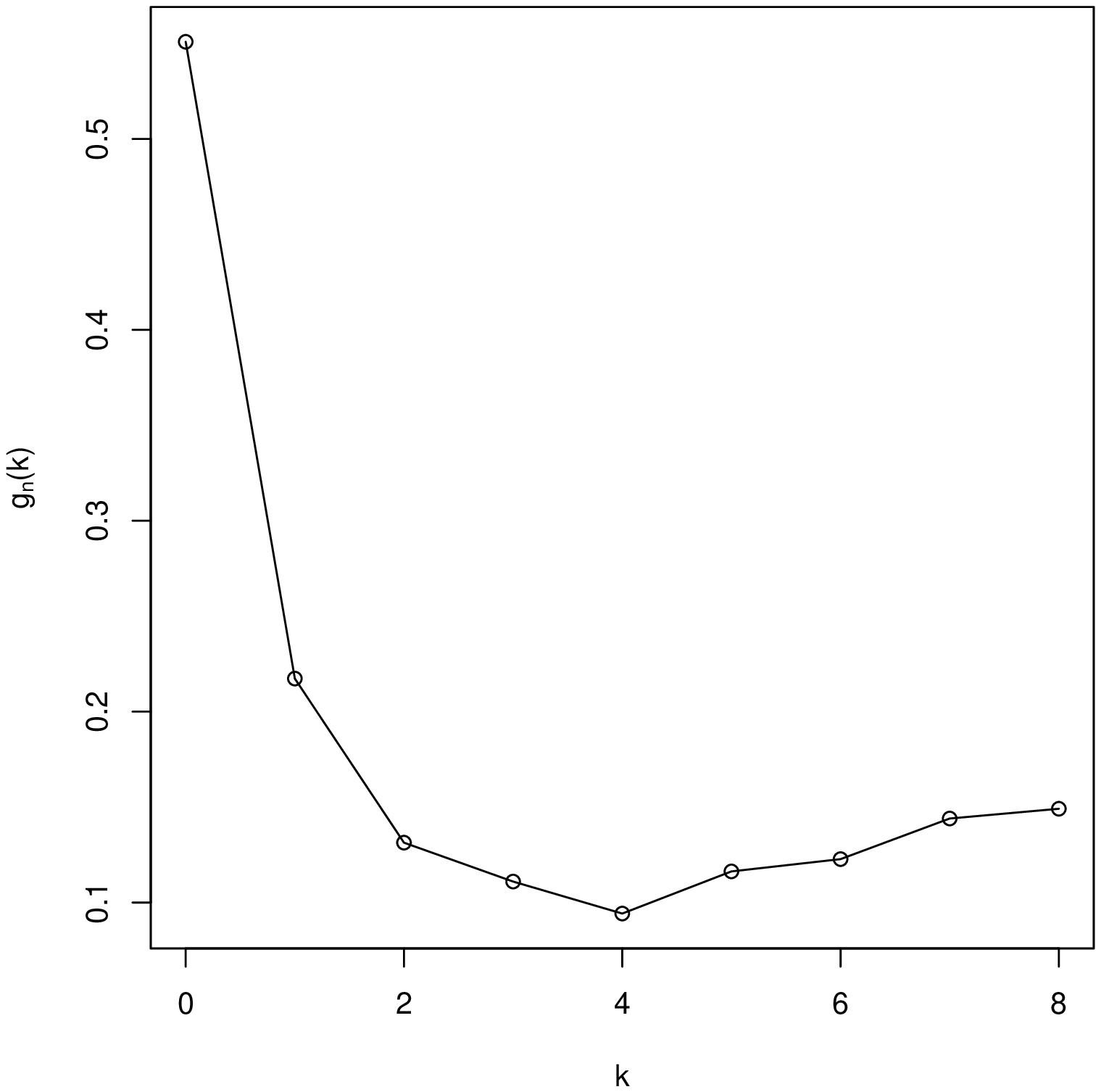}
    }
    \subfigure[The ladle plot]{
        \includegraphics[width=0.4\textwidth]{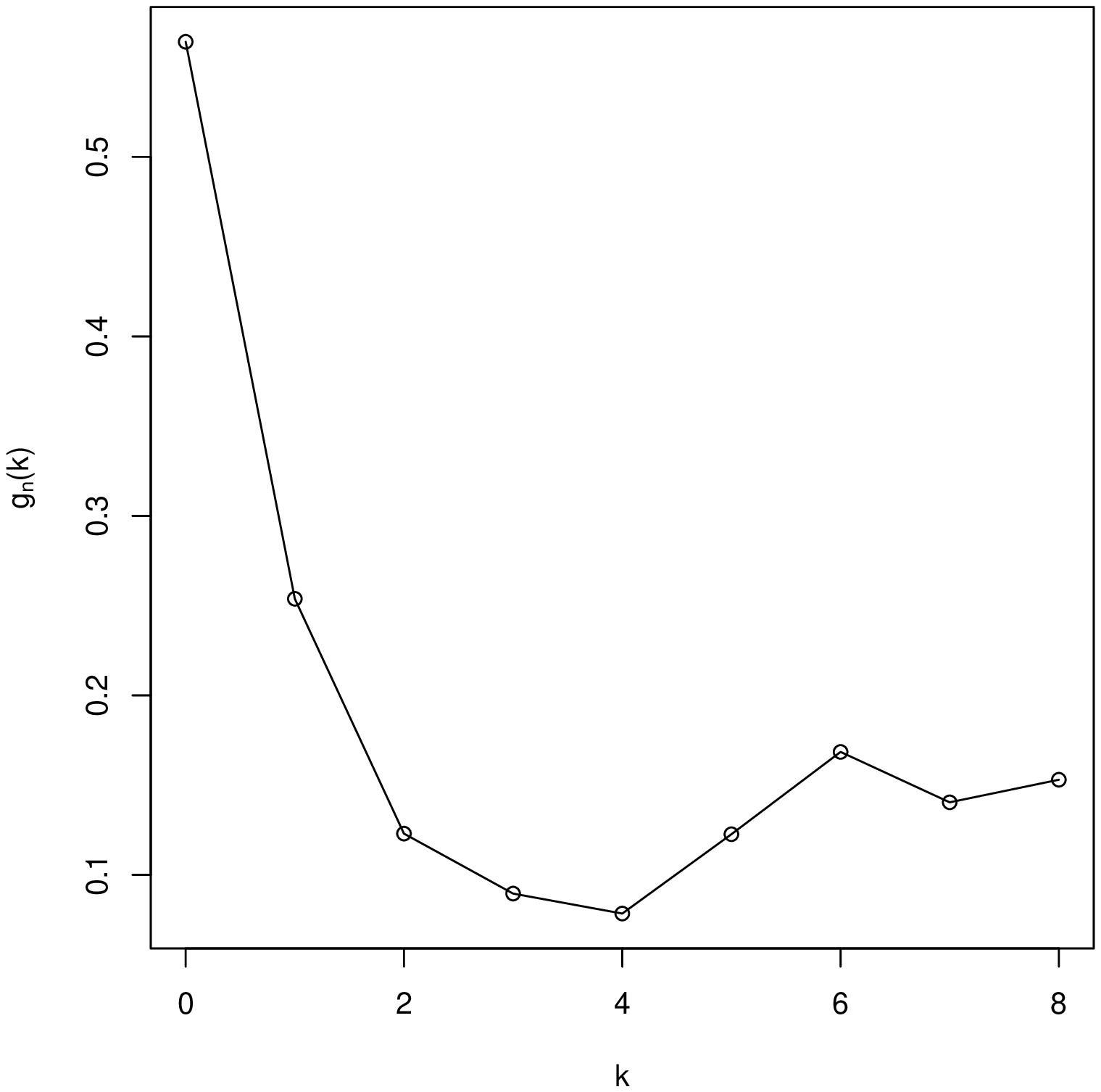}
    }
\end{figure}

For the handwritten digits data, we apply the ladle estimator with the distance metric learned by the Isomap method. Figure 4 displays the ladle plot for digits $\{0,8,9\}$ and $\{1,4,7\}$ respectively. We find that in both cases the ladle estimator yields $\hat{d} = 3$ or $4$. And from the scatter plots accumulated, we know that the first two sufficient predictors already provide useful information for image digits separation.

\clearpage

\section{Proofs of Theoretical Results}\label{SM}
\subsection{Proof of Proposition~1}
\begin{proposition}\label{property M1}  $\Lambda$ is positive semidefinite. Assume the linearity condition holds true, then
	$$\textup{Span}\left\{\Sigma^{-1}\Lambda \right\}\subseteq \mathcal{S}_{Y|X}.$$
\end{proposition}

\begin{proof}
 To prove the proposition, we introduction a fact that exists a separable $\mathbb{R}$-Hilbert space $\mathcal{H}$ and a mapping $\phi: \Omega\rightarrow \mathcal{H}$ such that $\forall Y,Y'\in \mathcal{H}$, $d(Y,Y')=\|\phi(Y)-\phi(Y')\|^2_{\mathcal{H}}$, as shown by \cite{Schoenberg1937,Schoenberg1938}.
 Let $\beta_{\phi}(\mu)=E\phi(Y)$. For $\forall a\in \mathbb{R}^p, a\neq (0,\ldots,0)^T$, we have
\begin{eqnarray*}
a^T\Lambda a
&=&-E\{\langle a^T(X-EX),a^T(X'-EX)\rangle d(Y,Y')\}\\
&=&-E\{\langle a^T(X-EX),a^T(X'-EX)\rangle \|\phi(Y)-\phi(Y')\|^2\}\\
&=&-E\{\langle a^T(X-EX),a^T(X'-EX)\rangle \langle \phi(Y)-\phi(Y'),\phi(Y)-\phi(Y')\rangle\}  \\
&=&-E\{\langle a^T(X-EX),a^T(X'-EX)\rangle \langle \phi(Y)-\beta_{\phi}(\mu)+\beta_{\phi}(\mu)-\phi(Y'),\\
&&\phi(Y)-\beta_{\phi}(\mu)+\beta_{\phi}(\mu)-\phi(Y')\rangle \} \\
&=&2E\{\langle a^T(X-EX),a^T(X'-EX)\rangle \langle \phi(Y)-\beta_{\phi}(\mu),\phi(Y')-\beta_{\phi}(\mu)\rangle\}  \\
&=&2\{E[a^T(X-EX)\otimes (\phi(Y)-\beta_{\phi}(\mu)]\}^2\geq 0
\end{eqnarray*}
Therefore, $\Lambda$ is a semidefined matrix.
By double expectation, we have
\begin{eqnarray*}
&&-\Sigma^{-1}E((X-EX)(X'-EX)^T d(Y,Y')) \\
&=&-\Sigma^{-1}E(E(X-E(X|Y))E(X'-E(X|Y'))^T d(Y,Y')) ,
\end{eqnarray*}
By the  property of SIR, we have $\Sigma^{-1}E(X-E(X|Y)) \in \mathcal{S}_{Y|X}$. The proof is completed.
\end{proof}

\subsection{Proof of Theorem~1}

{\theorem Assume the linearity condition and the singular values $\lambda_\ell$'s are distinct for $\ell=1,\ldots,d$. In addition, assume that $Ed^2(Y,Y') < \infty$ and $X$ has finite fourth moment, then
	\begin{eqnarray}
	n^{1/2}(\hat\beta_\ell-\beta_\ell)\overset{D}{\longrightarrow} N\left({0}_p, \Sigma_\ell\right),
	\end{eqnarray}
as $n\rightarrow \infty$, where $\Sigma_\ell= cov\{\Upsilon_\ell(X,Y)\} $.}

The following lemmas are needed before we prove Theorem 1. Let $E(X)=\mu,\overline{X}=n^{-1}\sum_{i=1}^{n}X_{i}$ and $\widehat{\Sigma}=n^{-1}\sum_{i=1}^{n}(X_{i}-\overline{X})(X_{i}-\overline{X})^T$. Lemma \ref{lemma M1} provides the asymptotic expansion of $\widehat{\Sigma}$. Its proof is obvious and thus omitted.

{\lemma\label{lemma M1} Assume $X$ has finite fourth moment. Then
\begin{align}\label{M 29}
\widehat{\Sigma}-\Sigma=\frac{1}{n}\sum\limits_{i=1}^{n}\Gamma(X_{i})+o_P(n^{-1/2}),
\end{align}
where $\Gamma(X_{i})=(X_{i}-\mu)(X_{i}-\mu)^T-\Sigma$.
}
\vspace{2em}

Let $\Lambda = -E\{(X-\mu)(X'-\mu)^T d(Y,Y')\}$ and
\begin{align}\label{M 30}
 \widehat{\Lambda} =-\frac{1}{n(n-1)}\sum_{1\leq i\neq j\leq n} (X_i-\overline{X})(X_j-\overline{X})^T d(Y_i,Y_j).
 \end{align}
Lemma \ref{lemma M2} provides the asymptotic expansion of $\widehat{\Lambda}$.
{\lemma\label{lemma M2} Assume $X$ has finite fourth moment. Then
\begin{align}\label{M 31}
\widehat{\Lambda}-\Lambda=\frac{1}{n}\sum\limits_{i=1}^{n}\Theta(X_{i},Y_{i})+o_P(n^{-1/2}),
\end{align}
where the exact form of $\Theta(X_{i},Y_{i})$ is provided in the proof.
}

\vspace{1em}

\begin{proof}
First we decompose $\widehat{\Lambda}$ in \eqref{M 30} as $\widehat{\Lambda}=\widehat{U}^{(1)}+\widehat{U}^{(2)}+\widehat{U}^{(3)}+\widehat{U}^{(4)}$, where
\begin{equation}\label{M 32}
\begin{aligned}
\widehat{U}^{(1)}&=-\frac{1}{n(n-1)}\sum_{i \neq j} (X_i-\mu)(X_j-\mu)^T d(Y_i,Y_j),\\
\widehat{U}^{(2)}&=\frac{1}{n(n-1)}\sum_{i\neq j} (\widehat{\mu}-\mu)(X_j-\mu)^T d(Y_i,Y_j), \\
\widehat{U}^{(3)}&=\frac{1}{n(n-1)}\sum_{i \neq j} (X_i-\mu)(\widehat{\mu}-\mu)^T d(Y_i,Y_j),\text{ and}\\
\widehat{U}^{(4)}&=-\frac{1}{n(n-1)}\sum_{i\neq j} (\widehat{\mu}-\mu)(\widehat{\mu}-\mu)^T d(Y_i,Y_j).
\end{aligned}
\end{equation}

\setlength{\parindent}{0pt}Let $\Lambda^{(1)}(X_{i},Y_{i},X_{j},Y_{j})=-(X_i-\mu)(X_j-\mu)^T d(Y_i,Y_j)$ and denote $\Lambda^{(1)}_{1}(x,y)=E\{\Lambda^{(1)}(X,Y,x,y)\}$. For the first term, we have
\begin{equation}\label{M 33}
\begin{aligned}
\widehat{U}^{(1)}-\Lambda&=\frac{1}{n(n-1)}\sum_{i\neq j}\{\Lambda^{(1)}(X_i,Y_i,X_j,Y_j)-\Lambda\}\\
&=\frac{1}{n}\sum\limits_{i=1}^{n}\{\Lambda^{(1)}_{1}(X_{i},Y_{i})-\Lambda\}+\frac{1}{n(n-1)}\sum_{i\neq j}A(X_{i},Y_{i})
\end{aligned}
\end{equation}
where
$$
A(X_{i},Y_{i},X_{j},Y_{j})=\{\Lambda^{(1)}(X_i,Y_i,X_j,Y_j)-\Lambda^{(1)}_{1}(X_{i},Y_{i})\}.
$$
By simple calculation,
\begin{eqnarray*}
&&E\|\frac{1}{n(n-1)}\sum_{i\neq j}A(X_{i},Y_{i})\|^2_{F}\\
&=&tr(E(\frac{1}{n(n-1)}\sum_{i\neq j}A(X_{i},Y_{i},X_{j},Y_{j}))(\frac{1}{n(n-1)}\sum_{k\neq t}A(X_{k},Y_{k},X_{t},Y_{t})))\\
&=&\frac{n(n-1)(n-2)(n-3)}{n^2(n-1)^2}tr(E(A(X,Y,X',Y')A(X'',Y'',X''',Y''')))\\
&&+\frac{n(n-1)(n-2)}{n^2(n-1)^2}tr(E(A(X,Y,X',Y')A(X,Y,X'',Y'')))\\
&&+\frac{n(n-1)}{n^2(n-1)^2}tr(E(A(X,Y,X',Y')A(X,Y,X',Y')))\\
&=&\frac{n(n-1)(n-2)}{n^2(n-1)^2}tr(E(E(A(X,Y,X',Y')|X,Y)E(A(X,Y,X'',Y'')|X,Y)))\\
&&+\frac{1}{n(n-1)}tr(E(A(X,Y,X',Y')A(X,Y,X',Y')))\\
&=&\frac{1}{n(n-1)}tr(E(A(X,Y,X',Y')A(X,Y,X',Y')))=O(n^{-2})
\end{eqnarray*}
we get $\frac{1}{n(n-1)}\sum_{i\neq j}A(X_{i},Y_{i})=O_p(n^{-1})$.

Let $\vartheta=E\{(X-\mu)d(Y,Y')\}$. Note that
\begin{align*}
  \frac{1}{n(n-1)}\sum\limits_{i\neq j}(X_{j}-\mu)d(Y_i,Y_j)\overset{P}{\longrightarrow}\vartheta.
\end{align*}
It follows that
\begin{align}\label{M 34}
\widehat{U}^{(2)}=\frac{1}{n}\sum_{i=1}^n (X_{i}-\mu)\vartheta^T+o_P(n^{-1/2}).
\end{align}
Similarly we have
\begin{align}\label{M 35}
\widehat{U}^{(3)}=\frac{1}{n}\sum_{i=1}^n \vartheta(X_{i}-\mu)^T+o_P(n^{-1/2}).
\end{align}
Note that $\widehat{U}^{(4)}=o_P(n^{-1/2})$. (\ref{M 33}), (\ref{M 34}) and (\ref{M 35}) together lead to \eqref{M 31}, where $\Theta(X_{i},Y_{i})=\Lambda^{(1)}_{1}(X_{i},Y_{i})-\Lambda+(X_{i}-\mu)\vartheta^T +\vartheta(X_{i}-\mu)^T$.

By algebra calculations, we have
\begin{eqnarray*}
\widehat{M}-M&=&\widehat{\Sigma}^{-1}\widehat{\Lambda}-\Sigma^{-1}\Lambda=(\widehat{\Sigma}^{-1}-\Sigma^{-1})\Lambda+\Sigma^{-1}(\widehat{\Lambda}-\Lambda)+O_p(n^{-1})\\
&=&-\Sigma^{-1}(\widehat{\Sigma}-\Sigma)\Sigma^{-1}\Lambda+\Sigma^{-1}(\widehat{\Lambda}-\Lambda)+o_p(n^{-1/2})\\
&=&-\frac{1}{n}\sum_{i=1}^n\Sigma^{-1}((X_i-\mu)(X_i-\mu)^T-\Sigma)\Sigma^{-1}\Lambda\\
&&+\frac{1}{n(n-1)}\sum_{i\neq j}\Sigma^{-1}((X_i-\mu)(X_j-\mu)^T d(Y_i,Y_j)-\Lambda)\\
&&+\Sigma^{-1}(\mu-\bar{X})\vartheta^T+\Sigma^{-1}\vartheta(\mu-\bar{X})^T+o_p(n^{-1/2})\\
&=&-\frac{1}{n}\sum_{i=1}^n\Sigma^{-1}((X_i-\mu)(X_i-\mu)^T-\Sigma)\Sigma^{-1}\Lambda\\
&&+\frac{1}{n}\sum_{i=1}^n\Sigma^{-1}(\Lambda^{(1)}_1(X_i,Y_i)-\Lambda)\\
&&+\Sigma^{-1}(\mu-\bar{X})\vartheta^T+\Sigma^{-1}\vartheta(\mu-\bar{X})^T+o_p(n^{-1/2})\\
&=&\frac{1}{n}\sum_{i=1}^nH(X_i,Y_i)+o_p(n^{-1/2})
\end{eqnarray*}
where
\begin{eqnarray}
H(X,Y)&=&-\Sigma^{-1}((X-\mu)(X-\mu)^T\\ \nonumber
&&-\Sigma)\Sigma^{-1}\Lambda+\Sigma^{-1}(\Lambda^{(1)}_1(X,Y)-\Lambda)\\ \nonumber
&&+\Sigma^{-1}(\mu-X)\vartheta^T+\Sigma^{-1}\vartheta(\mu-X)^T \\ \nonumber
\end{eqnarray}

Simple calculations lead to
\begin{eqnarray*}
\widehat{M}\widehat{M}^T-MM^T&=&\Sigma^{-1}(\widehat{\Lambda}-\Lambda)\Lambda^T\Sigma^{-1}
+(\widehat{\Sigma}^{-1}-\Sigma^{-1})\Lambda\Lambda^T\Sigma^{-1}\\ \nonumber
&&+\Sigma^{-1}\Lambda\Lambda^T(\widehat{\Sigma}^{-1}-\Sigma^{-1})
+\Sigma^{-1}\Lambda(\widehat{\Lambda}-\Lambda)^T\Sigma^{-1}
\end{eqnarray*}

Observe that $\lambda_\ell$ and $\beta_\ell$ satisfy the following singular value decomposition equation:
\begin{eqnarray*}
MM^T\beta_\ell=\lambda_\ell^2\beta_\ell, \text{ and}\quad \ell=1,\ldots,p,
\end{eqnarray*}
Hence,
\begin{eqnarray*}
\Sigma^{-1}\Lambda\Lambda^T\Sigma^{-1}\beta_\ell=\lambda_\ell^2\beta_\ell, \text{ and}\quad \ell=1,\ldots,p,
\end{eqnarray*}
where $\beta_\ell^T\beta_\ell=1$ and $\beta_\ell^T\beta_\jmath=0$ for $\ell\neq \jmath $. Similarly, in the sample level,
we have
\begin{eqnarray*}
\widehat{\Sigma}^{-1}\widehat{\Lambda}\widehat{\Lambda}^T\widehat{\Sigma}^{-1}\beta_\ell=\lambda_\ell^2\beta_\ell, \text{ and}\quad \ell=1,\ldots,p;
\end{eqnarray*}
where $\beta_\ell^T\beta_\ell=1$ and $\beta_\ell^T\beta_\jmath=0$ for $\ell\neq \jmath $. The singular value decomposition
form in the sample level implies that
\begin{eqnarray}\label{M 50}
&&\Sigma^{-1}(\widehat{\Lambda}-\Lambda)\Lambda^T\Sigma^{-1}\beta_\ell+(\widehat{\Sigma}^{-1}-\Sigma^{-1})\Lambda\Lambda^T\Sigma^{-1}\beta_\ell\\ \nonumber
&+&\Sigma^{-1}\Lambda\Lambda^T(\widehat{\Sigma}^{-1}-\Sigma^{-1})\beta_\ell
+\Sigma^{-1}\Lambda(\widehat{\Lambda}-\Lambda)^T\Sigma^{-1}\beta_\ell
+\Sigma^{-1}\Lambda\Lambda^T\Sigma^{-1}(\widehat{\beta}_\ell-\beta_\ell)\\ \nonumber
&=&\lambda_\ell(\widehat{\lambda}_\ell-\lambda_\ell)\beta_\ell
+(\widehat{\lambda}_\ell-\lambda_\ell)\lambda_\ell\beta_\ell+\lambda_{\ell}^2(\widehat{\beta}_\ell-\beta_\ell)+o_p(n^{-1/2})
\end{eqnarray}
for $\ell=1,\ldots,p$. Multiply both sides of \eqref{M 50} by $\beta^T_\ell$, we get from the left
\begin{eqnarray*}
&&\beta^T_\ell[\Sigma^{-1}(\widehat{\Lambda}-\Lambda) \Lambda^T\Sigma^{-1}+(\widehat{\Sigma}^{-1}-\Sigma^{-1})\Lambda\Lambda^T\Sigma^{-1}
+\Sigma^{-1}\Lambda\Lambda^T(\widehat{\Sigma}^{-1}-\Sigma^{-1})\\
&+&\Sigma^{-1}\Lambda(\widehat{\Lambda}-\Lambda)^T\Sigma^{-1}]\beta_\ell
=\lambda_\ell(\widehat{\lambda}_\ell-\lambda_\ell)+(\widehat{\lambda}_\ell-\lambda_\ell)\lambda_k+o_p(n^{-1/2})
\end{eqnarray*}
which further suggests that
\begin{eqnarray}\label{M 51}
\widehat{\lambda}_\ell&=&\lambda_\ell+\frac{\beta^T_\ell}{2\lambda_\ell}[\Sigma^{-1}(\widehat{\Lambda}-\Lambda)\Lambda^T\Sigma^{-1}
+(\widehat{\Sigma}^{-1}-\Sigma^{-1})\Lambda\Lambda^T\Sigma^{-1}\\ \nonumber
&+&\Sigma^{-1}\Lambda\Lambda^T(\widehat{\Sigma}^{-1}-\Sigma^{-1})
+\Sigma^{-1}\Lambda(\widehat{\Lambda}-\Lambda)^T\Sigma^{-1}]\beta_\ell+o_p(n^{-1/2})
\end{eqnarray}
By lemma A.2 of \cite{Cook2005}, we know that
\begin{eqnarray*}
\widehat{\Sigma}^{-1}-\Sigma^{-1}=-\Sigma^{-1}(\widehat{\Sigma}-\Sigma)\Sigma^{-1}+o_p(n^{-1/2})
\end{eqnarray*}
Hence, equation \eqref{M 51} becomes
\begin{eqnarray*}
\widehat{\lambda}_\ell&=&\lambda_\ell+\frac{\beta^T_\ell}{2\lambda_\ell}[\Sigma^{-1}(\widehat{\Lambda}-\Lambda)\Lambda^T\Sigma^{-1}
-\Sigma^{-1}(\widehat{\Sigma}-\Sigma)\Sigma^{-1}\Lambda\Lambda^T\Sigma^{-1}\\ \nonumber
&-&\Sigma^{-1}\Lambda\Lambda^T\Sigma^{-1}(\widehat{\Sigma}-\Sigma)\Sigma^{-1}
+\Sigma^{-1}\Lambda(\widehat{\Lambda}-\Lambda)^T\Sigma^{-1}]\beta_\ell+o_p(n^{-1/2})\\
&=&\lambda_\ell+\frac{1}{n}\sum_{i=1}^nC_{i,\lambda_\ell}+o_p(n^{-1/2})
\end{eqnarray*}
where
\begin{eqnarray*}
C_{i,\lambda_\ell}&=&\frac{\beta^T_\ell}{2\lambda_\ell}[\Sigma^{-1}\Theta(X_i,Y_i)\Lambda^T\Sigma^{-1}
-\Sigma^{-1}\Gamma(X_{i})\Sigma^{-1}\Lambda\Lambda^T\Sigma^{-1}\\ \nonumber
&-&\Sigma^{-1}\Lambda\Lambda^T\Sigma^{-1}\Gamma(X_{i})\Sigma^{-1}
+\Sigma^{-1}\Lambda\Theta(X_i,Y_i)\Sigma^{-1}]\beta_\ell+o_p(n^{-1/2})\\
\end{eqnarray*}
Now we return to the expansion of $\hat{\beta}_\ell$. Since $(\beta_1,\ldots,\beta_p)$ is a basis if $R^p$, there exists $c_{\ell j}$ for $j=1,\ldots,p$, such that $\widehat{\beta}_\ell-\beta_\ell=\sum_{j=1}^pc_{\ell j}\beta_j$ and $c_{\ell j}=O_p(n^{-1/2})$. We will derive the explicit form of $c_{\ell j}$ in the next step. Note that \eqref{M 50} can be rewritten as
\begin{eqnarray}\label{M 52}
&&(\Sigma^{-1}\Lambda\Lambda^T\Sigma^{-1}-\lambda_\ell^2)\sum_{j=1}^pc_{\ell j}\beta_j\\  \nonumber
&=&\lambda_\ell(\widehat{\lambda}_\ell-\lambda_\ell)\beta_\ell+(\widehat{\lambda}_\ell-\lambda_\ell)\lambda_\ell\beta_\ell
+[\Sigma^{-1}(\widehat{\Lambda}-\Lambda)\Lambda^T\Sigma^{-1}\\  \nonumber
&+&(\widehat{\Sigma}^{-1}-\Sigma^{-1})\Lambda\Lambda^T\Sigma^{-1}+\Sigma^{-1}\Lambda\Lambda^T(\widehat{\Sigma}^{-1}-\Sigma^{-1})
+\Sigma^{-1}\Lambda(\widehat{\Lambda}-\Lambda)^T\Sigma^{-1}]\beta_\ell\\  \nonumber
&=&\lambda_\ell(\widehat{\lambda}_\ell-\lambda_\ell)\beta_\ell+(\widehat{\lambda}_\ell-\lambda_\ell)\lambda_\ell\beta_\ell
+[\Sigma^{-1}(\widehat{\Lambda}-\Lambda)\Lambda^T\Sigma^{-1}\\ \nonumber
&-&\Sigma^{-1}(\widehat{\Sigma}-\Sigma)\Sigma^{-1}\Lambda\Lambda^T\Sigma^{-1}-\Sigma^{-1}\Lambda\Lambda^T\Sigma^{-1}(\widehat{\Sigma}-\Sigma)\Sigma^{-1}
+\Sigma^{-1}\Lambda(\widehat{\Lambda}-\Lambda)^T\Sigma^{-1}]\beta_\ell\\ \nonumber
&=&\lambda_\ell(\widehat{\lambda}_\ell-\lambda_\ell)\beta_\ell+(\widehat{\lambda}_\ell-\lambda_\ell)\lambda_\ell\beta_\ell
+\frac{1}{n}\sum_{1=1}^n\zeta_{\ell}(X_i,Y_i) \beta_\ell
\end{eqnarray}
where
\begin{eqnarray*}
\zeta_{\ell}(X_i,Y_i)&=&\Sigma^{-1}\Theta(X_i,Y_i)\Lambda\Sigma^{-1}
-\Sigma^{-1}\Gamma(X_{i})\Sigma^{-1}\Lambda\Lambda^T\Sigma^{-1}\\ \nonumber
&-&\Sigma^{-1}\Lambda\Lambda^T\Sigma^{-1}\Gamma(X_{i})\Sigma^{-1}
+\Sigma^{-1}\Lambda\Theta(X_i,Y_i)\Sigma^{-1}
\end{eqnarray*}
Multiply both sides of \eqref{M 52} by $\beta_j^T \ (j\neq \ell)$, we can get from the left
\begin{eqnarray}\label{M 53}
c_{\ell,j}=\frac{1}{n}\sum_{i=1}\frac{\beta_j^T\zeta_{\ell}(X_i,Y_i) \beta_\ell}{\lambda_j^2-\lambda_\ell^2}, \ j\neq \ell;
\end{eqnarray}
In addition, $\beta_\ell^T\beta_\ell=\widehat{\beta}_\ell^T\widehat{\beta}_\ell=1$ indicates that
$$
0={\sum_{j=1}^pc_{\ell j}\beta^T_j}\beta_\ell+\beta_\ell^T{\sum_{j=1}^pc_{\ell j}\beta_j},
$$
which further implies that $c_{\ell\ell}=0$. We define
\begin{equation}\label{sd}
\Sigma_\ell=\rm{cov}(\Upsilon_\ell(X,Y)),
\end{equation}
 where $p\times 1$ random vector
$
\Upsilon_\ell(X,Y)=\sum_{j=1,j\neq \ell}^{p}\frac{\beta_j\beta_j^T\zeta_{\ell}(X,Y)\beta_\ell}{\lambda_j^2-\lambda_\ell^2}.
$
Then plug \eqref{M 53} and \eqref{M 29} into \eqref{M 31}, and we get
\begin{eqnarray}
\widehat{\beta}_\ell=\beta_\ell+\frac{1}{n}\sum_{i=1}^n\Upsilon_\ell(X_i,Y_i)+o_p(n^{-1/2})
\end{eqnarray}
The conclusion is then straightforward via the central limit theorem.
\end{proof}

\subsection{Proof of Theorem~2}
\begin{theorem}\label{theo: ladle} Assume $Ed^2(Y,Y') < \infty$ and $X$ has finite fourth moment. And suppose Assumptions (1)--(2) hold, then
	\begin{align*}
	P_r\{\lim_{n\rightarrow\infty} P_r(\hat d =d |\mathcal{D}) {=1}\} =1,
	\end{align*}
	where $\mathcal{D}=\{(X_1,Y_1),(X_2,Y_2),\ldots\}$ is a sequence of independent copies of $(X,Y)$.
\end{theorem}

\begin{proof} The singular value of  $\widehat M$  are square root of the corresponding eigenvalue of matrix $\widehat M\widehat M^T$. Moreover,  the left singular vectors are the same as the eigenvectors of $\widehat M\widehat M^T$. Then we apply Theorem 2 in \cite{Luo2016} to get the desired result.
\end{proof}

\subsection{Proof of Proposition~2}
\begin{proposition} $\Lambda_{XX'}$ is a bounded linear and self-adjoint operator. For any $f,g \in \mathcal{H}_X$,
\begin{align*}
\langle f,\Lambda_{XX'}g\rangle_{\mathcal{H}_X}=-E\{(f(X)-Ef(X)) (g(X')-E{g(X')})d(Y,Y')\}.
\end{align*}
Moreover, there exists a separable $\mathbb{R}$-Hilbert space $\mathcal{H}$ and a mapping $\phi: \Omega\rightarrow \mathcal{H}$ such that
\begin{align*}
\langle f,\Lambda_{XX'}f\rangle_{\mathcal{H}_X}=2\{E[(f(X)-Ef(X)) (\phi(Y)-E\phi(Y))]\}^2=2(cov[f(X),\phi(Y)])^2,
\end{align*}
\end{proposition}

\begin{proof}
For arbitrary $f,g\in \mathcal{H}_X$, we have
\begin{eqnarray*}
|\langle f,\Lambda_{XX'}g\rangle_{\mathcal{H}_X}|&\leq&E|\langle f,((\kappa_{X}(\cdot,X)-\mu_X)\otimes (\kappa_{X}(\cdot,X')-\mu_X)d(Y,Y'))g\rangle_{\mathcal{H}_X}|\\
&=&E\{|\langle f,\kappa_{X}(\cdot,X)-\mu_X\rangle_{\mathcal{H}_X}||\langle \kappa_{X}(\cdot,X')-\mu_X, g\rangle_{\mathcal{H}_X}|d(Y,Y')\}\\
&\leq&\|f\|_{\mathcal{H}_X}\|g\|_{\mathcal{H}_X}E\langle \kappa_{X}(\cdot,X)-\mu_X, \kappa_{X}(\cdot,X)-\mu_X\rangle_{\mathcal{H}_X}(Ed^2(Y,Y'))^{1/2}\\
&=&\|f\|_{\mathcal{H}_X}\|g\|_{\mathcal{H}_X}(E\kappa_{X}(X,X)-\mu_X^2)(Ed^2(Y,Y'))^{1/2}
\end{eqnarray*}
Since
$$
\mu_X^2\leq (E\|\kappa_{X}(\cdot,X)\|_{\mathcal{H}_X})^{2}=(E\kappa_{X}(X,X)^{1/2})^{2}\leq E\kappa_{X}(X,X)<\infty,
$$
and $Ed^2(Y,Y')<\infty$. Therefore, $\Lambda_{XX'}$ is a bounded liner and self-adjoint operator.
\begin{eqnarray*}
\langle f,\Lambda_{XX'}g\rangle_{\mathcal{H}_X}&=&-E\langle f,((\kappa_{X}(\cdot,X)-\mu_X)\otimes (\kappa_{X}(\cdot,X')-\mu_X)d(Y,Y'))g\rangle_{\mathcal{H}_X}\\
&=&-E\{\langle f,\kappa_{X}(\cdot,X)-\mu_X\rangle_{\mathcal{H}_X}\langle \kappa_{X}(\cdot,X')-\mu_X, g\rangle_{\mathcal{H}_X}d(Y,Y')\}\\
&=&-E\{(f(X)-Ef(X))(g(X')-E{g(X')})d(Y,Y')\}
\end{eqnarray*}

For arbitrary $f\in \mathcal{H}_X$, we have
\begin{eqnarray*}
\langle f,\Lambda_{XX'}f\rangle_{\mathcal{H}_X}&=&-\langle f,E((\kappa_{X}(\cdot,X)-\mu_X)\otimes (\kappa_{X}(\cdot,X')-\mu_X)d(Y,Y'))f\rangle_{\mathcal{H}_X}\\
&=&-E\{\langle f,((\kappa_{X}(\cdot,X)-\mu_X)\otimes (\kappa_{X}(\cdot,X')-\mu_X))f\rangle_{\mathcal{H}_X} \|\phi(Y)-\phi(Y')\|_{\mathcal{H}}^2\}\\
&=&2E\{\langle f(X)-Ef(X),f(X')-Ef(X')\rangle_{\mathcal{H}_X} \langle \phi(Y)-\beta_{\phi}(\mu), \phi(Y')-\beta_{\phi}(\mu)\rangle_\mathcal{H}\}\\
&=&2\{E(f(X)-Ef(X))\otimes (\phi(Y)-\beta_{\phi}(\mu)) \}^2\geq 0
\end{eqnarray*}
Therefore, $\Lambda_{XX'}$ is a semidefined operator.
\end{proof}

\subsection{Proof of Proposition~3}
\begin{proposition}\label{property M2}  Suppose assumptions (3)--(5) hold, then
$$\overline{\textup{ran}}\left\{\Sigma_{XX}^{-1}\Lambda_{XX'}\right\}\subseteq \mathcal{G}_{Y|X}.$$
\end{proposition}
\begin{proof}%[of Proposition~3].%\ref{property M2}
 Firstly, we show that
\begin{equation*}
\overline{\rm{ran}}(\Lambda_{XX'})\subseteq \Sigma_{XX}\mathcal{G}_{Y|X}
\end{equation*}
which is equivalent to
\begin{equation*}
(\Sigma_{XX}\mathcal{G}_{Y|X})^{\bot}\subseteq \overline{\rm{ran}}(\Lambda_{XX'})^{\bot}.
\end{equation*}
Since $\overline{\rm{ran}}(\Lambda_{XX'})^{\bot}= \ker (\Lambda_{XX'})$, where $\ker(\Lambda_{XX'})$ denotes nuclear space generated by the operator $\Lambda_{XX'}$, it suffices to show that
\begin{equation*}
(\Sigma_{XX}\mathcal{G}_{Y|X})^{\bot}\subseteq \ker (\Lambda_{XX'}).
\end{equation*}
Now we define $\mathcal{G}_{\phi(Y)|X}$, we get
\begin{equation}\label{3}
(\Sigma_{XX}\mathcal{G}_{\phi(Y)|X})^{\bot}\subseteq \ker (\Lambda_{XX'}).
\end{equation}
Let $f\in (\Sigma_{XX}\mathcal{G}_{\phi(Y)|X})^{\bot}$. Then, for all $g\in \mathcal{G}_{\phi(Y)|X}$, we have
\begin{equation*}
\langle f,\Sigma_{XX}g \rangle_{\mathcal{H}_X}=cov\{f(X),g(X)\}=0.
\end{equation*}
Because $g$ is measurable with respect to $\mathcal{M}_{\phi(Y)|X}$, we have $g(X)=E[g(X)|\mathcal{M}_{\phi(Y)|X}]$. And
\begin{eqnarray*}
&&cov\{f(X),E[g(X)|\mathcal{M}_{\phi(Y)|X}]\}\\
&=&E[f(X)E[g(X)|\mathcal{M}_{\phi(Y)|X}]]-E[f(X)]E[g(X)]\\
&=&E[E[f(X)|\mathcal{M}_{\phi(Y)|X}]g(X)]-E[f(X)]E[g(X)]\\
&=&cov\{E[f(X)|\mathcal{M}_{\phi(Y)|X}],g(X)\}
\end{eqnarray*}
The second equation is based on the property of double expectation. Since $\mathcal{G}_{\phi(Y)|X}$ is dense in $L_2(P_X|\mathcal{M}_{\phi(Y)|X})$ modulo constants, there exists a sequence $\{f_n \subseteq \mathcal{G}_{\phi(Y)|X}\}$ such that $var[f_n(X)-f(X)]\rightarrow 0$. Then
\begin{equation}\label{1}
cov\{E[f(X)|\mathcal{M}_{\phi(Y)|X}], f_n(X)\}
=E\{E[f(X)|\mathcal{M}_{\phi(Y)|X}]f_n(X)\}-E[f(X)]E[f_n(X)]
=0
\end{equation}
On the other hand,
\begin{eqnarray}\label{2}
cov\{E[f(X)|\mathcal{M}_{\phi(Y)|X}], f_n(X)\}&\rightarrow&cov\{E[f(X)|\mathcal{M}_{\phi(Y)|X}], f(X)\}\\
&=&cov\{E[f(X)|\mathcal{M}_{\phi(Y)|X}], E[f(X)|\mathcal{M}_{\phi(Y)|X}]\}\nonumber
\end{eqnarray}
Combining \eqref{1} and \eqref{2}, we have
\begin{eqnarray*}
var\{E[f(X)|\mathcal{M}_{\phi(Y)|X}]\}=0
\end{eqnarray*}
This implies that $E[f(X)|\mathcal{M}_{\phi(Y)|X}]=$ constant almost surely. Since $\mathcal{M}_{\phi(Y)|X}$ is sufficient, we have $E[f(X)|\mathcal{M}_{\phi(Y)|X}]=E[f(X)|\phi(Y),\mathcal{M}_{\phi(Y)|X}]$. So  $E[f(X)|\phi(Y),\mathcal{M}_{\phi(Y)|X}]=$ constant almost surely. Consequently, $E[f(X)|\phi(Y)]=$ constant almost surely.

\begin{eqnarray*}
&&\Sigma^{-1}_{XX}E[(\kappa_X(\cdot,X)-\mu_X)\otimes (\kappa_X(\cdot,X')-\mu_{X})d(Y,Y')]\\
&=&\Sigma^{-1}_{XX}E[(\kappa_X(\cdot,X)-\mu_X)\otimes (\kappa_X(\cdot,X')-\mu_{X})\|\phi(Y)-\phi(Y')\|^2_{\mathcal{H}}]\\
&=&\Sigma^{-1}_{XX}E\{E[(\kappa_X(\cdot,X)-\mu_X)|\phi(Y)]\otimes E[(\kappa_X(\cdot,X')-\mu_{X})|\phi(Y')]\|\phi(Y)-\phi(Y')\|^2_{\mathcal{H}}]\}
\end{eqnarray*}
We can get
\begin{eqnarray*}
E[(\kappa_X(\cdot,X')-\mu_{X})|\phi(Y')]f=E[f(X'))|\phi(Y')]-\mu_{X}(f(X'))=0
\end{eqnarray*}
Therefore, $\Lambda_{XX'}f=0$. Then we have proved \eqref{3}.

By \eqref{3}, we have
\begin{eqnarray*}
\rm{ran}(\Lambda_{XX'})\subseteq \Sigma_{XX}\mathcal{G}_{\phi(Y)|X},
\end{eqnarray*}
which implies $\Sigma^{-1}_{XX}\rm{ran}(\Lambda_{XX'})\subseteq\mathcal{G}_{\phi(Y)|X}$. Note that
\begin{eqnarray*}
\Sigma^{-1}_{XX}\rm{ran}(\Lambda_{XX'})&=&\{\Sigma^{-1}_{XX}f: f=\Lambda_{XX'} g, g\in \mathcal{H}_{\phi(Y)}\}\\
&=&\{\Sigma^{-1}_{XX}\Lambda_{XX'} g:  g\in \mathcal{H}_{\phi(Y)}\}=\rm{ran}(\Sigma^{-1}_{XX}\Lambda_{XX'})
\end{eqnarray*}
Then, because $\mathcal{G}_{\phi(Y)|X}$ is closed, we have  $\overline{\rm{ran}}(\Sigma^{-1}_{XX}\Lambda_{XX'})\subseteq \mathcal{G}_{\phi(Y)|X}.$

Finally, we will show $\mathcal{G}_{\phi(Y)|X}\subseteq \mathcal{G}_{Y|X}$. It is easy to find that
\begin{eqnarray*}
Y\indep X|\mathcal{G}_{Y|X}\Rightarrow \phi(Y)\indep X|\mathcal{G}_{\phi(Y)|X}
\end{eqnarray*}
Therefore, we have $\mathcal{G}_{\phi(Y)|X}\subseteq \mathcal{G}_{Y|X}.$ The proof is completed.
\end{proof}

\subsection{Proof of Proposition~4}
\begin{proposition}\label{property M3}  Suppose assumptions (3)--(5) hold and $\mathcal{G}_{Y|X}$ is complete.
Then,
$$\overline{\textup{ran}}\left\{\Sigma_{XX}^{-1}\Lambda_{XX'}\right\}= \mathcal{G}_{Y|X}.$$
\end{proposition}
\begin{proof}
Form Proposition \ref{property M2}, we know $\overline{\rm{ran}}(\Lambda_{XX'})\subseteq \Sigma_{XX}\mathcal{G}_{Y|X}$. Therefore, we only need to show $\Sigma_{XX}\mathcal{G}_{Y|X} \subseteq \overline{\rm{ran}}(\Lambda_{XX'})$, or equivalently, $\rm{ker}(\Lambda_{X'X})\subseteq (\Sigma_{XX}\mathcal{G}_{Y|X})^{\perp}$. Let $f \in \rm{ker}(\Lambda_{X'X})$. Then $\Lambda_{X'X}f = 0$, which implies that $\Sigma^{-1}_{X'X'}\Lambda_{X'X}f=0$. By the proof of Proposition \ref{property M2}, we have $E(f(X)|\mathcal{M}_{\phi(Y)|X})=\rm{constant}$. Since $\mathcal{M}_{\phi(Y)|X} \subseteq \mathcal{M}_{Y|X}$, we have $E(f(X)|\mathcal{M}_{Y|X})=\rm{constant}$. It follows that, for any $g \in \Sigma_{XX}\mathcal{G}_{Y|X}$, we have
$$
cov(f(X),g(X))=cov(f(X),E(g(X)|\mathcal{M}_{Y|X}))=cov(E(f(X)|\mathcal{M}_{Y|X}),g(X))=0.
$$
That is, $f \in (\Sigma_{XX}\mathcal{G}_{Y|X})^{\perp}$. The proof is completed.
\end{proof}

\subsection{Proof of Theorem~3}
{\theorem\label{theorem3} Suppose assumptions (3)--(7) hold. In addition, assume that $Ed^2(Y,Y') < \infty$, then as $n\rightarrow \infty$
\begin{eqnarray*}
&&\|\hat{V}_{XX'}-V_{XX'}\|_{HS}=o_p(1),\ \ \
|\langle\hat{\psi}_1,\psi_1\rangle_{\text{HS}}|\stackrel{P}{\longrightarrow} 1,\\
&&\|\{\hat{f}_1(X)-E\hat{f}_1(X)\}-\{f_1(X)-Ef_1(X)\}\|{\longrightarrow} 0,
\end{eqnarray*}
where $\|\cdot\|$ in this theorem is the standard $L_2$ norm to measure the distance of functions and $\| \cdot \|_{\text{HS}}$ denotes the Hilbert-Schmidt norm.}

To prove this theorem, we need the following lemmas.
{\lemma\label{lemma M5} The cross-covariance operator $\Lambda_{XX'}$ is a Hilbert-Schmidt operator, and its Hilbert-Schmidt norm is given by
\begin{eqnarray*}
\|\Lambda_{XX'}\|^2_{\text{HS}}&=&\langle E\{(\kappa_{X}(,X)-\mu_X)\otimes(\kappa_{X}(,X')-\mu_X)d(Y,Y')\},\\
&&E\{(\kappa_{X}(,X)-\mu_X)\otimes(\kappa_{X}(,X')-\mu_X)d(Y,Y')\}\rangle\\
&=&E_{XX'YY'}E_{X''X'''Y''Y'''}[\langle(\kappa_{X}(,X)-\mu_X),(\kappa_{X}(,X'')-\mu_X)\rangle_{\mathcal{H}_X}\\
&&\langle(\kappa_{X}(,X')-\mu_X),(\kappa_{X}(,X''')-\mu_X)\rangle_{\mathcal{H}_X}d(Y,Y')d(Y'',Y''')]\\
&=&\|E_{XX'YY'}[(\kappa_{X}(,X)-\mu_X)(\kappa_{X}(,X')-\mu_X)d(Y,Y')]\|_{\mathcal{H}_X\otimes \mathcal{H}_X}^2
\end{eqnarray*}
where $(X,Y)$,$(X',Y')$,$(X'',Y'')$ and $(X''',Y''')$ are independently and identically with distribution $P_{XY}$.}

From the facts $\mathcal{H}_X \subset L_2(P_X)$, the law of large numbers implies for each $f \in \mathcal{H}_X$,
$$
\lim_{n\rightarrow \infty}\langle f,\widehat{\Lambda}_{XX'}f\rangle_{\mathcal{H}_X}=\langle f,\Lambda_{XX'}f\rangle_{\mathcal{H}_X}
$$
in probability. Moreover, the central limit theorem shows that the above convergence rate is of order $O_p(n^{-1/2})$. The following lemma shows the tight uniform result that $\|\widehat{\Lambda}_{XX'}-\Lambda_{XX'}\|_{HS}$ converges to zero in the order of $O_p(n^{-1/2})$.

{\lemma \label{lem1212} Under the Assumption 3 and $Ed^2(Y,Y') < \infty$, we have
\begin{eqnarray*}
\|\widehat{\Lambda}_{XX'}-\Lambda_{XX'}\|_{HS}=O_p(n^{-1/2})
\end{eqnarray*}
}
\begin{proof}
Write for simplicity $\eta=\kappa_{X}(\cdot,X)-\mu_{X}$ and $\mathcal{F}=\mathcal{H}_X\otimes \mathcal{H}_X$. Then $\eta_1,\ldots, \eta_n$ are i.i.d. random elements in $\mathcal{H}_X$. Lemma \ref{lemma M5} implies
\begin{eqnarray*}
\|\widehat{\Lambda}_{XX'}\|^2_{HS}&=&\|\frac{1}{n(n-1)}\sum_{i\neq j}^n[(\eta_i-\frac{1}{n}\sum_{s=1}^n\eta_s)(\eta_j-\frac{1}{n}\sum_{s=1}^n\eta_s)d(Y_i,Y_j)]\|_{\mathcal{F}}^2.
\end{eqnarray*}
Then we can derive that
\begin{eqnarray*}
&&\langle \Lambda_{XX'},\widehat{\Lambda}_{XX'}\rangle_{HS}\\
&=&\langle E[\eta\eta'd(Y,Y')],\frac{1}{n(n-1)}\sum_{i\neq j}^n[(\eta_i-\frac{1}{n}\sum_{s=1}^n\eta_s)(\eta_j-\frac{1}{n}\sum_{s=1}^n\eta_s)d(Y_i,Y_j)]\rangle_{\mathcal{F}}
\end{eqnarray*}
From these equations, we have
\begin{eqnarray*}
&&\|\widehat{\Lambda}_{XX'}-\Lambda_{XX'}\|^2_{HS}\\
&=&\|\Lambda_{XX'}\|^2_{HS}-2\langle \Lambda_{XX'},\widehat{\Lambda}_{XX'}\rangle_{HS}+\|\widehat{\Lambda}_{XX'}\|^2_{HS}\\
&=&\|\frac{1}{n(n-1)}\sum_{i\neq j}^n[(\eta_i-\frac{1}{n}\sum_{s=1}^n\eta_s)(\eta_j-\frac{1}{n}\sum_{s=1}^n\eta_s)d(Y_i,Y_j)]-E[\eta\eta'd(Y,Y')]\|^2_{\mathcal{F}}\\
&=&\|\frac{1}{n(n-1)}\sum_{i\neq j}^n(\eta_i\eta_jd(Y_i,Y_j)-E[\eta\eta'd(Y,Y')])\\
&&-[(\frac{1}{n}\sum_{s=1}^n\eta_s)((\frac{1}{n}\sum_{s=1}^n\eta_s)(\frac{1}{n(n-1)}\sum_{i\neq j}d(Y_i,Y_j))-\frac{1}{n(n-1)}\sum_{i\neq j}(\eta_i+\eta_j)d(Y_i,Y_j))]\|^2_{\mathcal{F}}\\
\end{eqnarray*}
which provides an upper bound
\begin{eqnarray*}
&&\|\widehat{\Lambda}_{XX'}-\Lambda_{XX'}\|_{HS}\\
&\leq&\|\frac{1}{n(n-1)}\sum_{i\neq j}^n(\eta_i\eta_jd(Y_i,Y_j)-E[\eta\eta'd(Y,Y')])\|_{\mathcal{F}}\\
&&+\|(\frac{1}{n}\sum_{s=1}^n\eta_s)\|_{\mathcal{H}}\|[((\frac{1}{n}\sum_{s=1}^n\eta_s)(\frac{1}{n(n-1)}\sum_{i\neq j}d(Y_i,Y_j))\\
&&-\frac{1}{n(n-1)}\sum_{i\neq j}(\eta_i+\eta_j)d(Y_i,Y_j))]\|_{\mathcal{H}}\\
&=&\|\frac{1}{n(n-1)}\sum_{i\neq j}^n(\eta_i\eta_jd(Y_i,Y_j)-E[\eta\eta'd(Y,Y')])\|_{\mathcal{F}}\\
&&+\|\frac{1}{n}\sum_{s=1}^n\eta_s\|_{\mathcal{H}}\|\frac{1}{n^2(n-1)}\sum_{k\neq i\neq j}\eta_kd(Y_i,Y_j)\|_{\mathcal{H}}\\
\end{eqnarray*}
By simple calculation, we obtain
\begin{eqnarray}\label{M00}
&&E\|\frac{1}{n(n-1)}\sum_{i\neq j}^n(\eta_i\eta_jd(Y_i,Y_j)-E[\eta\eta'd(Y,Y')])\|^2_{\mathcal{F}}\\ \nonumber
&=&\frac{1}{n^2(n-1)^2}\sum_{i\neq j, k\neq t}E\langle \eta_i\eta_jd(Y_i,Y_j)-E[\eta\eta'd(Y,Y')],\eta_k\eta_td(Y_k,Y_t)-E[\eta\eta'd(Y,Y')]\rangle_{\mathcal{F}}\\ \nonumber
&=&\frac{C_1}{n}E\langle \eta\eta'd(Y,Y')-E[\eta\eta'd(Y,Y')],\eta\eta''d(Y,Y'')-E[\eta\eta'd(Y,Y')]\rangle_{\mathcal{F}}\\ \nonumber
&&+\frac{2}{n(n-1)}E\|\eta\eta'd(Y,Y')-E[\eta\eta'd(Y,Y')]\|^2_{\mathcal{F}}\\ \nonumber
&=&O(n^{-1})
\end{eqnarray}
because $E\|\eta\eta'd(Y,Y')\|^2_{\mathcal{H}}<\infty$ by assumption 3 and $Ed^2(Y,Y')<\infty$.
\begin{eqnarray}\label{M01}
E\|\frac{1}{n}\sum_{s=1}^n\eta_s\|^2_{\mathcal{H}}=\frac{1}{n}E\|\eta_s\|^2_{\mathcal{H}}=O(n^{-1})
\end{eqnarray}
Since $E\eta''d(Y,Y')=E\eta''Ed(Y,Y')=0$. By the law of large numbers, for any $f,g\in \mathcal{H}_X$, we have
\begin{eqnarray}\label{M02}
\lim_{n\rightarrow \infty}\langle g,(\eta_kd(Y_i,Y_j))f\rangle_{\mathcal{H}_X} = \langle g,(\eta''d(Y,Y'))f\rangle_{\mathcal{H}_X}
\end{eqnarray}
in probability. Combining \eqref{M00}, \eqref{M01} and \eqref{M02}, we have
\begin{eqnarray*}
\|\widehat{\Lambda}_{XX'}-\Lambda_{XX'}\|_{HS}=O_p(n^{-1/2})
\end{eqnarray*}
\end{proof}

{\lemma \label{le85} Let $\varepsilon_n$ be a positive number such that $\varepsilon_n \rightarrow 0\ (n\rightarrow \infty)$. Then, for the i.i.d. sample $(X_1,Y_1),\ldots,(X_n,Y_n)$, we have
\begin{eqnarray*}
\|\widehat{V}_{XX'}-(\Sigma_{XX}+\varepsilon_n I)^{-1/2}\Lambda_{XX'}\Lambda_{XX'}^*(\Sigma_{XX}+\varepsilon_n I)^{-1/2}\|
=O_p(\frac{1}{\varepsilon_n^{3/2}n^{1/2}})
\end{eqnarray*}
\begin{proof}
The left hand side term can be decomposed as
\begin{eqnarray}\nonumber \label{le81}
&&\widehat{V}_{XX'}-(\Sigma_{XX}+\varepsilon_n I)^{-1/2}\Lambda_{XX'}\Lambda_{XX'}^*(\Sigma_{XX}+\varepsilon_n I)^{-1/2}\\ \nonumber
&=&[(\widehat{\Sigma}_{XX}+\varepsilon_n I)^{-1/2}-(\Sigma_{XX}+\varepsilon_n I)^{-1/2}]\widehat{\Lambda}_{XX'}\widehat{\Lambda}_{XX'}^*(\widehat{\Sigma}_{XX}+\varepsilon_n I)^{-1/2}\\ \nonumber
&+&(\Sigma_{XX}+\varepsilon_n I)^{-1/2}[\widehat{\Lambda}_{XX'}-\Lambda_{XX'}]\widehat{\Lambda}_{XX'}^*(\widehat{\Sigma}_{XX}+\varepsilon_n I)^{-1/2}\\ \nonumber
&+&(\Sigma_{XX}+\varepsilon_n I)^{-1/2}\Lambda_{XX'}(\widehat{\Lambda}_{XX'}^*-\Lambda_{XX'}^*)(\widehat{\Sigma}_{XX}+\varepsilon_n I)^{-1/2}\\
&+&(\Sigma_{XX}+\varepsilon_n I)^{-1/2}\Lambda_{XX'}\Lambda_{XX'}^*[(\widehat{\Sigma}_{XX}+\varepsilon_n I)^{-1/2}-(\Sigma_{XX}+\varepsilon_n I)^{-1/2}]
\end{eqnarray}
From the equation
$$
A^{-1/2}-B^{-1/2}=A^{-1/2}(B^{3/2}-A^{3/2})B^{-3/2}+(A-B)B^{-3/2}.
$$
The first term in the right hand of the equation can be written
\begin{eqnarray*}
&&[(\widehat{\Sigma}_{XX}+\varepsilon_n I)^{-1/2}-(\Sigma_{XX}+\varepsilon_n I)^{-1/2}]\widehat{\Lambda}_{XX'}\widehat{\Lambda}_{XX'}^*(\widehat{\Sigma}_{XX}+\varepsilon_n I)^{-1/2}\\
&=&\{(\widehat{\Sigma}_{XX}+\varepsilon_n I)^{-1/2}((\Sigma_{XX}+\varepsilon_n I)^{3/2}-(\widehat{\Sigma}_{XX}+\varepsilon_n I)^{3/2})\\
&&+(\widehat{\Sigma}_{XX}-\Sigma_{XX})\}(\widehat{\Sigma}_{XX}+\varepsilon_n I)^{-3/2}\widehat{\Lambda}_{XX'}\widehat{\Lambda}_{XX'}^*(\widehat{\Sigma}_{XX}+\varepsilon_n I)^{-1/2}\\
\end{eqnarray*}
From $(\widehat{\Sigma}_{XX}+\varepsilon_n I)^{-1/2}\leq \varepsilon_n^{-1/2}$, $\|(\widehat{\Sigma}_{XX}+\varepsilon_n I)^{-1/2}\widehat{\Lambda}_{XX'}\widehat{\Lambda}_{XX'}^*(\widehat{\Sigma}_{XX}+\varepsilon_n I)^{-1/2}\|\leq C$ and Lemma 8 in \cite{Fukumizu2007}. The norm of the above operator is bounded from above by
\begin{eqnarray*}
&&\frac{C}{\varepsilon_n}\{\frac{3}{\sqrt{\varepsilon_n}}\max\{\|\Sigma_{XX}+\varepsilon_n I\|^{3/2},\|\widehat{\Sigma}_{XX}+\varepsilon_n I\|^{3/2}\}+1\}\|\widehat{\Sigma}_{XX}-\Sigma_{XX}\|\\
&=&O_p(\varepsilon_n^{-3/2}n^{-1/2})
\end{eqnarray*}

For the second term, we have
\begin{eqnarray*}
(\Sigma_{XX}+\varepsilon_n I)^{-1/2}[\widehat{\Lambda}_{XX'}-\Lambda_{XX'}]\widehat{\Lambda}_{XX'}^*(\widehat{\Sigma}_{XX}+\varepsilon_n I)^{-1/2}=O_p(\frac{1}{\varepsilon_nn^{1/2}})
\end{eqnarray*}
The third and fourth terms are similar to the second and first terms. Correspondingly, their bounds are $O_p(\frac{1}{\varepsilon_nn^{1/2}})$ and $O_p(\frac{1}{\varepsilon_n^{3/2}n^{1/2}})$, respectively.
\end{proof}
}

{\lemma \label{le84} Assumption $V_{XX'}$ is compact. Then for a sequence $\varepsilon_n\rightarrow 0$,
\begin{eqnarray*}
\|(\Sigma_{XX}+\varepsilon_n I)^{-1/2}\Lambda_{XX'}\Lambda_{XX'}^*(\Sigma_{XX}+\varepsilon_n I)^{-1/2}-V_{XX'}\|=o_p(1)
\end{eqnarray*}
}
\begin{proof}
An upper bound of the left hand side of the assertion is given by
\begin{eqnarray}\label{5}
&&\|\{(\Sigma_{XX}+\varepsilon_nI)^{-1/2}-\Sigma_{XX}^{-1/2}\}\Lambda_{XX'}\Lambda_{XX'}^*(\Sigma_{XX}+\varepsilon_nI)^{-1/2}\|\\
&+&\|\Sigma_{XX}^{-1/2}\Lambda_{XX'}\Lambda_{XX'}^*\{(\Sigma_{XX}+\varepsilon_nI)^{-1/2}-\Sigma_{XX}^{-1/2}\}\|\nonumber
\end{eqnarray}
The first term of (\ref{5}) is bounded by
\begin{eqnarray}\label{4}
\|\{(\Sigma_{XX}+\varepsilon_nI)^{-1/2}\Sigma_{XX}^{1/2}-I\}\Sigma_{XX}^{-1/2}\Lambda_{XX'}\Lambda_{XX'}^*\Sigma_{XX}^{-1/2}\|.
\end{eqnarray}
Note that the range $\Sigma_{XX}^{-1/2}\Lambda_{XX'}\Lambda_{XX'}^*\Sigma_{XX}^{-1/2}$ is included in $\overline{\mathcal{R}(\Sigma_{XX})}$. Let $v$ be an arbitrary element in $\mathcal{R}(\Sigma_{XX}^{-1/2}\Lambda_{XX'}\Lambda_{XX'}^*\Sigma_{XX}^{-1/2})\bigcap\mathcal{R}(\Sigma_{XX})$. Then there exists $u \in \mathcal{H}_X$ such that $v=\Sigma_{XX}u$. Noting that $\Sigma_{XX}$ and $(\Sigma_{XX}+\varepsilon_nI)^{1/2}$ are commutative, we have
\begin{eqnarray*}
&&\|\{(\Sigma_{XX}+\varepsilon_nI)^{-1/2}\Sigma_{XX}^{1/2}-I\}v\|_{\mathcal{H}_X}\\
&=&\|\{(\Sigma_{XX}+\varepsilon_nI)^{-1/2}\Sigma_{XX}^{1/2}-I\}\Sigma_{XX}u\|_{\mathcal{H}_X}\\
&=&\|(\Sigma_{XX}+\varepsilon_nI)^{-1/2}\Sigma_{XX}^{1/2}\{\Sigma_{XX}^{1/2}-(\Sigma_{XX}+\varepsilon_nI)^{1/2}\}\Sigma_{XX}^{1/2}u\|_{\mathcal{H}_X}\\
&\leq&\|\Sigma_{XX}^{1/2}-(\Sigma_{XX}+\varepsilon_nI)^{1/2}\|\|\Sigma_{XX}^{1/2}u\|_{\mathcal{H}_X}.
\end{eqnarray*}
$\Sigma_{XX}+\varepsilon_nI\rightarrow \Sigma_{XX}$ in norm means that $(\Sigma_{XX}+\varepsilon_nI)^{1/2}\rightarrow \Sigma_{XX}^{1/2}$ in norm, the convergence
\begin{eqnarray*}
\{(\Sigma_{XX}+\varepsilon_nI)^{-1/2}\Sigma_{XX}^{1/2}-I\}v \longrightarrow 0\ \ \ (n\rightarrow\infty)
\end{eqnarray*}
holds for all $v\in \mathcal{R}(\Sigma_{XX}^{-1/2}\Lambda_{XX'}\Lambda_{XX'}^*\Sigma_{XX}^{-1/2})\bigcap\mathcal{R}(\Sigma_{XX})$. Because $\Sigma^{-1}\Lambda_{XX'}$ is compact, Lemma 9 in \cite{Fukumizu2007} shows \eqref{4} converges to zero. The convergence of second term in \eqref{5} can be proved similarly.
\end{proof}

{\lemma\label{86} Let $A$ be a compact positive operator on a Hilbert space ${H}$, and $A_n(n\in {N})$ be bounded positive operators on $\mathcal{H}$ such that $A_n$ converges to A in norm. Assume that the eigenspace of $A$ corresponding to the largest eigenvalue is one-dimensional spanned by a unit eigenvector $\phi$, and the maximum of the spectrum of $A_n$ is attained by a unit eigenvector $f_n$. Then
\begin{eqnarray*}
|\langle f_n,\phi\rangle_\mathcal{H}|\rightarrow 1 \ \ (n\rightarrow \infty).
\end{eqnarray*}
}

\begin{proof}
Because $A$ is compact and positive, the eigen-decomposition
\begin{eqnarray*}
A=\sum_{i=1}^{\infty}\rho_i\psi_i\langle\psi_i,\cdot\rangle_{\mathcal{H}}
\end{eqnarray*}
holds, where $\rho_1>\rho_2\geq\rho_3\geq \cdots \geq 0$ are eigenvalues and $\{\psi_i\}$ is the corresponding eigenvectors so that $\{\psi_i\}$ is the CONS of $\mathcal{H}$.

Let $\delta_n=|\langle f_n,\psi_1\rangle_{\mathcal{H}}|$. We have
\begin{eqnarray*}
\langle f_n,Af_n\rangle_{\mathcal{H}}&=&\rho_1\langle f_n,\psi_1\rangle_{\mathcal{H}}^2+\sum_{i=2}^{\infty}\rho_i\langle f_n,\psi_i\rangle_{\mathcal{H}}^2\\
&\leq&\rho_1\langle f_n,\psi_1\rangle_{\mathcal{H}}^2+\rho_2(1-\langle f_n,\psi_1\rangle_{\mathcal{H}}^2)=\rho_1\delta_n^2+\rho_2(1-\delta_n^2).
\end{eqnarray*}
On the other hand, the convergence
\begin{eqnarray*}
|\langle f_n,Af_n\rangle-\langle\psi_1,A\psi_1\rangle_{\mathcal{H}}|&\leq& |\langle f_n,Af_n\rangle_{\mathcal{H}}-\langle f_n,A_nf_n\rangle_{\mathcal{H}}|+|\langle f_n,A_nf_n\rangle_{\mathcal{H}}-\langle\psi_1,A\psi_1\rangle_{\mathcal{H}}|\\
&\leq&\|A-A_n\|_{\mathcal{H}}+|\|A_n\|_{\mathcal{H}}-\|A\|_{\mathcal{H}}|\rightarrow 0
\end{eqnarray*}
implies that $\langle f_n,Af_n\rangle$ must converges to $\rho_1$. These two facts, together with $\rho_1>\rho_2$, result in $\delta\rightarrow 1$.

From the norm convergence $Q_nA_nQ_n \rightarrow QAQ$, where $Q_n$ and Q are the orthogonal projections onto the orthogonal complements of $f_n$ and $f$, respectively, we have convergence of the eigenvector corresponding to the first eigenvalue. It is not difficult to obtain convergence of the eigenspaces corresponding to the $m$th eigenvalue in a similar way.
\end{proof}

\begin{proof}[of Theorem~3]%\ref{theorem4}
The first and second equations are proved by Lemma \ref{le85} and \ref{le84}. Now we prove the third equation. Without loss of generality, we can assume $\hat{\psi}_1\rightarrow \psi_1$ in $\mathcal{H}_X$. The squared $L_2(P_X)$ distance between $\hat{f}_1-E\hat{f}_1(X)$ and $f_1-Ef_1(X)$ is given by
\begin{eqnarray*}
\|\Sigma_{XX}^{1/2}(\hat{f}_1-f_1)\|^2_{\mathcal{H}_X}=\|\Sigma_{XX}^{1/2}\hat{f}_1\|^2_{\mathcal{H_X}}-2\langle \psi_1,\Sigma_{XX}^{1/2}\hat{f}_1 \rangle_{\mathcal{H}_X}+\|\psi_1\|^2_{\mathcal{H}_X}.
\end{eqnarray*}
Thus, it suffices to show $\Sigma_{XX}^{1/2}\hat{f}_1$ converges to $\psi_1 \in \mathcal{H}_X$ in probability. We have
\begin{eqnarray}\nonumber \label{le82}
\|\Sigma_{XX}^{1/2}\hat{f}_1-\psi_1\|_{\mathcal{H}_X}&\leq& \|\Sigma_{XX}^{1/2}\{(\widehat{\Sigma}_{XX}+\varepsilon_nI)^{-1/2}-(\Sigma_{XX}+\varepsilon_nI)^{-1/2}\}\widehat{\psi}_1\|_{\mathcal{H}_X}\\ \nonumber
&+&\|\Sigma_{XX}^{1/2}(\Sigma_{XX}+\varepsilon_nI)^{-1/2}(\widehat{\psi}_1-\psi_1)\|_{\mathcal{H}_X}\\
&+&\|\Sigma_{XX}^{1/2}(\Sigma_{XX}+\varepsilon_nI)^{-1/2}\psi_1-\psi_1\|_{\mathcal{H}_X}
\end{eqnarray}
Using the same argument as in the bound of the first term of \eqref{le81}, the first term in \eqref{le82} is shown to converge to zero. The second term obviously  converges to zero. Similar to Lemma \ref{le84}, the third term converge to zero, which completes the proof.
\end{proof}

\subsection{Proof of Proposition~4}
\begin{proposition}\label{coordinate representation.} Let $K_n$ be the $n\times n$ kernel matrix whose $(i, j)$th element is $\kappa_X(X_i, X_j)$. Denote $J_n$ as the $n\times n$ matrix whose elements are all one. Define $G_X= (I_n-J_n/n)K_n
(I_n-J_n/n)$ and let $D_Y$ be the $n\times n$  matrix whose $(i, j)$th element is $d(Y_i,Y_j)$. Then we have $ G_X\alpha_\ell= \gamma_\ell$, where $\gamma_\ell$ is the $\ell$th eigenvector of the following matrix
\[(G_X+\varepsilon_n I_n)^{-1}G_XD_YG_XD_YG_X(G_X+\varepsilon_n I_n)^{-1}.\]
\end{proposition}
\begin{proof}%[of Coordinate Mapping]
The subspace $\overline{\rm{ran}}(\widehat{\Lambda}_{XX'})$ is spanned by the set
\begin{eqnarray*}
{\mathcal{C}_X}=\{\kappa_X(\cdot,X_i)-E_n\kappa_X(\cdot,X): i=1,\ldots,n\}=\{\eta_1,\ldots,\eta_n\}.
\end{eqnarray*}

Define $[\cdot]_{\mathcal{C}_X}$ as the coordinate representation about the system ${\mathcal{C}_X}$. Note that the member of this spanning system are not linearly independent because their summation is the zero function. We in the next find the coordinate representation of $-\widehat{\Lambda}_{XX'}$.
\begin{eqnarray*}
[-\widehat{\Lambda}_{XX'}\eta_i]_{\mathcal{C}_X}&=&(n(n-1))^{-1}[(\sum_{k\neq t}^n \eta_k\otimes \eta_td(Y_k,Y_t))\eta_i]_{\mathcal{C}_X}\\
&=&(n(n-1))^{-1}(\sum_{k\neq t}^n [\eta_k]_{\mathcal{C}_X}[\eta_t]_{\mathcal{C}_X}^Td(Y_k,Y_t)G_{X})[\eta_i]_{\mathcal{C}_X}
\end{eqnarray*}
Because $\eta_i$ is simply the $i$th member of the spanning system ${\mathcal{C}_X}$, we have $[\eta_i]_{\mathcal{C}_X}=e_i$. Moreover,
\begin{eqnarray*}
\langle \eta_i,\eta_j\rangle_{\mathcal{H}_X}=\kappa_X(X_i,X_j)
-n^{-1}\sum_{l=1}^n\kappa_X(X_i,X_l)-n^{-1}\sum_{k=1}^n\kappa_X(X_j,X_k)+n^{-2}\sum_{k=1}^n\sum_{l=1}^n\kappa_X(X_k,X_l)
\end{eqnarray*}
 Therefore, the Gram matrix of the set $\mathcal{C}_X$ is $G_{X}=(I_n-J_n/n)K_n
 (I_n-J_n/n)$. Then
\begin{eqnarray*}
[-\widehat{\Lambda}_{XX'}\eta_i]_{\mathcal{C}_X}=(n(n-1))^{-1}(\sum_{k\neq t}e_ke_t^Td(Y_k,Y_t))G_{X}e_i=D_{Y}G_{X}e_i
\end{eqnarray*}
\begin{eqnarray*}
[-\widehat{\Lambda}_{XX'}]_{\mathcal{C}_X}=([-\widehat{\Lambda}_{XX'}\eta_1]_{\mathcal{C}_X},\ldots,[-\widehat{\Lambda}_{XX'}\eta_n]_{\mathcal{C}_X})
=D_{X}G_{X}(e_1,\ldots,e_n)=D_{Y}G_X
\end{eqnarray*}
Similarly, we can get $[\widehat{\Sigma}_{XX}]_{\mathcal{C}_X}=G_X$. The proof is completed.
\end{proof}

\end{document}